\documentclass[leqno]{article}
\usepackage[top=100pt,bottom=100pt,left=100pt,right=100pt]{geometry}
\usepackage{amsmath,amssymb,amsfonts,mathrsfs,exscale,graphics,amsthm,colonequals,comment,url,breakurl,todonotes,enumitem,mathscinet}
\usepackage[curve,matrix,arrow,cmtip]{xy}
\NoComputerModernTips

\usepackage{hyperref}

\numberwithin{equation}{section}

\theoremstyle{plain}
\newtheorem{theorem}[equation]{Theorem}

\newtheorem{lemma}[equation]{Lemma}
\newtheorem{corollary}[equation]{Corollary}
\newtheorem{proposition}[equation]{Proposition}

\theoremstyle{definition}
\newtheorem{definition}[equation]{Definition}

\newtheorem{construction}[equation]{Construction}
\newtheorem{remark}[equation]{Remark}
\newtheorem{example}[equation]{Example}

\setlist[enumerate,1]{label={(\roman*)}}

\DeclareMathOperator{\forg}{forg}
\def\Vec{\mathop{\tt Vec}\nolimits}

\def\Norm{\mathop{\tt Norm}\nolimits}

 \DeclareMathOperator{\End}{End}
 \DeclareMathOperator{\Spec}{Spec}

 \DeclareMathOperator{\Hom}{Hom}

 \DeclareMathOperator{\Aut}{Aut}

\DeclareMathOperator{\G}{G}

\let\into\hookrightarrow

\newcommand{\defeq}{\colonequals}

\newcommand{\Gm}[1][\empty]{
  \ifthenelse{\equal{#1}{\empty}}
    {\mathbb{G}_m}
    {\mathbb{G}_{m,#1}}}

 \newcommand{\Gred}[1][\empty]{
  \ifthenelse{\equal{#1}{\empty}}
    {G^{\text{red}}}
    {G^{\text{red},#1}}}

 \newcommand{\Rep}[1][\empty]{
  \ifthenelse{\equal{#1}{\empty}}
    {\mathop{\text{\tt Rep}}\nolimits}
    {\mathop{\text{$#1$-{\tt Rep}}}\nolimits}}

\DeclareMathOperator{\gr}{gr}

\newcommand\toover[1]{\mathrel{\smash{\overset{#1}{\to}}}}
\newcommand\varto[1]{\mathrel{\hbox to #1pt{\rightarrowfill}}}
\newcommand\vartoover[2]{\mathrel{\smash{\overset{#2}{\varto{#1}}}}}

\renewcommand{\implies}{\Rightarrow}

\let\longto\longrightarrow
\let\onto\twoheadrightarrow
\def\isoto{\stackrel{\sim}{\longto}}
 
\newcommand{\BG}{{\mathbb{G}}}

\newcommand{\BQ}{{\mathbb{Q}}}
\newcommand{\BR}{{\mathbb{R}}}

\newcommand{\BV}{{\mathbb{V}}}

\newcommand{\BZ}{{\mathbb{Z}}}

\newcommand{\CG}{{\mathcal G}}
\newcommand{\CH}{{\mathcal H}}

\newcommand{\CP}{{\mathcal P}}

\newcommand{\CS}{{\mathcal S}}
\newcommand{\CT}{{\mathcal T}}

\newcommand{\CX}{{\mathcal X}}

\def\UEnd{\mathop{\underline{\rm End}}\nolimits}

\def\UAut{\mathop{\underline{\rm Aut}}\nolimits}

\def\ULie{\mathop{\underline{\rm Lie}}\nolimits}

\def\ULie{\mathop{\underline{\rm Lie}}\nolimits}

\newcommand{\leftexp}[2]{{\vphantom{#2}}^{#1}{#2}}

\let\phi\varphi

\DeclareMathOperator{\GL}{GL}

\DeclareMathOperator{\id}{id}

\DeclareMathOperator{\Lie}{Lie}

\def\Mod{\mathop{\tt Mod}\nolimits}

\DeclareMathOperator{\colim}{colim}

\DeclareMathOperator{\Gal}{Gal}

\DeclareMathOperator{\Cent}{Cent}

\newcommand{\FunEx}[1][\empty]{
\ifthenelse{\equal{#1}{\empty}}
{\operatorname{Fun}_{\text{ex}}}
{\operatorname{Fun}_{\text{ex},#1}}}

\newcommand{\FunLex}[1][\empty]{
\ifthenelse{\equal{#1}{\empty}}
{\operatorname{Fun}_{\text{lex}}}
{\operatorname{Fun}_{\text{lex},#1}}}

\newcommand{\FunRex}[1][\empty]{
\ifthenelse{\equal{#1}{\empty}}
{\operatorname{Fun}_{\text{rex}}}
{\operatorname{Fun}_{\text{rex},#1}}}

\newcommand{\FunColim}[1][\empty]{
\ifthenelse{\equal{#1}{\empty}}
{\operatorname{Fun}_{\text{colim}}}
{\operatorname{Fun}_{\text{colim},#1}}}

\newcommand{\FunLA}[1][\empty]{
\ifthenelse{\equal{#1}{\empty}}
{\operatorname{Fun}_{\text{la}}}
{\operatorname{Fun}_{\text{la},#1}}}

\begin{document}

\title{Normed Fiber Functors}

\author{Paul Ziegler\footnote{TU Darmstadt,
 {\tt ziegler@mathematik.tu-darmstadt.de}
}}


\maketitle
\abstract{
  By Goldman-Iwahori, the Bruhat-Tits building of the general linear group $\GL_n$ over a local field $\ell$ can be described as the set of non-archimedean norms on $\ell^n$. Via a Tannakian formalism, we generalize this picture to a description of the Bruhat-Tits building of an unramified reductive group $G$ over $\ell$ as the set of norms on the standard fiber functor of a special parahoric integral model of $G$. We also give a moduli-theoretic description of the parahoric group scheme associated to a point of the building as the group scheme of tensor automorphisms of the lattice chains defined by the corresponding norm.
  }

\section{Introduction}

In the theory of reductive groups over non-archimedean fields, Bruhat-Tits buildings are the analogues of symmetric spaces in the theory of real Lie groups. Given a reductive group $G$ over a Henselian discretely valued non-archimedean field $\ell$, the (extended) Bruhat-Tits building $I(G)$ is a polysimplicial complex with a $G(\ell)$-action obtained by gluing affine spaces $\BR^n$, the so-called apartments of $I(G)$, along hyperplanes determined by the root system of $G$. This $G(\ell)$-set $I(G)$ encodes a lot of information about the group structure of $G(\ell)$ and its open compact subgroups. It was constructed by Bruhat and Tits at a large level of generality in \cite{BT1} and \cite{BT2}. 

A second important aspect of this theory are the so-called Bruhat-Tits group schemes. These are natural smooth integral models of $G$ over the ring of integers $\ell^\circ$ of $\ell$ whose set of $\ell^\circ$-points is equal to the stabilizer in $G(\ell)$ of a point of $I(G)$. They play an important role for example in the theory of Shimura varieties or in the representation theory of reductive groups over local fields.

Already before the work of Bruhat and Tits, the building $I(\GL_n)$ was defined by Goldman-Iwahori in \cite{GoldmanIwahori} as the set of non-archimedean norms $\alpha\colon \ell^n \to \BR^{\geq 0}$. For such a norm $\alpha \in I(\GL_n)$ and an element $\gamma \in \BR^{>0}$, let $V^{\alpha\leq\gamma}\defeq \{v \in V\mid \alpha(v) \leq \gamma\}$. Then the Bruhat-Tits group scheme associated to $\alpha$ has a moduli-theoretic description as the scheme of automorphisms of the lattice chain $(V^{\alpha\leq \gamma})_{\gamma \in \BR^{>0}}$ in $V$.

In this article we generalize this picture of $I(\GL_n)$ as a set of norms from the general linear group to unramified reductive groups via a Tannakian formalism. This extends previous work of Wilson \cite{Wilson} and Cornut \cite{CornutFiltrations}, by which this article was heavily inspired. For various classical groups, a description of $I(G)$ of this kind was already given by Bruhat and Tits in \cite{BTClassiques1} and \cite{BTClassiques2}.

\paragraph{Normed Fiber Functors}

The basic idea is to describe a point in $I(G)$ by specifying a non-archimedean norm on each finite-dimensional representation of $G$ in a way which is functorial and compatible with tensor products. Since the category $\Norm^\circ(\ell)$ of finite-dimensional normed vector spaces over $\ell$ is only $\ell^\circ$-linear and not $\ell$-linear, to achieve this functoriality, we need to work with a category of representations which is only $\ell^\circ$-linear as well. Because of this, we work with the category $\Rep^\circ \CG$ of dualizable representations of a suitable smooth affine integral model $\CG$ of $G$ over $\ell^\circ$. In the language of buildings, this choice of an integral model corresponds to the choice of a base point in $I(G)$. We denote by $\Vec_\ell$ the category of $\ell$-vector spaces and consider the standard fiber functor $\omega\colon \Rep^\circ \CG \to \Vec_\ell$.

\begin{definition}[{c.f. Definition \ref{NFFDef}}]
  A norm on $\omega$ consists of non-archimedean norms $\alpha_X$ on $\omega(X)$ for each $X \in \Rep^\circ \CG$ which are functorial in $X$ and compatible with tensor products.

\end{definition}

The above description of the Bruhat-Tits group schemes for $G=\GL_n$ naturally generalizes to the following construction (c.f. Definition \ref{SchStabDef}): To a norm $\alpha=(\alpha_X)_{X \in \Rep^\circ \CG}$, we associate the \emph{stabilizer group scheme} $\UAut^\otimes(\alpha)$ over $\ell^\circ$ whose $R$-valued sections for some $\ell^\circ$-algebra $R$ are given by $R$-linear automorphisms of $\omega(X)^{\alpha_X\leq\gamma} \otimes_{\ell^\circ} R$ for varying $X \in \Rep^\circ X$ and $\gamma \in \BR^{>0}$ which are functorial in $X$ and $\gamma$ and compatible with tensor products. This moduli functor is representable by an affine group scheme over $\ell^\circ$ with generic fiber $G$.

For a vector space $V$ over $\ell$ equipped with a norm $\alpha$, a \emph{splitting basis} of $(V,\alpha)$ is a basis $(e_i)$ of $V$ such that $\alpha(\sum_i \lambda_i e_i)=\max_i |\lambda_i|\alpha(e_i)$. Such a basis always exists over a discretely valued field $\ell$. We generalize this to a notion of \emph{splittability} for norms on a fiber functor which plays a key technical role in our analysis of normed fiber functors (c.f. Definition \ref{SplittingDef}). Then our main result on normed fiber functors is the following:
\begin{theorem}[c.f. Theorem \ref{MainTheorem}] \label{MainThmIntro}
  Let $\CG$ be a smooth affine group scheme over $\ell^\circ$ and $\alpha=(\alpha_X)_{X \in \Rep^\circ \CG}$ a norm on a fiber functor $\omega\colon \Rep^\circ \CG \to \Vec_\ell$.
  \begin{enumerate}
  \item The group scheme $\UAut^\otimes(\alpha)$ is smooth over $\ell^\circ$.
  \item The norm $\alpha$ is splittable.
  \end{enumerate}
\end{theorem}

\paragraph{The Tannakian formalism for $I(G)$}

Using Theorem \ref{MainThmIntro}, we can then give the desired Tannakian description of $I(G)$ for a unramified reductive group $G$ over $\ell$: We take $\CG$ to be a reductive integral model of $G$. Then we consider the ``tensorial building'' $N^\otimes(\CG)$ given by the set of norms on the standard fiber functor $\omega_\CG \colon \Rep^\circ \CG \to \Vec_\ell$. The group $G(\ell)=\UAut^\otimes(\omega_\CG)(\ell)$ acts naturally on $N^\otimes(\CG)$ via its action on $\omega(X)$ for each $X \in \Rep^\circ \CG$. We prove:

\begin{theorem}[c.f. Theorem \ref{BTComp}] \label{BCompThmIntro}
  \begin{enumerate}
  \item There is a natural $G(\ell)$-equivariant bijection $$c\colon N^\otimes(\CG) \isoto I(G).$$
  \item For each norm $\alpha \in N^\otimes(\CG)$, the group scheme $\UAut^\otimes(\alpha)$ is the parahoric group scheme associated to the point $c(\alpha) \in I(G)$.
  \end{enumerate}
\end{theorem}
For a more precise description of the bijection $c$, see Theorem \ref{BTComp}. The inverse of the map from Theorem \ref{BCompThmIntro} was already constructed in \cite[Theorem 132]{CornutFiltrations}.

As an application of this result, we give in Subsection \ref{FuncSS} a new, shorter, proof of a result of Landvogt from \cite{LandvogtFunctoriality} on the functoriality of Bruhat-Tits buildings, and in \ref{SpecFibSS} a description of the special fibers of parahoric group schemes.

The proofs of Theorems \ref{MainThmIntro} and \ref{BCompThmIntro} proceed by reduction to a normed overfield $m$ of $\ell$ for which the norm map $m \to \BR^{\geq 0}$ is surjective. Over such a field, all norms take a very simple form: Let $\ell^{\circ\circ} \subset \ell^\circ$ be the maximal ideal. Any $\ell^\circ$-lattice $L$ in a vector space $V$ over $\ell$ induces a norm $\alpha_L$ on $V$ determined by $\alpha_L(\lambda v)=|\lambda|$ for all $v \in L \setminus \ell^{\circ\circ} L$. If the norm map on the ground field is surjective, then any norm on a vector space is of this form, and consequently any norm on a fiber functor is determined by an integral model of this fiber functor. Using this fact, over such a large field $m$, the above results can be verified, and to prove them over $\ell$, we investigate how the group scheme $\Aut^\otimes(\alpha)$ changes when $\alpha$ is replaced by its base change to $m$.

\paragraph{Structure of the paper}
In Sections \ref{NVSSection} (resp. \ref{NFFSection}) we develop in parallel our results on normed vector spaces (resp. on normed fiber functors):

In Subsections \ref{StabSS1} and \ref{StabSS2} we define the group schemes $\UAut^\otimes(\alpha)$ and prove their representability.

Next we treat the notion of splittability. Given a norm on a fiber functor, we cannot find splitting bases of the norms on all the vector spaces $\omega(X)$ which are functorial and compatible with tensor products. Hence the notion of a splitting basis of a normed vector space $(V,\alpha)$ cannot be directly generalized to norms on fiber functors. Because of this, in Subsection \ref{SplittingSS1} we introduce two variants of this notion: The first (c.f Definition \ref{SplitDef0}) is given by pair $(L,\chi)$ consisting of an $\ell^\circ$-lattice $L \subset V$ together with a homomorphism $D^{\BR^{>0}}_{\ell^\circ} \to \GL(V)$ from the diagonalizable group scheme $D^{\BR^{>0}}$ with character group $\BR^{>0}$, and the second is given by a homomorphism $D^{\BR^{>0}}_{\ell^\circ} \to \UAut^\otimes(\alpha)$ over $\ell^\circ$. These are then generalized to normed fiber functors in Subsection \ref{SplittingSS2}.

In Subsections \ref{BCSS1} and \ref{BCSS2}, we consider the behaviour of our norms under base change: If $\alpha$ is a norm on a vector space or a fiber functor over $\ell$ and $m$ is some valued overfield of $\ell$, there is a natural base change $\alpha_m$ of $\alpha$ to $m$ which comes with a base change homomorphism
\begin{equation*}
  b^{m/\ell}\colon \UAut^\otimes(\alpha)_{m^\circ} \to \UAut^\otimes(\alpha_m)
\end{equation*}
over $m^\circ$ (c.f. Constructions \ref{UAutBCC0} and \ref{UAutBCC}). The generic fiber of $b^{m/\ell}$ is an isomorphism. The special fiber of $b^{m/\ell}$ is more complicated, and understanding its behaviour is crucial for the proof of our main results. If $|m^*|=\BR^{>0}$, then the image of $b^{m/\ell}_{\tilde m}$ is equal to the centralizer of a canonical homomorphism
\begin{equation*}
  \chi^{m/\ell} \colon D^{\BR^{>0}/|\ell^*|}_{\tilde m} \to \UAut^\otimes(\alpha)_{\tilde m}
\end{equation*}
(c.f. Constructions \ref{ChiCons} and \ref{ChiRamCons}).

To understand the kernel of $b^{m/\ell}_{\tilde m}$, in Subsections \ref{USS1} and \ref{USS2} we introduce and study a natural filtration of $\UAut^\otimes(\alpha)_{\tilde \ell}$ by certain unipotent subgroups.

Then in Subsection \ref{MainThmSS}, we combine all of the above to obtain our main result on normed fiber functors, Theorem \ref{MainTheorem}. In addition to the facts already stated in Theorem \ref{MainThmIntro}, this contains a number of facts about the structure of the special fiber $\UAut^\otimes(\alpha)_{\tilde \ell}$.

Finally, in Section \ref{BTSec}, we use our results on normed fiber functors to give in Theorem \ref{BTComp} our Tannakian description of Bruhat-Tits buildings.

\paragraph{Acknowledgements} The author thanks Christophe Cornut and the anonymous referee for pointing out a number of issues in a previous version of this article. The author is also grateful to Christophe Cornut, Martin Gallauer, Daniel Sch\"appi, Eva Viehmann and Torsten Wedhorn for helpful conversations and suggestions on the subject of this article. The author acknowledges support by the ERC in form of Eva Viehmann's Consolidator Grant 770936: Newton Strata and by the Deutsche Forschungsgemeinschaft (DFG) through the Collaborative Research Centre TRR 326 "Geometry and Arithmetic of Uniformized Structures", project number 444845124.

\section{Normed Vector Spaces} \label{NVSSection}

\subsection{Normed Vector Spaces}
We fix throughout a totally ordered abelian group $(\Gamma,\leq)$ which we write multiplicatively. By a \emph{non-archimedean field} we mean a field $\ell$ equipped with a non-archimedean norm $|\cdot|\colon \ell \to \Gamma \cup \{0\}$. For such a field, we denote by $\ell^{\leq \gamma}$ (resp. $\ell^{<\gamma}$) the subset of elements of norm $\leq \gamma$ (resp. of norm $<\gamma$). We also write $\ell^\circ$ for the subring $\ell^{\leq 1}$ of $\ell$ and $\ell^{\circ\circ}$ for the ideal $\ell^{<1}$ of $\ell^\circ$. We denote the quotient field $\ell^\circ/\ell^{\circ\circ}$ by $\tilde\ell$. By a \emph{non-archimedean extension} $m$ of $\ell$ we mean an overfield $m$ of $\ell$ equipped with a norm $|\cdot | \colon m \to \Gamma \cup \{0\}$ extending the norm on $\ell$. When we talk about $\ell$ being discretely valued we always exclude the case of the trivial valuation.

We fix a non-archimedean field $\ell$.

\begin{definition}
  \begin{enumerate}
  \item   A \emph{norm} on a vector space $V$ over $\ell$ is a function $\alpha\colon V \to \Gamma \cup \{0\}$ satisfying the following conditions:
  \begin{enumerate}
  \item $\alpha(v)=0$ if and only if $v=0$
  \item $\alpha(\lambda v)=|\lambda|\alpha(v)$ for all $\lambda \in \ell$ and $v \in V$.
  \item $\alpha(v+w)\leq \max(\alpha(v),\alpha(w))$ for all $v,w \in V$.
  \end{enumerate}

  \item A \emph{normed vector space} over $\ell$ is a vector space $V$ over $\ell$ together with a norm on $V$.

    \item A contractive homomorphism $h\colon (V,\alpha) \to (W,\beta)$ of normed vector spaces over $\ell$ is a $\ell$-linear map satisfying $\beta(h(v))\leq \alpha(v)$ for all $v \in V$.

    \item For a normed vector space $(V,\alpha)$ over $\ell$ and some $\gamma \in \Gamma$, we denote by $V^{\alpha \leq \gamma} \subset V$ (resp. $V^{\alpha < \gamma} \subset V$) the $\ell^\circ$-submodule of elements of norm $\leq \gamma$ (resp. $<\gamma$).
    \item For any $\ell^\circ$-algebra $R$ and any $\gamma \in \Gamma$, we denote the $R$-module $V^{\alpha\leq\gamma}\otimes_{\ell^\circ} R$ by $V^{\alpha\leq \gamma}_R$.
    \item Given a norm $\alpha$ on $V$ and an element $g \in \GL(V)$, we let $g\cdot \alpha$ be the norm $v \mapsto \alpha(g^{-1}v)$ on $V$.
\end{enumerate}
\end{definition}

\begin{definition}
  A contractive homomorphism $h\colon (V,\alpha) \to (W,\beta)$ of normed vector spaces over $\ell$ is \emph{strict} if the image norm on $h(V)$ induced from $V$ coincides with the subspace norm induced from $W$, that is if for all $w \in h(V)$ the minimum
  \begin{equation*}
    \min\{\alpha(v) \mid v\in h^{-1}(w) \}
  \end{equation*}
exists and is equal to $\beta(w)$.
\end{definition}

\begin{definition}
  Let $(V,\alpha)$ be a normed vector space over $\ell$.
  \begin{enumerate}
  \item A basis $(v_i)_{i\in I}$ of $V$ \emph{splits} the norm $\alpha$ on $V$ if 
    \begin{equation*}
      \alpha(\sum_{i\in I}\lambda_i v_i)=\max_{i\in I}|\lambda_i|\alpha(v_i)
    \end{equation*}
    for all tuples $(\lambda_i)_{i\in I} \in \ell^{\oplus I}$.

  \item The norm $\alpha$ is \emph{splittable} if there exists a basis of $V$ which splits $\alpha$.
  \end{enumerate}
\end{definition}

\begin{definition} \label{EllStarBasisDef}
  \begin{enumerate}
  \item For a normed vector space $(V,\alpha)$, we let $\gr(V,\alpha)$ be the $\Gamma$-graded $\tilde\ell$-vector space $\oplus_{\gamma \in \Gamma}V^{\alpha\leq\gamma}/V^{\alpha<\gamma}$.
  \item We endow $\gr(V,\alpha)$ with the $\ell^*$-action induced by the maps
    \begin{equation*}
      V^{\alpha\leq\gamma} \to V^{\alpha\leq |\lambda|\gamma}, \; v \mapsto \lambda v
    \end{equation*}
    for $\lambda \in \ell^*$ and $\gamma \in \Gamma$.
    \item A family $(f_i)_{i\in I}$ of homogenous elements $f_i$ of $\gr(V,\alpha)$ of degree $\gamma_i$ is an $\ell^*$-basis of $\gr(V,\alpha)$ if for every $\gamma \in \Gamma$ the following holds: Let $(f_j)_{j \in J\subset I}$ be those elements of the family for which $\gamma_i \in \gamma |\ell^*|$. Then for some (or equivalently any) elements $(\lambda_i)_{i\in J}$ of $\ell^*$ satisfying $\gamma=|\lambda_i|\gamma_i $ the family $(\lambda_i \cdot f_i)_{i \in J}$ forms a basis of $V^{\alpha\leq\gamma}/V^{\alpha<\gamma}$.
  \end{enumerate}

\end{definition}

  Such $\ell^*$-bases always exist, we can for example construct them as follows:
  \begin{example} \label{EllStarBasisEx}
    Let $(V,\alpha)$ be a normed vector space and $S \subset \Gamma$ be a set of representatives for $\alpha(V\setminus \{0\})/|\ell^*| \subset \Gamma/|\ell^*|$. Then any homogenous basis of $\oplus_{\gamma \in S} V^{\alpha\leq\gamma}/V^{\alpha<\gamma}$ is an $\ell^*$-basis of $\gr(V,\alpha)$.
  \end{example}

  \begin{lemma} \label{EllStarCard}
    \begin{enumerate}
    \item An $\ell^*$-basis of $\gr(V,\alpha)$ always exists.
    \item Any two $\ell^*$-bases of $\gr(V,\alpha)$ have the same cardinality.
    \end{enumerate}
  \end{lemma}
  \begin{proof}
    (i) follows from Example \ref{EllStarBasisEx}.

    For (ii), it follows from Definition \ref{EllStarBasisDef}, that for any $\ell^*$-basis $(f_i)_{i \in I}$ of $\gr(V,\alpha)$ and any family $(\lambda_i)_{i\in I} \in (\ell^*)^I$, the family $(\lambda_i \cdot f_i)_{i\in I}$ is again an $\ell^*$-basis. One can choose these $\lambda_i$ in such a way, that the family $(\lambda_i\cdot f_i)_{i \in I}$ is of the form constructed in Example \ref{EllStarBasisEx}. So (ii) follows from the fact that all homogenous bases of $\oplus_{\gamma \in S}V^{\alpha\leq\gamma}/V^{\alpha<\gamma}$ have the same cardinality.
  \end{proof}
  
\begin{lemma} \label{EllInd}
  Let $(V,\alpha)$ be a normed vector space over $\ell$ and $(e_i)_{i\in I} \in V^I$ a family of vectors in $V$ with images $\bar e_i$ in $V^{\alpha\leq \alpha(e_i)}/V^{\alpha < \alpha(e_i)}$. The following are equivalent:
  \begin{enumerate}
  \item For each family $(\lambda_i)_{i \in I} \in \ell^{\oplus I}$ the identity
    \begin{equation*}
      \alpha(\sum_i \lambda_i e_i)=\max_i |\lambda_i| \alpha(e_i)
    \end{equation*}
    holds.
  \item For all $(\lambda_i)_i \in (\ell^*)^{\oplus I}$ the family $(\lambda_i \cdot \bar e_i)_i$ is linearly independent in $\gr(V,\alpha)$.
  \end{enumerate}
  A family $(e_i)_{i\in I}$ satisfying these conditions is linearly independent in $V$.
\end{lemma}
\begin{proof}
  For a family $(\lambda_i)_{i \in I} \in \ell^{\oplus I}$ such that all non-zero $\lambda_i e_i$ have the same norm $\alpha(\lambda_ie_i)=\gamma$, the linear combination $\sum_i \lambda_ie_i$ has norm $\gamma$ if and only if $\sum_{i \in I: \lambda_i\not= 0} \lambda_i \cdot \bar e_i$ is non-zero in $V^{\alpha\leq\gamma}/V^{\alpha<\gamma}$. This implies the equivalence of (i) and (ii). For the last claim, one notes that any non-trivial way of writing $0\in V$ as a linear combination of the $e_i$ would violate (i).
\end{proof}

This implies:
\begin{lemma} \label{EllBase}
  Let $(V,\alpha)$ be a normed vector space over $\ell$. For a basis $(e_i)_{i\in I}$ of $V$ the following are equivalent:
  \begin{enumerate}
  \item The family $(e_i)_{i\in I}$ is a splitting basis of $(V,\alpha)$. 
  \item The images $\bar e_i$ of the $e_i$ in $V^{\alpha\leq \alpha(e_i)}/V^{\alpha < \alpha(e_i)}$ form a $\ell^*$-basis of $\gr(V,\alpha)$.
  \end{enumerate}
\end{lemma}

\begin{definition}
  \begin{enumerate}
  \item   We denote by $\Norm^\circ(\ell)$ the category whose objects are finite-dimensional $\ell$-vector spaces equipped with a splittable norm, and whose morphisms are contractive homomorphisms of such vector spaces.
  \item We denote by $\forg\colon \Norm^\circ(\ell) \to \Vec_\ell$ the forgetful functor to the category of vector spaces over $\ell$.
  \item We call a sequence $ 0 \to U \to V \to W \to 0$ in $\Norm^\circ(\ell)$ short exact if it consists of strict morphisms and if the underlying sequence of vector spaces is exact.
  \end{enumerate}
\end{definition}

  The category of normed vector spaces, as well as the category $\Norm^\circ(\ell)$ admit finite direct sums, given by the direct sum of the underlying vector spaces equipped with the maximum norm.

  \begin{proposition} \label{NormSplit}
    Let $(V,\alpha) \in Norm^\circ(\ell)$.
    \begin{enumerate}
    \item For every subspace $U \subset V$, the restriction $\alpha|_U$ is splittable.
    \item For every quotient $V \onto W$ of $V$ and every $w \in V$ the minimum
  \begin{equation*}
    \beta(w)=\min\{\alpha(v) \mid v\in h^{-1}(w) \}
  \end{equation*}
exists, and this defines a splittable norm $\beta$ on $W$.
\item  Every short exact sequence in $\Norm^\circ(\ell)$ splits.
    \end{enumerate}
  \end{proposition}
  \begin{proof}
    Let $0 \to U \to V \to W \to 0$ be a short exact sequence of vector spaces. Let $\{\gamma_1,\hdots,\gamma_s\} \subset \Gamma/|\ell^*|$ be a set of representatives of $\alpha(V \setminus \{0\})/|\ell^*|$. We write again $\alpha$ for $\alpha|_{U}$ and for each $1\leq i \leq s$ we let $W^{\beta\leq\gamma_i}$ (resp. $W^{\beta < \gamma_i}$) be the image of $W^{\alpha\leq\gamma_i}$ (resp. $W^{\alpha<\gamma_i}$) in $W$. This gives us an exact sequence
    \begin{equation*}
      0 \to U^{\alpha\leq\gamma_i}/U^{\alpha<\gamma_i} \to V^{\alpha\leq \gamma_i}/V^{\alpha<\gamma_i} \to W^{\beta\leq \gamma_i}/W^{\beta < \gamma_i} \to 0.
    \end{equation*}
    We choose elements $f_{i,1},\hdots, f_{i,r_i}$ of $U^{\alpha\leq\gamma_i}$ which map to a basis of $U^{\alpha\leq\gamma_i}/U^{\alpha<\gamma_i}$ and elements $g_{i,1},\hdots,g_{i,t_i}$ of $V^{\alpha\leq\gamma_i}$ which map to a basis of $W^{\beta\leq\gamma_i}/W^{\beta<\gamma_i}$. Then $$(f_{i,1},\hdots,f_{i,r_i},g_{i,1},\hdots,g_{i,t_i})$$ maps to a basis of $V^{\alpha\leq\gamma_i}/V^{\alpha<\gamma_i}$ and hence the family of all the $f_{i,j}$ and $g_{i,j}$ satisfies the conditions of Lemma \ref{EllInd}. In particular this family is linearly independent in $V$.

    Next we claim that the images of the $g_{i,j}$ in $W$ are linearly independent. If not, then some linear combination $x$ of these elements lies in $U$. But then by the choice of the $f_{i,j}$, there exists some linear combination $y$ of the $f_{i,j}$ for which $\alpha(x+y)<\alpha(x)$. This violates Lemma \ref{EllInd} (i). 

    If $r$ is the number of elements $f_{i,j}$ and $t$ the number of elements $g_{i,j}$, then we find $r \leq \dim(U)$, $t \leq \dim(W)$ and hence
    \begin{equation*}
      r+s \leq \dim(U)+\dim(W)=\dim(V)
    \end{equation*}

    On the other hand, since the $f_{i,j}$ and $g_{i,j}$ form an $\ell^*$-basis of $\gr(V,\alpha)$ and since $(V,\alpha)$ is splittable, Lemma \ref{EllBase} and Lemma \ref{EllStarCard} (ii) imply $\dim(V)=r+s$. So all these inequalities are equalities.

    Hence the $f_{i,j}$ form a basis of $U$, which is a splitting basis of $(U,\alpha|_{U})$ by Lemma \ref{EllBase}. This shows (i).

    Similarly, the $f_{i,j}$ and $g_{i,j}$ together form a splitting basis of $(V,\alpha)$ and the images $\pi(g_{i,j})$ of the $g_{i,j}$ in $W$ form a basis of $W$. Hence if we let $\beta$ be the norm on $W$ given by
    \begin{equation*}
      \beta\left(\sum_{i,j} \lambda_{i,j} \pi(g_{i,j})\right)=\max_{i,j} |\lambda_{i,j}| \gamma_i,
    \end{equation*}
    we see that this norm $\beta$ satisfies the formula from (ii). This proves (ii).

    Finally, for (iii), we note that the assignment $\pi(g_{i,j}) \mapsto g_{i,j}$ defines a section in $\Norm^\circ(\ell)$ of $\pi\colon (V,\alpha) \to (W,\beta)$.
  \end{proof}

  This shows in particular that equipping $\Norm^\circ(\ell)$ with the class of short exact sequences makes it into an exact category in the sense of Quillen.

  Proposition \ref{NormSplit} also implies:
  \begin{lemma} \label{NormExact}
    Let $(V_1,\alpha_1) \to (V_2,\alpha_2) \to (V_3,\alpha_3)$ be an exact sequence in $\Norm^\circ(\ell)$. Then for each $\gamma \in \Gamma$, the sequence
    \begin{equation*}
      0 \to V_1^{\alpha_1\leq\gamma} \to V_2^{\alpha_2\leq\gamma} \to V_3^{\alpha_3\leq \gamma} \to 0 
    \end{equation*}
    of $\ell^\circ$-modules is exact.
  \end{lemma}

  \begin{proposition} \label{ExtSplittable}
    Consider a sequence $(U,\alpha) \to (V,\alpha) \to (W,\alpha)$ of strict contractive morphisms of normed vector spaces whose underlying sequence of vector spaces is short exact and such that $(U,\alpha) \in \Norm^\circ(\ell)$ and $(W,\alpha) \in \Norm^\circ(\ell)$. Then $(V,\alpha)$ is also in $\Norm^\circ(\ell)$.
  \end{proposition}
  \begin{proof}
    The assumptions imply that for each $\gamma \in \Gamma$ the sequence
    \begin{equation*}
      0 \to U^{\alpha\leq\gamma}/U^{\alpha<\gamma} \to V^{\alpha\leq \gamma}/V^{\alpha<\gamma} \to W^{\alpha\leq\gamma}/W^{\alpha<\gamma} \to 0
    \end{equation*}
    is exact. Hence for each $\gamma$ we may choose elements $f_1,\hdots,f_r$ of $U^{\alpha\leq\gamma}$ mapping to a basis of $U^{\alpha\leq\gamma}/U^{\alpha<\gamma}$ and elements $g_1,\hdots,g_t$ of $V^{\alpha\leq\gamma}$ mapping to a basis of $W^{\alpha\leq\gamma}/W^{\alpha<\gamma}$. If we do this for every $\gamma$ in a set of representatives of $(\alpha(U\setminus \{0\}) \cup \alpha(W \setminus \{0\}))/|\ell^*|$ it follows using Lemma \ref{EllBase} that we obtain a splitting basis of $(V,\alpha)$. 
  \end{proof}
\begin{lemma} \label{TensorNorm}
  Consider objects $(V,\alpha), (W,\beta) \in \Norm^\circ(\ell)$.
  \begin{enumerate}
  \item   The formula
  \begin{equation*}
    (\alpha\otimes \beta)(z)=\min\left\{ \max_j \alpha(v_j)\beta(w_j) \mid v_j \in V, w_j \in W, z=\sum_j v_j \otimes w_j \right\}
  \end{equation*}
  defines a norm on $V \otimes W$.
  \item If $(e_i)_{1\leq i\leq n}$ (resp. $(f_j)_{1\leq j\leq m}$) is a splitting basis of $(V,\alpha)$ (resp. $(W,\beta)$), then $(e_i \otimes f_j)_{i,j}$ is a splitting basis of $(V\otimes W, \alpha \otimes \beta)$.
  \end{enumerate}
\end{lemma}
\begin{proof}
  Let $(e_i)$ and $(f_i)$ be as in (ii). Given an element $\sum_j v_j \otimes w_j$ of $V \otimes W$, one checks that the minimum in (i) exists by expressing the $v_j$ in terms of the $e_i$ and the $w_j$ in terms of the $f_i$. Then one readily verifies (i) and (ii).
\end{proof}
This makes $\Norm^\circ(\ell)$ into a symmetric monoidal category. Similarly one checks:
\begin{lemma}
  Let $(V,\alpha) \in \Norm^\circ(\ell)$.
  \begin{enumerate}
  \item The formula
    \begin{equation*}
      \alpha^*(\phi) \defeq \max\left\{ |\phi(v)|/\alpha(v) \mid v \in V \setminus \{0\} \right\}
    \end{equation*}
    defines a norm $\alpha^*$ on the dual space $V^*$ of $V$.
  \item The dual basis of any splitting basis for $\alpha$ is a splitting basis for $\alpha^*$.
  \item The object $(V^*,\alpha^*) \in \Norm^\circ(\ell)$ is dual to $(V,\alpha)$ in the symmetric monoidal category $\Norm^\circ(\ell)$.
  \end{enumerate}
\end{lemma}

So altogether $\Norm^\circ(\ell)$ is an $\ell^\circ$-linear exact rigid symmetric monoidal category. For objects $(V,\alpha), (W,\beta) \in \Norm^\circ(\ell)$, we denote by
\begin{equation*}
  \Hom(\alpha,\beta)\defeq (V^* \otimes W)^{\alpha^* \otimes \beta\leq 1}
\end{equation*}
the module of contractive homomorphisms $(V,\alpha) \to (W,\beta)$. We also write $\End(\alpha)$ for the $\ell^\circ$-algebra $\Hom(\alpha,\alpha)$.

\begin{definition}
  For a finitely generated $\ell$-vector space $V$, by a $\ell^\circ$-lattice $L \subset V$ we mean a finitely generated (and hence free) $\ell^\circ$-submodule $L$ for which the natural homomorphism $L \otimes_{\ell^\circ} \ell \to V$ is an isomorphism.
\end{definition}
\begin{construction}\label{LatticeNorm}
  Let $V$ be a finite dimensional vector space over $\ell$ and $L \subset V$ an $\ell^\circ$-lattice. Then we obtain a norm $\alpha_L$ on $V$ by
  \begin{equation*}
    \alpha_L(\lambda v)\defeq |\lambda| 
  \end{equation*}
for $v \in L \setminus \ell^{\circ\circ} L$ and $\lambda \in \ell$. This norm satisfies $V^{\alpha_L \leq 1}=L$ and takes values in $|\ell|$. 
\end{construction}

Conversely: 
\begin{lemma} \label{SpecialNormClass}
  If $(V,\alpha) \in \Norm^\circ(\ell)$ satisfies $\alpha(V) \subset |\ell|$, then $\alpha=\alpha_L$ for the lattice $L=V^{\alpha \leq 1}$.
\end{lemma}
\begin{proof}
  For $v \in V\setminus \{0\}$, let $\lambda \in \ell^*$ be an element with $\alpha(v)=|\lambda|$. Then $\alpha_L(v)=\alpha_L(\lambda (\lambda^{-1}v))=|\lambda|=\alpha(v)$ since $\alpha(\lambda^{-1}v)=1$ and hence $\lambda^{-1}v \in L \setminus \ell^{\circ\circ}L$.
\end{proof}

\begin{definition}
  Let $V$ be a finite-dimensional vector space over $\ell$. By a \emph{filtration} on $V$ we mean a splittable norm on $V$ (with values in $\Gamma$) with respect to the trivial norm $\ell \to \Gamma \cup \{0\}$ which sends every non-zero element of $\ell$ to $1 \in \Gamma$.

  We denote the set of (usual) splittable norms on $V$ by $N(V)$ and the set of filtrations by $F(V)$.
\end{definition}

The following is shown in \cite[6.1.3]{CornutFiltrations}:
\begin{proposition} \label{PullMapExists}
  There exists a unique map
  \begin{equation*}
    +\colon N(V) \times F(V) \to N(V)
  \end{equation*}
  described by the following: For any $F \in F(V)$ and $\alpha \in N(V)$, there exists a basis $(e_1,\hdots,e_n)$ of $V$ splitting both $F$ and $\alpha$, and any such basis again splits $\alpha+F$ with values $(\alpha+F)(e_i)=\alpha(e_i)F(e_i)$.
\end{proposition}
\subsection{Stabilizer Groups} \label{StabSS1}
We fix a non-archimedean field $\ell$ and an object $(V,\alpha) \in \Norm^\circ(\ell)$.

\begin{definition} \label{SchStabDef0}

 For an $\ell^\circ$-algebra $R$, we let $\UEnd(\alpha)(R)$ be the set of tuples $(\phi^{\gamma})_{\gamma \in \Gamma}$ satisfying the following conditions:
  \begin{enumerate}
  \item For each $\gamma \in \Gamma$, the object $\phi^{\gamma}$ is an $R$-linear endomorphism of the $R$-module $V^{\alpha\leq \gamma}_R$.
  \item For all $\gamma \leq \delta$ in $\Gamma$, the diagram
    \begin{equation*}
      \xymatrix{
        V^{\alpha\leq \gamma}_R \ar[r] \ar[d]_{\phi^{\gamma}} & V^{\alpha\leq \delta}_R \ar[d]^{\phi^{\delta}} \\
        V^{\alpha\leq \gamma}_R \ar[r] & V^{\alpha \leq \delta}_R,
}
    \end{equation*}
in which the horizontal arrows are the base change to $R$ of the inclusion $$V^{\alpha \leq \gamma} \into V^{\alpha\leq \delta},$$ commutes.
\item For all $\lambda \in \ell^*$ and all $\gamma \in \Gamma$, the diagram
    \begin{equation*}
      \xymatrix{
        V^{\alpha \leq \gamma}_R \ar[r] \ar[d]_{\phi^{\gamma}} & V^{\alpha \leq \gamma|\lambda|}_R \ar[d]^{\phi^{\gamma |\lambda|}} \\
        V^{\alpha \leq \gamma}_R \ar[r] & V^{\alpha \leq \gamma |\lambda|}_R,
}
    \end{equation*}
    in which the horizontal arrows are the base change to $R$ of the isomorphism $$V^{\alpha \leq \gamma} \isoto V^{\alpha \leq \gamma |\lambda|}, \; v \mapsto \lambda v,$$ commutes.
    \end{enumerate}
Under termwise addition and composition, the set $\UEnd(\alpha)(R)$ is naturally a $R$-algebra.

 For a homomorphism $R \to S$ of $\ell^\circ$-algebras, there is a natural base change homomorphism $\UEnd(\alpha)(R) \to \UEnd(\alpha)(S)$. So we obtain a presheaf of rings $\UEnd(\alpha)$ on the category of affine schemes over $\ell^\circ$.

 Finally we denote by $\UAut(\alpha) \subset \UEnd(\alpha)$ the group presheaf given by the units of $\UEnd(\alpha)$.
 \end{definition}

 \begin{proposition} \label{UAutRepr0}
   \begin{enumerate}
   \item The presheaves $\UEnd(\alpha)$ and $\UAut(\alpha)$ are sheaves for the fpqc-topology.
   \item Let $R$ be an $\ell^\circ$-algebra for which the $R$-modules $V^{\alpha\leq \gamma}_R$ are locally free of finite type for all $\gamma \in \Gamma$. Then $\UEnd(\alpha)_R$ and $\UAut(\alpha)_R$ are representable by an affine group scheme of finite type over $R$.
   \item More precisely, in the situation of (ii), there exist finitely many elements $\gamma_1,\hdots,\gamma_n \in \Gamma$ for which the natural homomorphisms

   \begin{equation*}
     \UEnd(\alpha)_R \to \prod_{1 \leq i \leq n}\End_R(V^{\alpha \leq \gamma_i}_R)
   \end{equation*}
   and
   \begin{equation*}
     \UAut(\alpha)_R \to \prod_{1\leq i \leq n}\GL(V^{\alpha\leq \alpha(e_i)}_R)
   \end{equation*}
   are representable by closed immersions.
   \end{enumerate}
 \end{proposition}
 \begin{proof}
   Claim (i) follows from descent theory.

   For (iii), let $(e_i)_{1\leq i\leq n}$ be a splitting basis of $(V,\alpha)$. Then the norm $\alpha(v)$ of any element $v \in V$ lies in $\alpha(e_i)|\ell^*|$ for some $1\leq i \leq n$. Using this one checks that the natural homomorphism
   \begin{equation*}
     \UEnd(\alpha)_R \to \prod_{1 \leq i \leq n}\End_R(V^{\alpha \leq \alpha(e_i)}_R)
   \end{equation*}
   is representable by a closed immersion.
   Then
   \begin{equation*}
     \UAut(\alpha)_R \to \prod_{1\leq i \leq n}\GL(V^{\alpha\leq \alpha(e_i)}_R)
   \end{equation*}
   is representable by a closed immersion as well.

   Then (iii) implies (ii).
 \end{proof}

 \begin{lemma} \label{UAutReprConc}
   Proposition \ref{UAutRepr0} (ii) in particular applies in the following cases:
   \begin{enumerate}
   \item For $R=\ell^\circ$ in case $\ell^\circ$ is a discrete valuation ring and $\Gamma$ has rank one.
   \item Always for $R=\ell$ or $R=\tilde\ell$.
   \end{enumerate}
 \end{lemma}
 \begin{proof}
   Let $(e_i)$ be a splitting basis of $(V,\alpha)$. For $\gamma \in \Gamma$, such a basis induces an isomorphism
   \begin{equation*}
     V^{\alpha\leq\gamma}_R \cong \oplus_i \ell^{\leq \gamma/\alpha(e_i)}\otimes_{\ell^\circ} R.
   \end{equation*}
   So the $R$-modules $V^{\alpha\leq\gamma}_R$ are free of finite rank if for all $\gamma \in \Gamma$, the $R$-module $\ell^{\leq\gamma} \otimes_{\ell^\circ} R$ is free of finite rank.

   In case (i), each $\ell^{\leq\gamma}$ is free of rank one and so the above applies. In case (ii), for $R=\tilde\ell$, the above applies since $\ell^{\leq\gamma}$ is either free of rank one over $\ell^\circ$ or satisfies $\ell^{\circ\circ}\ell^{\leq\gamma}=\ell^{\leq\gamma}$, and for $R=\ell$ since $\ell^{\leq\gamma} \otimes_{\ell^\circ} \ell \subset \ell$ is either $0$ or $\ell$.
 \end{proof}

 \begin{lemma} \label{GenUAutDesc}
   Assume that for each element $\gamma \in \Gamma$ there exists an element $x \in \ell^*$ satisfying $|x| \leq \gamma$.
   \begin{enumerate}
   \item For each $\gamma \in \Gamma$ the inclusion $V^{\alpha\leq\gamma} \into V$ induces an isomorphism $V^{\alpha\leq\gamma}_\ell \cong V$.
   \item For varying $\gamma$, the isomorphisms from (i) induce an isomorphism $\UAut(\alpha)_\ell \cong \GL(V)$.
   \end{enumerate}
 \end{lemma}
 \begin{proof}
   A splitting basis $(e_i)_{1\leq i \leq n}$ of $(V,\alpha)$ gives identifications
   \begin{equation*}
     V^{\alpha\leq\gamma} \cong \oplus_i \ell^{\leq \gamma/\alpha(e_i)}
   \end{equation*}
   for all $\gamma \in \Gamma$. The assumption ensures that all the $\ell^\circ$-submodules $\ell^{\leq \gamma/\alpha(e_i)}$ of $\ell$ appearing here are not zero. This implies (i). Then (ii) follows from the definition of $\UAut(\alpha)$.
 \end{proof}
 We also give a different description of $\UEnd(\alpha)$:
 \begin{construction} \label{UAutCompCons0}
   We consider the ring $\End(\alpha)= (V\otimes V^*)^{\alpha\leq 1}$ of contractive endomorphisms of $V$ with respect to $\alpha$. For an element $\gamma \in \Gamma$, the evaluation map $V \otimes V^* \otimes V \to V, \; v \otimes h \otimes w \mapsto h(w)v$ restricts to a morphism
   \begin{equation*}
     \End(\alpha) \otimes V^{\alpha\leq \gamma} = (V \otimes V^*)^{\alpha \leq 1} \otimes V^{\alpha \leq \gamma} \to V^{\alpha \leq \gamma}.
   \end{equation*}
   For an $\ell^\circ$-algebra $R$, via the base change of this morphism to $R$, any element of $\End(\alpha)_R$ acts naturally on $V^{\alpha\leq \gamma}_R$. For varying $\gamma$, these actions satisfy the conditions from Definition \ref{SchStabDef0} and hence give a ring homomorphism
   \begin{equation} \label{EndCompMor}
     g_{\alpha,R} \colon \End(\alpha)_R \to \UEnd(\alpha)(R).
   \end{equation}
\end{construction}

 \begin{proposition} \label{UAutComp0}
   For any $\ell^\circ$-algebra $R$, the homomorphism $g_{\alpha,R}$ is an isomorphism.
 \end{proposition}
 \begin{proof}
   We fix a splitting basis $(e_i)$ of $(V,\alpha)$. 

   For an $\ell$-linear endomorphism $\phi$ of $V$ we consider the coefficients $a_{ij} \in \ell$ given by $h(e_i)=\sum_{j} a_{ij} e_j$ for all $i$. Then $\phi$ is in $\End(\alpha)$ if and only if $|a_{ij}| \leq \alpha(e_i)/\alpha(e_j)$ for all $1\leq i,j \leq n$. This gives an identification
   \begin{equation*}
       \End(\alpha) \cong \oplus_{i,j} \ell^{\leq \alpha(e_i)/\alpha(e_j)}
     \end{equation*}
     which in turn induces an isomorphism
\begin{equation} \label{EndIdent}
  \End(\alpha)_R \cong \oplus_{i,j} \ell^{\leq \alpha(e_i)/\alpha(e_j)}\otimes R.
\end{equation}

We construct an inverse of $g_{\alpha,R}$ as follows: Let $\phi=(\phi^{\gamma})_{\gamma} \in \UEnd(\alpha)(R)$ for some $\ell^\circ$-algebra $R$.  For any $\gamma \in \Gamma$, the basis $(e_1,\hdots, e_n)$ gives an identification
   \begin{equation*}
     V^{\alpha \leq \gamma}_R \cong \oplus_{i}\ell^{\leq \gamma/\alpha(e_i)} \otimes_{\ell^\circ} R.
   \end{equation*}

   We fix an index $i$ and take $\gamma = \alpha(e_i)$. We consider the element $e_i \otimes 1 \in V^{\alpha\leq \gamma}_R$ and use the above identification to write its image $\phi^{\gamma}(e_i \otimes 1)$ under $\phi$ as a tuple $(a_{ij} \otimes r_j)_j$ with elements $a_{ij} \in \ell^{\leq \alpha(e_i)/\alpha(e_j)}$ and $r_j \in R$. Under the identification \eqref{EndIdent}, for varying $i$ these tuples give an element $h \in \End(\alpha)_R$. This defines a map $g'_{\alpha,R}\colon \UEnd(\alpha)(R) \to \End(\alpha)_R$. One readily checks that the composition $g'_{\alpha,R} \circ g_{\alpha,R}$ is the identity. So $g'_{\alpha,R}$ is surjective.

   To finish the proof, we note that these tuples $(a_{ij} \otimes r_j)_j$ for all $i$ completely determine $\phi$: Indeed let $\gamma \in \Gamma$ be arbitrary and $\lambda \in \ell$ with $\lambda e_i \in V^{\alpha \leq \gamma}$. Then $\lambda e_i \otimes 1 \in V^{\alpha\leq \gamma}_R$ is the image of $e_i \otimes 1 \in V^{\alpha \leq \alpha(e_i)}$ under the composition
   \begin{equation*}
     V^{\alpha\leq \alpha(e_i)}_R \toover{v \mapsto \lambda v} V^{\alpha \leq |\lambda| \alpha(e_i)}_R \to V^{\alpha \leq \gamma}_R.
   \end{equation*}
Hence $g'_{\alpha,R}$ is injective. So $g'_{\alpha,R}$, and hence also $g_{\alpha,R}$, is bijective.
\end{proof}
   In particular $g_{\alpha,R}$ restricts to a group isomorphism
   \begin{equation} \label{AutCompMor}
     \End(\alpha)^*_R \defeq (\End(\alpha)_R)^* \isoto \UAut(\alpha)(R).
   \end{equation}

\subsection{Splittings} \label{SplittingSS1}
We fix a non-archimedean field $\ell$.

\begin{definition}
  For an abelian group $A$, we denote by $D^A$ the diagonalizable group scheme over $\BZ$ with character group $A$. So for any ring $R$ and any $R$-module $M$, a representation of $D^A_R$ on $M$ is the same as a direct sum decomposition $M=\oplus_{w \in A}M^w$ into weight spaces $M^w$. 
\end{definition}

In Section \ref{NFFSection} we will consider tensor functor from certain categories of representations into $\Norm^\circ(\ell)$. To consider splittings of such tensor functors, we need to reformulate the notion of a splitting basis of an object $(V,\alpha) \in \Norm^\circ(\ell)$ in a way which can be made compatible with the formation of tensor products. For this we replace the decomposition $V=\oplus_i \ell e_i$ given by a splitting basis by the $\Gamma$-grading of $V$ making each $e_i$ homogenous of degree $\alpha(e_i)$, which we specify via a homomorphism $\chi\colon D^\Gamma_\ell \to \GL(V)$. But since replacing a vector $e_i$ in a splitting basis by a scalar multiple $\lambda e_i$ changes its norm to $|\lambda| \alpha(e_i)$, such a grading alone is not enough to specify a norm: We also need to add some kind of integral structure to $\chi$. We give two variants of such an integral structure: The first is to specify a lattice $L \subset V$ for which $\chi$ extends to a homomorphism $D^\Gamma_{\ell^\circ} \to \GL(L)$ over $\ell^\circ$. The second is to require $\chi$ to extend to a certain homomorphism $D^{\Gamma}_{\ell^\circ} \to \UAut(\alpha)$ over $\ell^\circ$.

\begin{construction} \label{ChiNormCons0}
  Let $V$ be a finite-dimensional $\ell$-vector space. To an $\ell^\circ$-lattice $L \subset V$ together with homomorphism $$\chi\colon D^\Gamma_{\ell^\circ} \to \GL(L)$$ over $\ell^\circ$, we associate a norm $\alpha_{L,\chi}$ on $V$ as follows: 

The homomorphism $\chi$ induces weight decompositions $L=\oplus_{w \in \Gamma}L^w$ and $V=\oplus_{w\in \Gamma}V^w$. By Construction \ref{LatticeNorm}, each lattice $L^w \subset V^w$ induces a norm $\alpha_{L^w}$ on $V^w$. Now we let
  \begin{equation*}
    \alpha_{L,\chi}(\sum_{w\in \Gamma}v_w)=\max_{w\in \Gamma}(\alpha_{L^w}(v_w)w)
  \end{equation*}
for all finite sums $\sum_{w \in \Gamma}v_w \in V$ with each $v_w$ in $V^w$. This defines a norm on $V$.
\end{construction}

\begin{lemma} \label{SplitEq0}
  \begin{enumerate}
  \item In the situation of Construction \ref{ChiNormCons0}, the norm $\alpha_{L,\chi}$ is splittable.
  \item Every splittable norm $\alpha$ on $V \in \Vec_\ell$ is equal to the norm $\alpha_{L,\chi}$ for suitable $L$ and $\chi$.
  \end{enumerate}
\end{lemma}
\begin{proof}
  (i) By construction, a basis of $L$ formed out of bases of the $L^w$ splits $\alpha_{L,\chi}$.

  (ii) Let $(e_1,\hdots,e_n)$ be a splitting basis of $(V,\alpha)$. We take $L$ to be the $\ell^\circ$-lattice spanned by $(e_1,\hdots,e_n)$ and $\chi\colon D^\Gamma_{\ell^\circ} \to GL(L)$ the homomorphism which acts with weight $\alpha(e_i)$ on each $e_i$. Then $\alpha=\alpha_{L,\chi}$.
\end{proof}

This motivates the following definition:
\begin{definition} \label{SplitDef0}
  Let $\alpha$ be a norm on a finite-dimensional vector space $V$ over $\ell$. A pair $(L,\chi)$ as in Construction \ref{ChiNormCons0} \emph{splits} $\alpha$ if $\alpha=\alpha_{L,\chi}$.
\end{definition}

Now we fix $(V,\alpha) \in \Norm^\circ(\ell)$. We also a give the following notion of a splitting of $(V,\alpha)$ which doesn't involve a lattice:

\begin{definition} \label{IntSplitDef0}
  Let $R$ be a $\ell^\circ$-algebra and consider a homomorphism $\chi\colon D^\Gamma_R \to \UAut(\alpha)_R$ or $D^{\Gamma/|\ell^*|}_R \to \UAut(\alpha)_R$. For each $\gamma \in \Gamma$, the group sheaf $\UAut(\alpha)_R$ acts naturally on $V^{\alpha \leq \gamma}_R/V^{\alpha < \gamma}_R$. We say that $\chi$ \emph{splits} $\alpha$ over $R$ if for all $\gamma\in \Gamma$, all weights of $\chi$ on $V^{\alpha \leq \gamma}_R/ V^{\alpha < \gamma}_R$ are in $\gamma |\ell^*|$.
\end{definition}

\begin{remark}
   Since $V^{\alpha\leq\gamma}/V^{\alpha<\gamma}$ is an $\tilde \ell$-module, in the situation of Definition \ref{IntSplitDef0}, the homomorphism $\chi$ splits $(V,\alpha)$ over $R$ if and only if $\chi_{R \otimes_{\ell^\circ} \tilde \ell}$ splits $(V,\alpha)$ over $R \otimes_{\ell^\circ} \tilde\ell$.

\end{remark}

b
These two notions of a splitting are equivalent in the following sense:
\begin{proposition} \label{SplitExtEquiv0}
  For a homomorphism $\chi\colon D^\Gamma_\ell \to \GL(V)$ the following are equivalent:
  \begin{enumerate}
  \item There exists an $\ell^\circ$-lattice $L \subset V$ together with an extension $\chi_1 \colon D^\Gamma_{\ell^\circ} \to \GL(L)$ of $\chi$ for which the pair $(L,\chi_1)$ splits $\alpha$.
  \item The homomorphism $\chi$ extends to a homomorphism $\chi_2 \colon D^\Gamma_{\ell^\circ} \to \UAut(\alpha)$ which splits $\alpha$ over $\ell^\circ$. 
  \end{enumerate}

  If this is the case, then such objects $\chi_1$, $\chi_2$ and $L$ are uniquely determined by the above.

\end{proposition}
\begin{proof}
  Let $V=\oplus_{w \in \Gamma} V^w$ be the weight decomposition given by $\chi$.
  
  $(i) \implies (ii)$: Let $(L,\chi_1)$ be as in (i) and $L=\oplus_{w \in \Gamma} L^w$ the associated weight decomposition. Then, for $\gamma \in \Gamma$, Construction \ref{ChiNormCons0} implies
  \begin{equation*}
    V^{\alpha\leq \gamma}=\oplus_{w\in \Gamma}(V^w)^{\alpha\leq\gamma}.
  \end{equation*}
  These decompositions are functorial in $\gamma$ and hence give an extension $\chi_2 \colon D^{\Gamma}_{\ell^\circ} \to \UAut(\alpha)$ of $\chi$. For $v \in V^w$, we find $\alpha(v)=\alpha_{L^w}(v) w \in |\ell^*| w$. Hence $\chi_2$ splits $\alpha$ over $\ell^\circ$. The fact that $V^{\alpha\leq\gamma}$ is contained in $V$ implies that such an extension $\chi_2$ is unique if it exists.

  $(ii) \implies (i)$: We reverse the above argument: If $\chi_2$ exists, then the induced weight decomposition of $V^{\alpha \leq \gamma}$ is given by
  \begin{equation*}
    V^{\alpha\leq \gamma}=\oplus_{w\in \Gamma} \left(V^w \cap V^{\alpha \leq \gamma}\right).
  \end{equation*}
This shows that there is  a direct sum decomposition
\begin{equation*}
  (V,\alpha)=\oplus_{w\in \Gamma} (V^w,\alpha|_{V^w})
\end{equation*}
of normed vector spaces. By Proposition \ref{NormExact}, each of the factors in this decomposition is splittable. For each $w\in \Gamma$ we choose a basis $(e_i^w)_i$ of $V^w$ which splits $\alpha|_{V^w}$. The fact that $\chi_2$ splits $\alpha$ over $\ell^\circ$ implies that for each $w$ and $i$ there exists an element $x_i^w \in \ell^*$ for which
\begin{equation}
  \label{eq:bla3430}
  \alpha(e_i^w)=w |x_i^w|.
\end{equation}
We consider the $\ell^\circ$-lattice $L$ in $V$ spanned by the elements $((x_i^w)^{-1} e_i^w)_{i,w}$. Then \eqref{eq:bla3430} says that $(L,\chi)$ splits $\alpha$. Along the same lines one checks the claim about uniqueness.
\end{proof}

For later use we note the following fact:
\begin{lemma} \label{SplitTransl}
  Let $L \subset V$ be a $\ell^\circ$-lattice, let $\chi\colon D^\Gamma_{\ell^\circ} \to \GL(L)$ a homomorphism and let $g \in \GL(V)$. We consider $\leftexp{g}{\chi}$ as a homomorphism $D^{\Gamma}_{\ell^\circ} \to GL(g\cdot L)$. Then $\alpha_{g\cdot L,\leftexp{g}{\chi}}=g\cdot \alpha_{L,\chi}$.
\end{lemma}
\begin{proof}
 This follows directly from Construction \ref{ChiNormCons0}.
\end{proof}

\subsection{Base change of normed vector spaces} \label{BCSS1}
We fix a non-archimedean field $\ell$, an object$(V,\alpha) \in \Norm^\circ(\ell)$ and a non-archimedean extension $m$ of $\ell$. The following is shown in the same way as Lemma \ref{TensorNorm}:
\begin{lemma} \label{NormBC}
  \begin{enumerate}
  \item For $w \in V_m$, the minimum
  \begin{equation*}
    \alpha_m(w)\defeq \min \left\{ \max_{i}|\lambda_i|\alpha(v_i) \bigg\vert w=\sum_i v_i\otimes \lambda_i \text{ with } v_i \in V,\; \lambda_i \in m \right\}
  \end{equation*}
  exists and this defines a norm $\alpha_m$ on $V_m$ which restricts to $\alpha$ on $V$.
\item The image in $V_m$ of any splitting basis of $\alpha$ is a splitting basis of $\alpha_m$.
\item This construction defines an exact $\ell^\circ$-linear tensor functor
  \begin{equation*}
    \Norm^\circ(\ell) \to \Norm^\circ(m)
  \end{equation*}
  \end{enumerate}  
\end{lemma}

For an $m^\circ$-algebra $R$ and $\gamma \in \Gamma$, we will generally denote the $R$-module $(V_m)^{\alpha_m \leq \gamma} \otimes_{m^\circ} R$ by $V^{\alpha_m \leq \gamma}_R$. (By contrast, note that by our convention from above, we write $V^{\alpha\leq\gamma}_R$ for the $R$-module $V^{\alpha\leq\gamma} \otimes_{\ell^\circ} R$.)

\begin{construction} \label{BCPresCons}
For an element $\gamma \in \Gamma$, we construct a chain complex

     \begin{multline} \label{BCPres}
      \left(\oplus_{\sigma, \sigma' \in \Gamma, \sigma' < \sigma} V^{\alpha\leq \gamma/\sigma}\otimes_{\ell^\circ} m^{\leq \sigma'} \right) \oplus \left( \oplus_{\sigma \in \Gamma, \lambda \in \ell^*} V^{\alpha\leq \gamma/\sigma}\otimes_{\ell^\circ}m^{\leq \sigma} \right) \\
      \to \oplus_{\sigma \in \Gamma} V^{\alpha\leq \gamma/\sigma}\otimes_{\ell^\circ}m^{\leq \sigma} \to V_m^{\alpha_{m} \leq \gamma} \to 0
    \end{multline}
as follows: 

The first arrow is given on each summand with $\sigma' < \sigma$ by the inclusion
\begin{align*} 
V^{\alpha\leq \gamma/\sigma}\otimes_{\ell^\circ} m^{\leq \sigma'} &\to \left( V^{\alpha\leq \gamma/\sigma'}\otimes_{\ell^\circ} m^{\leq \sigma'} \right) \oplus \left( V^{\alpha\leq \gamma/\sigma}\otimes_{\ell^\circ} m^{\leq \sigma} \right), \\
 x &\mapsto (x,-x)
\end{align*}
and on each summand with $\sigma \in \Gamma$ and $\lambda \in \ell^*$ by the inclusion
\begin{align*}
  V^{\alpha\leq \gamma/\sigma}\otimes_{\ell^\circ}m^{\leq \sigma} &\to \left( V^{\alpha\leq \gamma/\sigma}\otimes_{\ell^\circ}m^{\leq \sigma} \right) \oplus \left( V^{\alpha\leq \gamma/\sigma|\lambda|}\otimes_{\ell^\circ}m^{\leq \sigma |\lambda|} \right) \\
  v\otimes x &\mapsto (v\otimes x, -( v/\lambda) \otimes (\lambda x)).
\end{align*}
The second arrow is given by addition in $V_m$.
\end{construction}
\begin{lemma} \label{alPres}
 In the situation of Construction \ref{BCPresCons}, the sequence \eqref{BCPres} is exact.
\end{lemma}
\begin{proof}

Using a splitting basis of $(V,\alpha)$, one reduces the claim to the case $\dim V=1$. In this case the claim follows by a direct verification.
\end{proof}

\begin{construction} \label{UAutBCC0}
  For any $m^\circ$-algebra $R$, the base changes of the right exact sequence \eqref{BCPres} to $R$ for varying $\gamma \in \Gamma$ induce a natural homomorphism
\begin{equation*}
  \UAut(\alpha)(R) \to \UAut(\alpha_m)(R).
\end{equation*}
  For varying $R$ these give a base change morphism
  \begin{equation} \label{UAutBC}
    b^{m/\ell}\colon \UAut(\alpha)_{m^\circ} \to \UAut(\alpha_m).
  \end{equation}
\end{construction}
\begin{lemma}\label{UnrBC0}
  If $|m^*|=|\ell^*|$, then $b^{m/\ell}\colon \UAut(\alpha)_{m^\circ} \to \UAut(\alpha_m)$ is an isomorphism.
\end{lemma}
\begin{proof}
  In this case, for all $\gamma \in \Gamma$, the inclusion $V^{\alpha\leq \gamma}_{m^\circ} \into V^{\alpha_m\leq \gamma}_m$ inside $V_m$ is an equality. This implies the claim.
\end{proof}

\begin{lemma} \label{NormExactnessDescent}
  Let $0 \to (X,\alpha_X) \to (Y,\alpha_Y) \to (Z,\alpha_Z) \to 0$ be a sequence in $\Norm^\circ(\ell)$ which becomes short exact after base change to $m$. Then this sequence is short exact in $\Norm^\circ(\ell)$.
\end{lemma}
\begin{proof}
  By descent, the underlying sequence of vector spaces is exact. The fact that $\alpha_Y$ restricts to $\alpha_X$ is automatic from the corresponding statement over $m$. To show that $(Y,\alpha_Y) \to (Z,\alpha_Z)$ is strict, we argue as follows: By Proposition \ref{NormSplit} there exists a sub-vector space $Z'$ of $Y$ such that $(Z',\alpha_Y|_{Z'})$ is splittable and gives a direct sum decomposition $(Y,\alpha_Y)=(X,\alpha_X)\oplus (Z',\alpha_Y|_{Z'})$ in $\Norm^\circ(\ell)$. Then it suffices to show that the induced morphism $(Z',\alpha_Y|_{Z'}) \to (Z,\alpha_Z)$ in $\Norm^\circ(\ell)$ is an isomorphism. This can be checked after base change to $m$, where it follows from the fact that $(Y_m,\alpha_{Y,m}) \to (Z_m,\alpha_{Z,m})$ is strict.
\end{proof}

\begin{lemma} \label{RestrictionSplittable}
  Let $V$ be a finite-dimensional vector space over $\ell$ and $\alpha$ a splittable norm on $V_m$. Then the restriction $\alpha|_V$ is splittable.
\end{lemma}
\begin{proof}
  We proceed by induction on $\dim(V)$. Since any norm on a one-dimensional vector space is splittable, the claim holds in case $\dim(V)=1$. In general, let $e_1 \in V$ be a non-zero vector. By applying Proposition \ref{NormSplit} to the sub-vector space $U$ of $V_m$ spanned by $e_1$, we see that we can extend $e_1$ to a splitting basis $(e_1,\hdots,e_n)$ of $V_m$. Then the inquality
  \begin{equation*}
    \alpha(\sum_{i=1}^n \lambda_i e_i) \geq |\lambda_1|\alpha(e_1),
  \end{equation*}
  for any $\lambda_1,\hdots,\lambda_n \in m$ shows that the projection $\pi\colon V \to \ell e_1$ associated to this basis is contractive with respect to the norm $\alpha|_V$ on $V$ and the splittable norm $\beta$ given by $\beta(\lambda e_1)= |\lambda| \alpha(e_1)$ on $\ell e_1$.

  If we let $U\defeq \ker(\pi)$, then $\alpha|_{U_m}$ is splittable by Proposition \ref{NormSplit} and hence $\alpha|_U$ is splittable by induction. Hence by applying Proposition \ref{ExtSplittable} to the sequence $U \to V \to \ell e_1$ we find that $\alpha|_V$ is splittable.
\end{proof}
From now on we assume that $|m^*|=\Gamma$.

\begin{lemma} \label{BCNorm2}
  For all $\gamma \in \Gamma$ the following inclusion holds:
    \begin{equation*}
     V^{\alpha < \gamma} \otimes_{\ell^\circ} m^\circ \subset m^{\circ\circ} V^{\alpha_m \leq \gamma}_m
    \end{equation*}
    If $|\ell^*|$ is dense in $|m^*|$, then this inclusion is an equality.
\end{lemma}
\begin{proof}
  Using a splitting basis of $(V,\alpha)$, one reduced to the case $\dim(V)=1$. In this case, the claim follows from the fact that $|m^*|=\Gamma$.
\end{proof}

\begin{definition} \label{BCDecDef}
  Let $\mu$ and $\gamma$ be elements of $\Gamma$. We denote the class $\mu |\ell^*| \in \Gamma/|\ell^*|$ by $[\mu]$ and let
  \begin{equation*}
    \gr_{[\mu]}^{\ell/m}(V^{\alpha_m \leq \gamma}_{\tilde m}) \subset V^{\alpha_m \leq \gamma}_{\tilde m}
  \end{equation*}
  be the image of the composition
  \begin{equation*}
    V^{\alpha \leq \mu} \otimes m^{\leq \gamma/\mu} \into V_m^{\alpha_m \leq \gamma} \onto V^{\alpha_m \leq \gamma}_{\tilde m}.
  \end{equation*}
 (Note that this image only depends on $[\mu]$.)
\end{definition}

\begin{proposition} \label{BCGrProps}
  \begin{enumerate}
  \item The submodules $\gr_{[\mu]}^{\ell/m}(V^{\alpha_m \leq \gamma}_{\tilde m})$ give a $\Gamma/|\ell^*|$-grading of $V^{\alpha_m\leq \gamma}_{\tilde m}$. 
  \item With respect to the gradings from (i), for each $x \in m^*$, the isomorphism $V^{\alpha\leq \gamma}_{\tilde m} \isoto V^{\alpha \leq \gamma |x|}_{\tilde m}$ induced by multiplication by $x$ is an isomorphism of graded vector spaces.
  \item Given a second normed vector space $(W,\alpha) \in \Norm^\circ(\ell)$ and an element $\gamma' \in \Gamma$, we consider the natural morphism
    \begin{equation*}
      V^{\alpha\leq\gamma}_{\tilde m} \otimes_{\tilde m} W^{\alpha \leq \gamma'}_{\tilde m} \to (V\otimes W)^{\alpha \leq \gamma\gamma'}_{\tilde m},
    \end{equation*}
    whose source we equip with the tensor product of the gradings from (i). This is a morphism of graded vector spaces.
  \end{enumerate}
\end{proposition}
\begin{proof}
  (i) Using a splitting basis of $V$ one reduces to the case that $V$ is one-dimensional. In this case, let $\mu \in \Gamma$ be an element for which there exist elements $v \in V^{\alpha\leq \mu}$ and $x \in m^{\leq \gamma/\mu}$ for which the image of $x \otimes v$ in $V^{\alpha_m \leq \gamma}_{\tilde m}$ is non-zero. Then Lemma \ref{BCNorm2} implies $\alpha(v)=\mu$. Since the norms of non-zero elements of $V$ form a single class $[\mu_0]\in \Gamma/|\ell^*|$, this shows that $\gr^{m/\ell}_{[\mu_0]}(V^{\alpha_m\leq \gamma}_{\tilde m})=V^{\alpha_m\leq \gamma}_{\tilde m}$ and $\gr^{m/\ell}_{[\mu]}(V^{\alpha_m\leq \gamma}_{\tilde m})=0$ for all $[\mu]\in \Gamma/|\ell^*|$ different from $[\mu_0]$. This proves (i).

  (ii) and (iii) follow directly from the definition of the grading.
\end{proof}

\begin{construction}
  \label{ChiCons}
  For each $\gamma \in \Gamma$, the grading from Proposition \ref{BCGrProps} (i) corresponds to a homomorphism $D^{\Gamma/|\ell^*|}_{\tilde m} \to \GL(V^{\alpha_m\leq \gamma}_{\tilde m})$. Proposition \ref{BCGrProps} (ii) implies that these homomorphisms fit together to a homomorphism
  \begin{equation*}
    \chi^{m/\ell}\colon D^{\Gamma/|\ell^*|}_{\tilde m} \to \UAut(\alpha_m)_{\tilde m}.
  \end{equation*}
\end{construction}

\begin{proposition} \label{DilProp1}
  Let $R$ be a $m^\circ$-algebra and $\phi \in \UAut(\alpha_m)(R)$. We denote the ring $R\otimes_{m^\circ} \tilde m$ by $\tilde R$.
  \begin{enumerate}
  \item If $\phi$ is in the image of the base change homomorphism \eqref{UAutBC}, then the image of $\phi$ in $\UAut(\alpha_m)( \tilde R)$ centralizes the homomorphism
    \begin{equation*}
       \chi^{m/\ell}_{\tilde R} \colon D^{\Gamma/|\ell^*|}_{\tilde R} \to \UAut(\alpha_m)_{\tilde R}.
    \end{equation*}

  \item   Assume that $|\ell^*|$ is dense in $|m^*|$ and that $R$ is flat over $m^\circ$. Then the converse of (i) holds as well.
  \end{enumerate}
\end{proposition}
\begin{proof}
  Note that a section of $\UAut(\alpha_m)(\tilde R)$ centralizes $\chi^{m/\ell}_{\tilde R}$ if and only if its action on all $V^{\alpha_m \leq \gamma}_{\tilde R}$ preserves the grading which defines $\chi^{m/\ell}$. Using this, claim (i) follows from the definition of these gradings.

  For (ii), note that under the given assumptions, by Lemma \ref{BCNorm2} for all elements $\mu \leq \gamma$ there is the following chain of inclusions:
  \begin{equation*}
    m^{\circ\circ}V^{\alpha_m\leq \gamma}_R \subset (V^{\alpha \leq \mu} \otimes_{\ell^\circ} m^{\leq \gamma/\mu})\otimes_{m^\circ} R \subset V^{\alpha_m \leq \gamma}_R
  \end{equation*}
  This shows that an $R$-linear automorphism of $V^{\alpha_m \leq \gamma}_R$ preserves $(V^{\alpha \leq \mu} \otimes_{\ell^\circ} m^{\leq \gamma/\mu})\otimes_{m^\circ} R$ if and only if it preserves
  the submodule $\gr^{m/\ell}_{[\mu]}(V^{\alpha_m \leq \gamma}_{\tilde m})_{\tilde R}$ of $V^{\alpha_m \leq \gamma}_{\tilde R}$. This implies (ii).
\end{proof}

\begin{lemma} \label{WeightBCIsoLemma}
  Let $\chi\colon D^{\Gamma/|\ell^*|}_{\ell^\circ} \to \UAut(\alpha)$ be a homomorphism which splits $\alpha$ over $\ell^\circ$. For $\gamma \in \Gamma$, we write the weight decomposition of $V^{\alpha\leq\gamma}$ induced by $\chi$ as
  \begin{equation*}
    V^{\alpha\leq \gamma}=\oplus_{[\mu]\in\Gamma/|\ell^*|} (V^{\alpha\leq\gamma})^{[\mu]}.
  \end{equation*}
  We also let $\chi'\defeq b^{m/\ell}\circ\chi_{m^\circ} \colon D^{\Gamma/|\ell^*|}_{m^\circ}\to \UAut(\alpha_m)$ and again write the associated weight decompositions as
  \begin{equation*}
    V^{\alpha_m\leq \gamma}_m=\oplus_{[\mu]\in\Gamma/|\ell^*|} (V_m^{\alpha_m \leq\gamma})^{[\mu]}.
  \end{equation*}
  \begin{enumerate}
  \item   For all $\gamma, \mu \in \Gamma$, the inclusion $V^{\alpha\leq \mu} \otimes m^{\leq \gamma/\mu} \into V^{\alpha_m\leq \gamma}_m$ restricts to an isomorphism

  \begin{equation} \label{WeightBCIso}
    (V^{\alpha\leq\mu})^{[\mu]} \otimes m^{\leq \gamma/\mu} \isoto (V_m^{\alpha_m \leq \gamma})^{[\mu]}.
  \end{equation}
\item The inclusions $V^{\alpha\leq\gamma'} \into V^{\alpha\leq \gamma}$ for $\gamma'\leq \gamma$ induce an isomorphism
  \begin{equation} \label{WeightCompIso}
    \colim_{\gamma'\leq \gamma, \gamma'\in [\mu]} (V^{\alpha\leq \gamma'})^{[\mu]} \isoto (V^{\alpha\leq \gamma})^{[\mu]}.
  \end{equation}
    \end{enumerate}
\end{lemma}
\begin{proof}
  (i) The definition of the two weight decompositions appearing here implies that the inclusion $V^{\alpha\leq \mu} \otimes m^{\leq \gamma/\mu} \into V^{\alpha_m\leq \gamma}_m$ restricts to an injection $$(V^{\alpha\leq\mu})^{[\mu]} \otimes m^{\leq \gamma/\mu} \into (V_m^{\alpha_m \leq \gamma})^{[\mu]}.$$

  To show that this inclusion is surjective let $v \in (V_m^{\alpha_m \leq \gamma})^{[\mu]}$. We can write $v$ as $\sum_{i=1}^s v_i \otimes x_i$ for some elements $v_i \in V$ and $x_i \in m^*$ satisfying $|x_i|\alpha(v_i)\leq \gamma$. After decomposing each $v_i$ into its homogenous components and throwing away the components not of degree $[\mu]$, we may assume that each $v_i$ lies in $(V^{\alpha\leq \alpha(v_i)})^{[\mu]}$. Then the fact that $\chi$ splits $\alpha$ implies that $\alpha(v_i)\in [\mu]$ for all $i$. Hence after replacing each $v_i \otimes x_i$ by $xv_i \otimes x_i/x$ for a suitable $x \in \ell^*$, we may assume that $\alpha(v_i)=\mu$ for all $i$. Then the $x_i$ are in $m^{\leq \gamma/\mu}$ and the decomposition $v=\sum_i v_i \otimes x_i$ gives a preimage of $v$ under \eqref{WeightBCIso} as desired.

  (ii) The colimit appearing here is simply a union inside $V$. That this union is equal to $(V^{\alpha\leq \gamma})^{[\mu]}$ follows from the fact that by Definition \ref{IntSplitDef0} any element $v \in (V^{\alpha\leq\gamma})^{[\mu]}$ satisfies $\alpha(v) \in [\mu]$.
\end{proof}
\begin{proposition} \label{SplBCComp}
  \begin{enumerate}
  \item A homomorphism $\chi \colon D^{\Gamma/|\ell^*|}_{\tilde m} \to \UAut(\alpha)_{\tilde m}$ splits $\alpha$ over $\tilde m$ if and only if $b^{m/\ell}_{\tilde m}\circ  \chi=\chi^{m/\ell}$.
  \item If a homomorphism $\chi\colon D^{\Gamma/|\ell^*|}_{\ell^\circ} \to \UAut(\alpha)$ splits $\alpha$ over $\ell^\circ$, then $b^{m/\ell}$ restricts to an isomorphism
    \begin{equation*}
      {\Cent}_{\UAut(\alpha)_{m^\circ}}(\chi_{m^\circ}) \isoto \Cent_{\UAut(\alpha_m)}(b^{\ell/m}\circ\chi_{m^\circ}).
    \end{equation*}
  \end{enumerate}
\end{proposition}
\begin{proof}
  (i) First we claim that for $\mu \in \Gamma$, the inclusion $V^{\alpha\leq \mu}\into V_m^{\alpha_m\leq \mu}$ induces an isomorphism
  \begin{equation} \label{MuIso}
    \left(V^{\alpha\leq\mu}/V^{\alpha<\mu}\right) \otimes_{\tilde\ell} \tilde m \isoto \gr^{m/\ell}_{[\mu]}(V^{\alpha_m \leq \mu}_{\tilde m}).
  \end{equation}
  Using a splitting basis this claim can be reduced to the case $\dim(V)=1$, where it follows from a direct verification. So $\chi_{\tilde m}$ acts on the source of \eqref{MuIso} with weight $[\mu]$ if and only if it does so on the target of \eqref{MuIso}. This statement translates to claim (i).

  (ii) Let $R$ be an $m^\circ$-algebra. We use the notation from Lemma \ref{WeightBCIsoLemma}. Giving a section of the group $\Cent_{\UAut(\alpha)_{m^\circ}}(\chi_{m^\circ})(R)$ amounts to giving $R$-linear automorphisms of the modules $(V^{\alpha\leq\gamma})^{[\mu]}_R$ for all $\gamma \in \Gamma$ and $[\mu]\in \Gamma/|\ell^*|$ which are functorial in $\gamma$ and with respect to the morphisms from Definition \ref{SchStabDef0} (iii). For such $\gamma$ and $[\mu]$, the isomorphism \eqref{WeightCompIso} induces an isomorphism
  \begin{equation*}
    \colim_{\gamma'< \gamma, \gamma'\in [\mu]}(V^{\alpha\leq \gamma'})^{[\mu]}_R \isoto (V^{\alpha\leq \gamma})^{[\mu]}_R. 
  \end{equation*}

  So one sees that giving a section of $\Cent_{\UAut(\alpha)_{m^\circ}}(\chi_{m^\circ})(R)$ is equivalent to giving $R$-linear automorphisms of the modules $(V^{\alpha\leq\mu})^{[\mu]}_R$ for all $\mu \in \Gamma$  which are functorial with respect to the morphisms from Definition \ref{SchStabDef0} (iii).
  
  On the other hand, giving a section of the group $\Cent_{\UAut(\alpha_m)}(b^{m/\ell}\circ\chi)$ is equivalent to giving analogously functorial $R$-linear automorphisms of the modules $(V_m^{\alpha_m\leq\gamma})^{[\mu]} \otimes_{m^\circ} R$ for all $\gamma$ and $[\mu]$. Via the isomorphisms \eqref{WeightBCIso} these two data are equivalent.
  \end{proof}

\subsection{The subgroups $U_\delta(\alpha) \subset \UAut(\alpha)_{\tilde \ell}$} \label{USS1}
We fix a non-archimedean field $\ell$.

\begin{definition} \label{SpecFiberFil}
  Let $(V,\alpha) \in \Norm^\circ(\ell)$, let $\gamma \in \Gamma$, let $\delta \in \Gamma^{\leq 1}$ and let $R$ be a $\tilde \ell$-algebra.
  
  \begin{enumerate}
  \item   We let $$F^{\leq \delta}(V^{\alpha\leq \gamma}_{\tilde \ell}) \subset V^{\alpha\leq\gamma}_{\tilde\ell}$$
be the image of the homomorphism
\begin{equation*}
  V^{\alpha \leq \delta\gamma} \into V^{\alpha\leq \gamma} \onto V^{\alpha\leq\gamma}_{\tilde\ell}.
\end{equation*}
\item We denote the base change $F^{\leq \delta}(V^{\alpha\leq \gamma}_{\tilde \ell})_R$ by $F^{\leq \delta}(V^{\alpha\leq \gamma}_R)$.
\item We also let $$F^{<\delta}(V^{\alpha\leq\gamma}_R)\defeq \cup_{\delta'<\delta} F^{\leq \delta'}(V^{\alpha\leq\gamma}_R) \subset V^{\alpha\leq\gamma}_R$$ and
  \begin{equation*}
    \gr^{\delta}(V^{\alpha\leq\gamma}_R)\defeq F^{\leq \delta}(V^{\alpha\leq \gamma}_R)/F^{< \delta}(V^{\alpha\leq \gamma}_R).
  \end{equation*}
  \end{enumerate}
  We note that the objects just defined are naturally functorial in $(V,\alpha)$.

  \begin{lemma} \label{FiltExact1}
  For $\delta \in \Gamma^{\leq 1}$ and $\gamma \in \Gamma$, the natural functors $\Norm^\circ(\ell) \to \Mod_{\tilde\ell}$ sending $(V,\alpha)$ to $F^{\leq\delta}(V^{\alpha\leq\gamma}_{\tilde\ell})$, to $F^{<\delta}(V^{\alpha\leq\gamma}_{\tilde\ell})$ or to $\gr^\delta(V^{\alpha\leq\gamma}_{\tilde\ell})$ are exact.
\end{lemma}
\begin{proof}
  Since all exact sequences in $\Norm^\circ(\ell)$ split, it suffices to check that these functors are additive. This follows from their construction.
\end{proof}

  We fix again an object $(V,\alpha) \in \Norm^\circ(\ell)$.
  
  \begin{lemma} \label{SpecFilSplittable}
   Let $\gamma \in \Gamma$. There exists a $\Gamma$-grading on $V^{\alpha\leq\gamma}_{\tilde\ell}$ which splits the filtration $(F^{\leq\delta}(V^{\alpha\leq\gamma}_{\tilde\ell}))_{\delta}$. 
 \end{lemma}
 \begin{proof}
   Let $(e_i)_{1\leq i \leq n}$ be a splitting basis of $(V,\alpha)$. This induces a decomposition
   \begin{equation*}
     V^{\alpha\leq\gamma}_{\tilde\ell} \cong \oplus_i \ell^{\leq \gamma/\alpha(e_i)} \otimes \tilde \ell.
   \end{equation*}
   For each $1\leq i \leq n$, the $\ell^\circ$-module $\ell^{\leq\gamma/\alpha(e_i)}$ is either free of rank one, in which case we choose a generator $x_i$ of this module, or satisifies $\ell^{\circ\circ} \ell^{\leq\gamma/\alpha(e_i)}=\ell^{\leq\gamma/\alpha(e_i)}$, in which case $\ell^{\leq\gamma/\alpha(e_i)}_{\tilde\ell}$ is zero. So images of the $x_i$ form a basis of $V^{\alpha\leq\gamma}_{\tilde \ell}$ and giving each $x_i$ the weight $\alpha(x_i)/\gamma$ defines a $\Gamma$-grading as desired.
 \end{proof}
\end{definition}

\begin{definition} \label{IellDef}
  We let
  \begin{equation*}
    I_\ell\defeq \{ \gamma \in \Gamma^{\leq 1} \mid \forall x \in \ell^{\circ\circ}\colon |x| < \gamma \}.
  \end{equation*}
\end{definition}
This is a convex subset of $\Gamma$ with maximal element $1$. For example if $\Gamma$ and $|\ell^*|$ have rank one, then $I_\ell=\{1\}$ if $|\ell^*|$ is dense in $\Gamma$, and $I_\ell=(|\pi|,1]$ for a generator $\pi$ of $\ell^{\circ\circ}$ otherwise.

\begin{lemma} \label{IellInj}
  Any class $[\mu]\in \Gamma/|\ell^*|$ contains at most one element of $I_\ell$.
\end{lemma}
\begin{proof}
   If $\delta \in I_\ell$ and $|x|\delta \in I_\ell$ for some $x \in \ell^{\circ\circ}$, then $|x| < |x|\delta$ and hence $1< \delta$, a contradiction. 
\end{proof}
\begin{lemma} \label{cl}
  Let $\gamma \in \Gamma$ and $\delta \in \Gamma^{<1}$.
  \begin{enumerate}
  \item If $\delta \in I_\ell$, then the surjection $V^{\alpha\leq\delta\gamma} \to F^{\leq \delta}(V^{\alpha\leq\gamma}_{\tilde \ell})$ factors through an isomorphism
    \begin{equation*}
     V^{\alpha\leq\delta\gamma}/V^{\alpha<\delta\gamma} \isoto \gr^\delta(V^{\alpha\leq\gamma}_{\tilde \ell}).
    \end{equation*}
  \item If $\delta \not \in I_\ell$, then $F^{\leq\delta}(V^{\alpha\leq\gamma}_{\tilde \ell})=0$.
  \end{enumerate}
\end{lemma}
\begin{proof}
  (i) By the definition of $\gr^\delta(V^{\alpha\leq\gamma}_{\tilde \ell})$, the composition
  \begin{equation}
    \label{eq:blubcomp24}
     V^{\alpha\leq\delta\gamma} \to F^{\leq\delta}(V^{\alpha\leq\gamma}_{\tilde\ell}) \to  \gr^\delta(V^{\alpha\leq\gamma}_{\tilde\ell})
  \end{equation}
  has kernel $V^{\alpha<\gamma\delta}+\left(\ell^{\circ\circ} V^{\alpha\leq\gamma} \cap V^{\alpha\leq\gamma\delta}\right)$. So it suffices to show that $\ell^{\circ\circ} V^{\alpha\leq \gamma} \subset V^{\alpha<\gamma\delta}$. This follows from the assumption that $\delta\in I_\ell$.
 
(ii) If $\delta \not\in I_\ell$, then there exists an element $x \in \ell^{\circ\circ}$ for which $\delta \leq |x|$. This implies that the morphism $V^{\alpha\leq\delta\gamma}\to V^{\alpha\leq\gamma}_{\tilde\ell}$ is zero.
\end{proof}

In the following, we endow $\End(\alpha)_{\tilde\ell}=(V^* \otimes V)^{\alpha\leq 1}_{\tilde\ell}$ with the filtration by the submodules $F^{\leq\delta}((V^*\otimes V)^{\alpha\leq 1}_{\tilde\ell})$ associated to the normed vector space $V^* \otimes V$. It follows from the definition of this filtration that it is multiplicative for the ring structure on $\End(\alpha)_{\tilde\ell}$ in the sense that
\begin{equation*}
  F^{\leq\delta}(\End(\alpha)_{\tilde\ell}) \cdot F^{\leq\delta'}(\End(\alpha)_{\tilde\ell}) \subset F^{\leq\delta\delta'}(\End(\alpha)_{\tilde\ell})
\end{equation*}
for all $\delta,\delta'\leq 1$.

To simplify the notation, in this subsection, for an element $e \in \UEnd(\alpha)(R)$, we denote by $\tilde e \in \End(\alpha)_R$ the element corresponding to $e$ under the isomorphism \eqref{EndCompMor}.

\begin{lemma} \label{EaDesc}
  Let $R$ be a $\tilde\ell$-algebra, let $h \in \UEnd(\alpha)(R)$ and let $\delta \in \Gamma^{\leq 1}$. The following are equivalent:
  \begin{enumerate}
  \item For every $\gamma \in \Gamma$, every $\delta'\in \Gamma^{\leq 1}$ and every $v \in F^{\leq \delta'}(V^{\alpha\leq \gamma}_R)$, the element $h(v)-v$ of $V^{\alpha\leq\gamma}_R$ lies in $F^{\leq\delta\delta'}(V^{\alpha\leq\gamma}_R)$.
  \item The element $\tilde h - \id$ of $\End(\alpha)_R$ lies in $F^{\leq\delta}(\End(\alpha)_R)$.
  \end{enumerate}
\end{lemma}
\begin{proof}
  We fix a splitting basis $(e_1,\hdots,e_n)$ of $(V,\alpha)$. As in the proof of Proposition \ref{UAutComp0}, such a basis induces isomorphisms
  \begin{equation*}
      \End(\alpha)_R \cong \oplus_{i,j} \ell^{\leq \alpha(e_i)/\alpha(e_j)}\otimes R
    \end{equation*}
    and
    \begin{equation*}
      V^{\alpha\leq\gamma}_R \cong \oplus_i \ell^{\leq \gamma/\alpha(e_i)} \otimes R
    \end{equation*}
    for all $\gamma \in \Gamma$. Let $\tilde h-\id$ correspond to $(a_{ij} \otimes r_{ij})_{i,j}$ under the first isomorphism. Then condition (ii) is equivalent to the condition that each entry $a_{ij} \otimes r_{ij}$ is in the image of $\ell^{\leq\delta\alpha(e_i)/\alpha(e_j)} \otimes R \to \ell^{\leq \alpha(e_i)/\alpha(e_j)} \otimes R$. In the same way, an element $(x_i \otimes r_i)_i$ of $V^{\alpha\leq \gamma}_R$ (with each $x_i$ in $\ell^{\leq\gamma/\alpha(e_i)}$ and $r_i$ in $R$) lies in $F^{\leq\delta}(V^{\alpha\leq\gamma}_R)$ if and only if each $x_i \otimes r_i$ lies in the image of $\ell^{\leq\delta\gamma/\alpha(e_i)}\otimes R$. 

    By the proof of Proposition \ref{UAutComp0}, for each $1\leq i\leq n$ the coefficients $(a_{ij} \otimes r_{ij})_j$ of $\tilde h$ are obtained as the image of $e_i \otimes 1 \in V^{\alpha\leq\alpha(e_i)}_R$ under $h$. Using this it follows that (i) implies (ii).

    The fact that (ii) implies (i) follows from the fact that for $\gamma \in \Gamma$ and $\delta, \delta' \in \Gamma^{\leq 1}$, the morphism $\End(\alpha) \otimes V^{\alpha\leq \gamma} \to V^{\alpha\leq \gamma}$ from Construction \ref{UAutCompCons0} induces a morphism
    \begin{equation*}
      F^{\leq\delta}(\End(\alpha)_R) \otimes F^{\leq\delta'}(V^{\alpha\leq\gamma}_R) \to F^{\leq \delta\delta'}(V^{\alpha\leq\gamma}_R).
    \end{equation*}

\end{proof}
\begin{definition}
  \begin{enumerate}
  \item  For an $\tilde\ell$-algebra $R$, we let $U_{\leq \delta}(\alpha)(R) \subset \UAut(\alpha)(R)$ be the set of elements $\phi \in \UAut(\alpha)(R)$ which (considered as an element of $\UEnd(\alpha)(R)$) satisfy the equivalent conditions of Lemma \ref{EaDesc}.
  \item These subsets are functorial in $R$ and so we obtain a sub-presheaf $U_{\leq \delta}(\alpha) \subset \UAut^\otimes(\alpha)_{\tilde\ell}$.
  \item We also define a sub-presheaf $U_{<\delta}(\alpha) \subset \UAut(\alpha)_{\tilde \ell}$ by $U_{<\delta}(\alpha)\defeq \colim_{\mu<\delta} U_{\leq \mu}(\alpha)$.
  \end{enumerate}
\end{definition}
Note that in particular $U_{\leq 1}(\alpha)=\UAut(\alpha)_{\tilde\ell}$.

The following uses the fact that $\UAut(\alpha)_{\tilde\ell}$ is representable by an affine group scheme of finite type over $\tilde\ell$ by Lemma \ref{UAutReprConc}.
\begin{proposition}\label{UdRepr1}
  Let $\delta \in \Gamma^{\leq 1}$.

      The presheaves $U_{\leq \delta}(\alpha)$ and $U_{<\delta}(\alpha)$ are representable by closed subgroup schemes of $\UAut^\otimes(\alpha)_{\tilde\ell}$. 
    \end{proposition}
    
\begin{proof}
  For a $\tilde\ell$-algebra $R$, the fact that $U_\delta(\alpha)(R)$ is a subgroup of $\UAut(\alpha)(R)$ follows from the fact that
  \begin{equation*}
    (\tilde u \tilde u' - \id)=(\tilde u-\id)\tilde u' + (\tilde u' - \id)
  \end{equation*}
  in $\End(\alpha)_R$ for all sections $u,u' \in U_{\leq \alpha}(R)$ using the multiplicativity of the filtration on $\End(\alpha)_R$ and Lemma \ref{EaDesc} (ii).

  To show that $U_{\leq\delta}(\alpha)$ and $U_{<\delta}(\alpha)$ are representable, one uses Proposition \ref{UAutRepr0} (iii) to reduce to the analogous claim for a single module $V^{\alpha\leq \gamma}_{\tilde \ell}$, which can be readily verified.
\end{proof}

\begin{lemma} \label{UComm1}
  Let $\delta, \delta' \in \Gamma^{\leq 1}$.
  \begin{enumerate}
  \item  The commutator of $U_{\leq \delta}(\alpha)$ and $U_{\leq \delta'}(\alpha)$ is contained in $U_{\leq \delta\delta'}(\alpha)$.
  \item  The commutator of $U_{\leq \delta}(\alpha)$ and $U_{< \delta'}(\alpha)$ is contained in $U_{< \delta\delta'}(\alpha)$.
  \end{enumerate}
\end{lemma}
\begin{proof}
  For sections $u$ of $U_{\leq \delta}(\alpha)$ and $v$ of $U_{\leq \delta'}(\alpha)$ we can write
  \begin{equation*}
    \tilde u\tilde v \tilde u^{-1}\tilde v^{-1}-\id = (\tilde u-\id)(\tilde v -\id)\tilde u^{-1} \tilde v^{-1} +(\tilde v - \id)(\tilde u^{-1}-\id)\tilde v^{-1}
  \end{equation*}
in $\End(\alpha)_R$. Using the multiplicativity of the filtration on $\End(\alpha)_R$ and Lemma \ref{EaDesc} (ii) this implies the claim.
\end{proof}

\begin{definition}
  Let $\delta \in \Gamma^{\leq 1}$. By Lemma \ref{UComm1}, the group scheme $U_{<\delta}(\alpha)$ is normal in $U_{\leq \delta}(\alpha)$. Hence we may consider the quotient scheme
  \begin{equation*}
    \gr_{\delta} \UAut^\otimes(\alpha)_{\tilde \ell} \defeq U_{\leq \delta}(\alpha)/ U_{<\delta}(\alpha).
  \end{equation*}
\end{definition}

Next we show that the group schemes $\gr_\delta \UAut(\alpha)_{\tilde \ell}$ for $\delta < 1$ are vector group schemes over $\tilde \ell$ and hence that $U_{<1}(\alpha)$ is a split unipotent group scheme. Here, for a vector space $W$ over some field $K$, we denote by $\BV(W)$ the associated vector group scheme satisfying $\BV(W)(K)=W$.

\begin{proposition} \label{GrdDesc1}
  Let $\delta \in \Gamma^{<1}$ and let $R$ be a $\tilde \ell$-algebra.
  \begin{enumerate}
  \item  The map
  \begin{align*}
    U_{\leq \delta}(R) &\to \gr^\delta(\End(\alpha)_R) = F^{\leq \delta}(\End(\alpha)_R)/F^{<\delta}(\End(\alpha)_R)\\
u  &\mapsto (\tilde u-\id)+F^{<\delta}(\End(\alpha)_R)
  \end{align*}
 is a group homomorphism with respect to the additive group structure on the target.
\item The homomorphism from (i) has kernel $U_{<\delta}(\alpha)(R)$.
\item  The induced homomorphism
  \begin{equation*}
      U_{\leq \delta}(\alpha)(R)/U_{<\delta}(\alpha)(R) \to \gr^\delta(\End(\alpha)_R)
  \end{equation*}
  is an isomorphism.
\item Hence we obtain an isomorphism
\begin{equation*}
  \gr_\delta \UAut^\otimes(\alpha)_{\tilde \ell} \to \BV(\gr^\delta(\End(\alpha)_{\tilde\ell})
\end{equation*}
of group schemes over $\tilde \ell$.
  \end{enumerate}
\end{proposition}
\begin{proof}
  (i) For $u,u' \in U_{<\delta}(R)$ we find
  \begin{equation*}
    (\tilde u \tilde u' - \id) - ((\tilde u-\id)+(\tilde u'-\id))=(\tilde u-\id)(\tilde u'-\id) \in F^{\leq \delta^2}(\End(\alpha)_R).
  \end{equation*}

  (ii) This follows from the definition of $U_{<\delta}(\alpha)(R)$.

  (iii) We already know from (ii) that this homomorphism is injective. Let $e \in F^{\leq \delta}(\End(\alpha)_R)$.  As a consequence of Lemma \ref{SpecFilSplittable}, for each non-zero $x \in \End(\alpha)_R$, the minimum
  \begin{equation*}
    d(x)\defeq \min \{\gamma \in \Gamma^{\leq 1} \mid x \in F^{\leq \delta}(\End(\alpha)_R) \}
  \end{equation*}
  exists, and the function $x \mapsto d(x)$ has only finitely many values. Since $e\cdot F^{\leq \delta'}(\End(\alpha)_{R}) \subset F^{\leq \delta\delta'}(\End(\alpha)_{R})$ for all $\delta' \in \Gamma^{\leq 1}$, we find $d(e\cdot x) < d(x)$ for all non-zero $x \in \End(\alpha)_R$. This implies that $e$ is nilpotent. Hence $1+e \in \End(\alpha)_R^*\cong \UAut(\alpha)(R)$. This element $1+e$ gives a preimage in $U_{\leq \delta}(\alpha)(R)$ of the class of $e$ in $\gr^\delta(\End(\alpha)_R)$. So the homomorphism in (iii) is surjective.
  
(iv) The isomorphism in (iii) is functorial in $R$. Hence the presheaf $$R \mapsto U_{\leq\delta}(\alpha)(R)/U_{<\delta}(\alpha)(R)$$ is a sheaf. This sheaf is then canonically isomorphic to $\gr_{\delta} \UAut^\otimes(\alpha)_{\tilde\ell}$. This gives (iv).
\end{proof}

\begin{corollary} \label{UaSplit}
  \begin{enumerate}
  \item For $\delta \in \Gamma^{<1}$ the group scheme $U_{\leq \delta}(\alpha)$ is a split unipotent group scheme (i.e. an extension of copies of the additive group scheme).
  \item For $\delta \in \Gamma^{\leq 1}$, there exists an element $\mu < \delta$ of $\Gamma$ for which $U_{<\delta}(\alpha)=U_{\leq \mu}(\alpha)$.
  \end{enumerate}
\end{corollary}
\begin{proof}
  We prove (i) by transfinite induction over $\Gamma^{<1}$. For a given $\delta < 1$, if (i) holds for all $\mu<\delta$, then all $U_{\leq \mu}(\alpha)$ for $\mu<\delta$ are smooth and connected and hence the increasing union $U_{<\delta}(\alpha)=\cup_{\mu<\delta}U_{\leq \mu}(\alpha)$ stabilizes. So $U_{<\delta}(\alpha)$ is a split unipotent group scheme and hence by Proposition \ref{GrdDesc1} so is $U_{\leq \delta}(\alpha)$.

  Then the stabilization argument just used shows (ii) for all $\delta\leq 1$. 
\end{proof}

We can also describe the group $\gr_0 \UAut(\alpha)_{\tilde \ell}$:
\begin{lemma} \label{Gr0Desc}
  Let $\gamma_1,\hdots,\gamma_r \in \Gamma$ be a set of representatives of $\alpha(V\setminus \{0\})/|\ell^*| \subset \Gamma/|\ell^*|$. Then the natural action of $\UAut(\alpha)_{\tilde \ell}$ on the $\tilde\ell$-vector space $\oplus_{i=1}^r V^{\alpha\leq\gamma_i}/V^{\alpha < \gamma_i}$ induces an isomorphism
  \begin{equation*}
    \gr_0 \UAut(\alpha)_{\tilde\ell} \isoto \prod_{i=1}^r \GL(V^{\alpha\leq\gamma_i}/V^{\alpha<\gamma_i}).
  \end{equation*}
\end{lemma}
\begin{proof}
  By definition of the $U_{\leq \delta}(\alpha)$, there exists a natural homomorphism
  \begin{equation*}
    \UAut(\alpha)_{\tilde\ell} \to \prod_{i=1}^r \GL(V^{\alpha\leq\gamma_i}/V^{\alpha<\gamma_i})
  \end{equation*}
  with kernel $U_{<1}(\alpha)$. It remains to prove that this homomorphism is surjective. For this, let $m$ be a non-archimedean extension of $\ell$ whose residue field $\tilde m$ is an algebraic closure of $\tilde \ell$.

  By rescaling an arbitrary splitting basis of $\alpha$, we choose a splitting basis $(e_j)_{j=1}^n$ of $\alpha$ for which $\alpha(e_j) \in \{\gamma_1,\hdots,\gamma_r\}$ for all $j$. Then the images $\bar e_j \in V^{\alpha\leq \alpha(e_j)}/V^{\alpha<\alpha(e_j)}$ of the $e_j$ form a basis of $\oplus_{i=1}^r V^{\alpha\leq\gamma_i}/V^{\alpha < \gamma_i}$. So if we write an element $x \in \prod_{i=1}^r \GL(V^{\alpha\leq\gamma_i}/V^{\alpha<\gamma_i})(\tilde m)$ as a matrix with respect to the basis $(\bar e_j)_j$, then by lifting the coefficients of such a basis to $m^\circ$ we get the matrix with respect to $(e_j)_j$ of a preimage of $x$ in $\UAut(m^\circ)$.
\end{proof}
\begin{remark}
  Alternatively, we could describe $\gr_0 \UAut(\alpha)_{\tilde\ell}$ as the scheme of automorphisms of $\gr(V,\alpha)$ which commute with the $\ell^*$-action on this vector space.
\end{remark}
\begin{lemma} \label{SplitFilComp}
  Let $R$ be an $\tilde\ell$-algebra an $\chi\colon D^{\Gamma/|\ell^*|}_R \to \UAut(\alpha)_R$ a homomorphism which splits $\alpha$ over $R$. For $\gamma \in \Gamma$ let
  \begin{equation*}
    V^{\alpha\leq\gamma}_R =\oplus_{[\mu]\in\Gamma/|\ell^*|} (V^{\alpha\leq\gamma}_R)^{[\mu]}
  \end{equation*}
  be the weight decomposition given by $\chi$.

  This decomposition splits the filtration $(F^{\leq\delta}(V^{\alpha\leq\gamma}_R))_{\delta\in\Gamma^{\leq 1}}$ in the following sense:
  \begin{enumerate}
  \item For each $\delta \in I_\ell$, the submodule $(V^{\alpha\leq\gamma}_R)^{[\delta\gamma]}$ is contained in $F^{\leq\delta}(V^{\alpha\leq\gamma}_R)$ and is isomorphic to $\gr^{\delta}(V^{\alpha\leq\gamma}_R)$ via the projection map $F^{\leq\delta}(V^{\alpha\leq\gamma}_R) \to \gr^\delta(V^{\alpha\leq\gamma}_R)$.
  \item For each $[\mu] \in \Gamma/|\ell^*|$ for which $(V^{\alpha\leq\gamma}_R)^{[\mu]}$ is non-zero, there exists a unique $\delta \in I_\ell$ for which $[\mu]=[\delta\gamma]$.
  \item So we obtain canonical isomorphisms
    \begin{equation*}
      V^{\alpha\leq\gamma}_R =\oplus_{\delta \in I_\ell} (V^{\alpha\leq\gamma}_R)^{[\delta\gamma]}=\oplus_{\delta\in I_\ell} \gr^\delta(V^{\alpha\leq\gamma}_R).
    \end{equation*}
  \end{enumerate}
  
\end{lemma}
\begin{proof}
  The submodules $F^{\leq\delta}(V^{\alpha\leq\gamma}_R)$ are by construction preserved by the action of $\UAut(\alpha)_R$ and hence by the action of $D^{\Gamma/|\ell^*|}_R$ via $\chi$. Since $\chi$ splits $\alpha$ over $R$, by Lemma \ref{cl} it can only have weight $[\delta\gamma]$ on each non-zero subquotient $\gr^\delta(V^{\alpha\leq\gamma}_R)$.

  Furthermore by Lemmas \ref{IellInj} and \ref{cl}, the subquotients $\gr^\delta(V^{\alpha\leq\gamma}_R)$ can only be non-zero for $\delta \in I_\ell$ and any class $[\mu]\in \Gamma/|\ell^*|$ contains at most one element of $I_\ell$.

These considerations imply (i) and (ii). Then (iii) is given by the combination of (i) and (ii).
\end{proof}

\begin{lemma} \label{SplitConj}
  Let $R$ be a $\tilde \ell$-algebra. For any two homomorphisms $$\chi,\chi'\colon D^{\Gamma/|\ell^*|}_R \to \UAut(\alpha)_R$$ which split $\alpha$ over $R$, there exists a unique section $u \in U_{<1}(\alpha)(R)$ for which $\chi'=\leftexp{u}{\chi}$.
\end{lemma}
\begin{proof}
  By Lemma \ref{SplitFilComp}, for $\gamma \in \Gamma$, both $\chi$ and $\chi'$ induce an isomorphism $$V^{\alpha\leq\gamma}_R \cong \oplus_{\delta\in I_\ell} \gr^\delta(V^{\alpha\leq\gamma}_R).$$ By composing one of these isomorphism with the inverse of the other we obtain an automorphism of $V^{\alpha\leq\gamma}_R$. For varying $\gamma$, these automorphisms form an element $u\in \UAut(\alpha)(R)$. The construction of $u$ implies that it lies in $U_{<1}(\alpha)(R)$ and that via $u$ the homomorphisms $\chi$ and $\chi'$ are conjugate.

 Along the same lines one sees that any $u \in U_{<1}(\alpha)(R)$ for which $^u \chi=\chi'$ must arise in this way. This shows the uniqueness of $u$.
\end{proof}

\begin{lemma}  \label{SplittingUnique0}
 Let $R$ be a $\ell^\circ$-algebra such that every connected component $C$ of $\Spec(R)$ has non-empty special fiber $C_{\tilde \ell}$. Let $\CT \subset \UAut(\alpha)_R$ be a split closed subtorus. There exists at most one homomorphism $\chi\colon D^{\Gamma/|\ell^*|}_R \to \CT$ which splits $\alpha$ over $R$. 

\end{lemma}
\begin{proof}

  For $\gamma \in \Gamma$, we consider the weight decomposition 
  \begin{equation*}
      V^{\alpha\leq\gamma}_R=\oplus_{w \in X^*(\CT)}(V^{\alpha\leq\gamma}_R)^w
    \end{equation*}
    given by the action of $\CT$ on $V^{\alpha\leq\gamma}_R$. Any homomorphism $\chi\colon D^{\Gamma/|\ell^*|}_R \to \CT$ is uniquely determined by the weights $\langle \chi,w \rangle \in \Gamma/|\ell^*|$ with which it acts on the modules $(V^{\alpha\leq\gamma}_R)^w$ for all $w \in X^*(\CT)$ and $\gamma \in \Gamma$. In particular $\chi$ is uniquely determined by its geometric fibers. Hence using the assumption on $R$ we may replace $R$ by $R_{\tilde \ell}$ and assume that $R$ is a $\tilde \ell$-algebra.

    Then on $V^{\alpha\leq\gamma}_R$ we have the filtration by submodules $F^{\leq\delta}(V^{\alpha\leq\gamma}_R)$ which is by construction preserved by the action of $\UAut(\alpha)_R$ and hence of $\CT$. If $\chi$ splits $\alpha$ over $R$, then by Lemma \ref{cl} its weights on the subquotients $\gr^\delta(V^{\alpha\leq\gamma}_R)$ are uniquely determined. Hence $\chi$ is uniquely determined. 
\end{proof}

For the following, note that by the conjugacy of split maximal tori, any smooth commutative group scheme of finite type over a field has a unique maximal split torus.

\begin{lemma} \label{SplitLift}
  Let $T \subset \UAut(\alpha)_{\tilde\ell}$ be a subtorus and $\chi\colon D^{\Gamma/|\ell^*|}_{\tilde\ell} \to T \subset \UAut(\alpha)_{\tilde\ell}$ a homomorphism which splits $\alpha$ over $\tilde\ell$.
  \begin{enumerate}
  \item The homomorphism $\chi$ factors through the unique maximal split subtorus of $T$.
  \item The homomorphism $\chi$ can be lifted to a homomorphism $D^\Gamma_{\tilde\ell} \to T$ which splits $\alpha$ over $\tilde\ell$.
  \end{enumerate}
\end{lemma}
\begin{proof}
  Since by Lemma \ref{SplitConj} any two such $\chi$ are conjugate over $\tilde\ell$, it suffices to prove the claim for any single such $\chi$. So we may take $\chi$ to be the composition $D^{\Gamma/|\ell^*|}_{\tilde\ell} \to D^{\Gamma}_{\tilde\ell} \toover{\tilde\chi_{\tilde\ell}} \UAut(\alpha)_{\tilde\ell}$ for a homomorphism $\tilde\chi\colon D^\Gamma_{\ell^\circ} \to \UAut(\alpha)$ splitting $\alpha$ over $\ell^\circ$, which exists by Lemma \ref{SplitEq0} and Proposition \ref{SplitExtEquiv0}. Then (ii) is satisfied by construction and (i) follows from the fact that since $\Gamma$ is torsion-free, any homomorphism $D^\Gamma_{\tilde\ell} \to T$ factors through the unique maximal split subtorus of $T$.
\end{proof}

Let now $m$ be a non-archimedean extension of $\ell$ satisfying $|m^*|=\Gamma$.

\begin{lemma} \label{VKer}  
 For $\gamma \in \Gamma$, the composition
  \begin{equation*}
    V^{\alpha \leq \gamma} \otimes_{\ell^\circ} m^\circ \into V_m^{\alpha_m \leq \gamma} \to V^{\alpha_m \leq \gamma}_{\tilde m}
  \end{equation*}
  has kernel $V^{\alpha < \gamma}\otimes_{\ell^\circ} m^\circ + V^{\alpha\leq \gamma}\otimes_{\ell^\circ} m^{\circ\circ}$.
\end{lemma}
\begin{proof}
  Using a splitting basis of $(V,\alpha)$, one reduces to the case $\dim(V)=1$. Then the claim can be readily verified.
\end{proof}

\begin{proposition} \label{GrCompProp}
   Let $\gamma\in \Gamma$ and $\delta \in I_\ell$. For any $x \in m$ of norm $\delta$, the composition
  \begin{equation} \label{CompComp}
    V^{\alpha \leq \delta\gamma} \otimes m^{\circ} \into V_m^{\alpha_m \leq \delta\gamma} \vartoover{30}{v \mapsto (1/x)v} V_m^{\alpha_m \leq \gamma} \to V^{\alpha_m \leq \gamma}_{\tilde m}
  \end{equation}
  induces a bijection
  \begin{equation*}
    \gr^{\delta}(V^{\alpha\leq\gamma}_{\tilde m}) \isoto \gr^{m/\ell}_{[\delta\gamma]}(V^{\alpha_m\leq\gamma}_{\tilde m}).
  \end{equation*}
\end{proposition}
\begin{proof}
  Since $(1/x)(V^{\alpha\leq\delta\gamma}\otimes m^\circ)=V^{\alpha\leq\delta\gamma}\otimes m^{1/\delta}$ inside $V_m$, the composition \eqref{CompComp} has image $\gr^{m/\ell}_{[\delta\gamma]}(V^{\alpha_m\leq \gamma}_{\tilde m})$.

  By Lemma \ref{VKer}, this composition has kernel
  \begin{equation*}
    V^{\alpha<\delta\gamma}\otimes m^{\circ} + V^{\alpha\leq\delta\gamma}\otimes m^{\circ\circ}.
  \end{equation*}
  Hence it factors through $\left(V^{\alpha\leq \delta\gamma}/V^{\alpha < \delta\gamma}\right) \otimes \tilde m$, which is canonically isomorphic to  $\gr^{\delta}(V^{\alpha\leq\gamma}_{\tilde m})$ by Lemma \ref{cl}.
\end{proof}

\subsection{Tamely Ramified Descent}
We fix a non-archimedean field $\ell$.
\begin{definition} \label{GalNormAction}
  If $m$ is a non-archimedean extension of $\ell$ which is Galois over $\ell$ and $V$ is an $\ell$-vector space, then the natural $\Gal(m/\ell)$-action on $V_m$ induces an action on the set of norms on $V_m$ via $(\sigma\cdot \alpha)(v)=\alpha(\sigma^{-1}(v))$.

  One checks that this action sends splittable norms to splittable norms.
\end{definition}

In this subsection, we prove the following descent result:

\begin{theorem} \label{TameDesc0}
  Assume that $\Gamma$ is of rank one and $\ell$ is Henselian and discretely valued. Let $m$ be a non-archimedean extension of $\ell$ which is Galois over $\ell$ and tamely ramified.  For a finite-dimensional vector space $V$ over $\ell$, if a splittable norm $\alpha$ on $V_m$ is invariant under the $\Gal(m/\ell)$-action on $V_m$, then it arises by base change from a unique splittable norm on $V$.
\end{theorem}
The proof of this result will take up the remainder of this subsection. We fix the setup from Theorem \ref{TameDesc0}.

 If $\alpha$ arises by base change from a norm on $V$, this norm is equal to $\alpha|_V$. This shows the uniqueness of such a norm, if it exists. We will descend $\alpha$ to a splittable norm on $V$ by showing that there exists a splitting basis of $\alpha$ consisting of elements of $V$. Since any splitting basis of $(V,\alpha)$ is already defined over a finite intermediate extension of $\ell \subset m$, we may and do assume that $m$ is finite over $\ell$.

 We note that since $\Gal(m/\ell)$ preserves $\alpha$, for all $\gamma \in \Gamma$ the natural action of $\Gal(m/\ell)$ on $V_m$ restricts to actions on the balls $V_m^{\alpha\leq\gamma}$ and $V^{\alpha<\gamma}_m$ and hence induces a natural action on $V^{\alpha\leq\gamma}_m / V^{\alpha<\gamma}_m$. 

\begin{lemma} \label{GalCoh0}
  For all $\gamma \in \Gamma$, the Galois cohomology groups $H^1(\Gal(m/\ell),V^{\alpha\leq\gamma}_m)$ and $H^1(\Gal(m/\ell),V^{\alpha<\gamma}_m)$ are the zero group.
\end{lemma}
\begin{proof}
  The fact that $\ell$ and $m$ are discretely normed implies $V^{\alpha<\gamma}_m=V_m^{\alpha\leq \gamma'}$ for some $\gamma'<\gamma$ in $\Gamma$. So it suffices to consider $H^1(\Gal(m/\ell),V^{\alpha\leq\gamma})$.

  For any subspace $U \subset V$, the subspace norm $\alpha|_{U_m}$ and the quotient norm $\beta$ on $(V/U)_m$ given by Proposition \ref{NormSplit} induce an exact sequence
  \begin{equation*}
    0 \to U_m^{\alpha\leq \gamma} \to V_m^{\alpha\leq\gamma} \to (U/V)_m^{\beta\leq\gamma} \to 0
  \end{equation*}
  of $\Gal(m/\ell)$-modules. Using the associated exact sequence of Galois cohomology groups, one reduces by induction to the case $\dim(V)=1$. Then any non-zero element of $V$ gives a splitting basis of $(V_m,\alpha)$, and so after rescaling $\alpha$ by an element of $\Gamma$ we may assume that $V=\ell$ and that $\alpha$ is the norm function on $m$.

  So we are considering the group $H^1(\Gal(m/\ell),m^{\leq \gamma})$, which lies in the exact sequence
  \begin{equation*}
    H^1(\Gal(m/\ell),m^{<\gamma}) \to H^1(\Gal(m/\ell),m^{\leq\gamma}) \to H^1(\Gal(\tilde m / \tilde \ell),m^{\leq\gamma}/m^{<\gamma})=0,
  \end{equation*}
  in which the last cohomology group vanishes by Galois descent from $\tilde m$ to $\tilde \ell$. By applying this to all $\gamma'< \gamma$ in $|m^*|$, it follows that the natural homomorphism
  \begin{equation*}
    \varprojlim_{\gamma'< \gamma, \gamma'\in |m^*|} H^1(\Gal(m/\ell),m^{\leq \gamma'}) \to H^1(\Gal(m/\ell),m^{\leq \gamma})
  \end{equation*}
  is surjective. Since Galois cohomology commutes with inverse limits, this implies the claim.
\end{proof}

\begin{lemma} \label{GrDescLemma0}
  If there exist $\Gamma$-invariant elements $f_1,\hdots,f_{\dim(V)}$ of $\gr(V_m,\alpha)$ which form a $\ell^*$-basis of $\gr(V_m,\alpha)$, then $\alpha|_V$ is splittable.
\end{lemma}
\begin{proof}
  For each $\gamma \in \Gamma$, by Lemma \ref{GalCoh0} the sequence
  \begin{equation*}
    0 \to (V_m^{\alpha < \gamma})^{\Gal(m/\ell)} \to (V_m^{\alpha\leq \gamma})^{\Gal(m/\ell)} \to (V_m^{\alpha\leq\gamma}/V_m^{\alpha<\gamma})^{\Gal(m/\ell)} \to 0
  \end{equation*}
  is exact. So we can lift each $f_i \in (V_m^{\alpha\leq\gamma_i}/V_m^{\alpha<\gamma_i})^{\Gal(m/\ell)}$ to an element $e_i \in V^{\alpha\leq\gamma_i}_m \cap V$. Then by Lemmas \ref{EllInd} and \ref{EllBase} the $e_i$ form a splitting basis of $(V_m,\alpha)$.   
\end{proof}

\begin{proof}[Proof of Theorem \ref{TameDesc0}]
  We first consider the structure of the field extension $\ell \subset m$: Let $\pi$ be a uniformizer of $m$ and $e$ the ramification degree of $m$ over $\ell$. By assumption, the integer $e$ is invertible in $\tilde \ell$. There exists an element $u \in m$ of norm one such that $u\pi^e  \in \ell$. After potentially replacing $m$ by a finite unramified Galois extension, we can assume that the image $\bar u \in \tilde m$ of $u$ admits an $e$-th root in $\tilde m$. Since $\ell$ is Henselian, so is $m$. Hence, since $e$ invertible in $\tilde \ell$, the element $u$ admits an $e$-th root in $m$. Then after dividing $\pi$ by such a root we can assume that $\pi^e \in \ell$.

  Then for any $\sigma \in \Gal(m/\ell)$, the element $\zeta=\sigma(\pi)/\pi \in m^\circ$ is a $e$-th root of unity, and by \cite[9.15]{Neukirch}, the assignment $\sigma \mapsto \zeta$ gives an isomorphism from the inertia subgroup $I \subset \Gal(m/\ell)$ to the group of $e$-th roots of unity in $m^\circ$. In particular $I$ is cyclic and the order of $I$ divides $e$.

  Let now $S=\{\gamma_1, \hdots, \gamma_r \} \subset \Gamma$ be a set of representatives of $\alpha(V_m \setminus \{0\}) /|m^*| \subset \Gamma/|m^*|$. Since the inertia group $I$ acts trivially on $\tilde m$, it acts linearly on each $\tilde m$-vector space $V^{\alpha\leq\gamma_r}_m/V^{\alpha<\gamma_r}_m$. Since $I$ is cyclic and the order of $I$ is invertible in $\tilde m$, for each $1\leq s\leq r$ this action gives a weight decomposition
  \begin{equation*}
    V^{\alpha\leq\gamma_s}_m/V^{\alpha<\gamma_s}_m = \oplus_{\mu \colon I \to \tilde m^*} (V^{\alpha\leq\gamma_s}_m/V^{\alpha<\gamma_s}_m)^\mu.
  \end{equation*}

  We choose a basis $f_{s,1},\hdots,f_{s,r_s}$ of $V^{\alpha\leq\gamma_s}_m/V^{\alpha<\gamma_s}_m$ such that each $f_{s,i}$ is homogenous of some weight $\mu_{s,i}$ for this decomposition. Let $\tau$ be a generator of $I$. By the above $\tau(\pi)/\pi \in m^\circ$ is a generator of the group of $e$-root of unity in $m$. Hence there exists for all $s,i$ an integer $t_{j,i}$ such that $\tau(\pi^{t_{s,i}})/\pi^{t_{s,i}} \in m^\circ$ reduces to the $e$-th root of unity $\mu_{s,i}(\tau) \in \tilde m$. This implies that the elements $\pi^{-t_{s,i}} \cdot f_{s,i}$ form a $\ell^*$-basis of $\gr(V_m,\alpha)$ which consists of $I$-invariant elements. Hence by Lemma \ref{GrDescLemma0} the norm $\alpha|_{V_{m^I}}$ on $V_{m^I}$ is splittable. So we can replace $m$ by $m^I$ and have reduced to the case that $m$ is unramified over $\ell$.

 If $m$ is unramified over $\ell$, then by Galois descent of vector spaces, there exists for each $1\leq s \leq r$ a basis of the $\tilde m$-vector space $V^{\alpha\leq\gamma_s}_m/V^{\alpha<\gamma_s}_m$ consisting of elements which are fixed by $\Gal(m/\ell)=\Gal(\tilde m / \tilde \ell)$. So another application of Lemma \ref{GrDescLemma0} finishes the proof.

\end{proof}

\section{Normed Fiber Functors} \label{NFFSection}

\subsection{Group Schemes over Valuation Rings}
In this subsection we collect some facts about affine group schemes. Presumably none of these are new, but in some cases we give a proof because we could not find a suitable reference.

Let first $R$ be a ring and $\CG$ a flat affine group scheme over $R$. We denote by $\Rep^\circ \CG$ the exact $R$-linear category of dualizable representations of $\CG$ over $R$ and by $\Rep \CG$ the category of all representations of $\CG$ over $R$.

For an $R$-algebra $R'$, a fiber functor on $\Rep^\circ \CG$ over $R'$ is an exact $R$-linear tensor functor $\Rep^\circ \CG \to \Mod_{R'}$. We denote by $\omega_\CG\colon \Rep^\circ \CG \to \Mod_R$ the standard fiber functor (i.e. the forgetful functor).

\begin{definition}
    Let $R' \to R''$ be a homomorphism of $R$-algebras and $\omega\colon \Rep^\circ \CG \to \Mod_{R'}$ a fiber functor.
  \begin{enumerate}

  \item Let $\UEnd(\omega)(R'')$ be the $R''$-module of natural transformations $\omega\to \omega$ from the fiber functor $\omega_{R''}\colon \Rep^\circ \CG \to \Mod_{R'} \to \Mod_{R''}$ to itself. 
  \item Let $\UEnd^\otimes(\omega)(R'')$ be the submodule of $\End(\omega)(R')$ consisting of those natural transformations $h$ such that for all $X$ and $Y$ in $\Rep^\circ \CG$, the associated endomorphism $h_{X\otimes Y}$ of $\omega(X\otimes Y)_{R''}$ is equal to $h_{X} \otimes \id_{\omega(Y)_{R''}} + \id_{\omega(X)_{R''}} \otimes h_{Y}$. 
  \item We endow $\UEnd^\otimes(\omega)(R'')$ with the natural conjugation action of the group $\UAut^\otimes(\omega)(R'')$ of tensor automorphisms of $\omega_{R'' }$.
  \item These constructions are naturally functorial in the $R'$-algebra $R''$, and so we obtain presheaves $\UEnd^\otimes(\omega) \subset \UEnd(\omega)$ over $\Spec(R')$ with an action of $\UAut^\otimes(\omega)$.
  \item Let $\rho_\CG \in \Rep \CG$ be the regular representation (that is the action of $\CG$ on its ring of global sections induced by the conjugation action of $\CG$ on itself) and $I_\CG \subset \rho_\CG$ its augmentation ideal.
  \end{enumerate}
\end{definition}
\begin{remark}
  Note that $\UAut^\otimes(\omega)$ is not contained in $\UEnd^\otimes(\omega)$ in general, since the compatibility conditions with the tensor product used to define these functors are different.
\end{remark}

\begin{construction}
  Let $R'$ be an $R$-algebra and $\omega\colon \Rep^\circ \CG \to \Mod_{R'}$ a fiber functor. We identify the sheaf $\UEnd^\otimes(\omega)$ with the sheaf $\ULie(\UAut^\otimes(\omega))$ of Lie algebras associated to the group sheaf $\UAut^\otimes(\omega)$ in \cite[Def II.3.9.0.1]{SGA3I} as follows:

  By definition, for some $R'$-algebra $R''$, a section of $\ULie(\UAut^\otimes(\omega))(R'')$ is given by an element of $\ker(\UAut^\otimes(\omega)(R''[\epsilon]) \to \UAut^\otimes(R''))$. Such an element amounts to giving automorphisms $\id_{\omega(X)_{R''}}+\epsilon h_X$ of $\omega(X)_{R''[\epsilon]}$ for all $X \in \Rep^\circ \CG$ which are compatible with tensor products, where the $h_X$ are endomorphisms of $\omega(X)_{R''}$. One checks that these automorphisms are compatible with tensor products if and only if the $h_X$ form an element of $\UEnd^\otimes(\omega)(R'')$. This gives a natural isomorphism
  \begin{equation} \label{LieEndIso}
    \ULie(\UAut^\otimes(\omega)) \cong \UEnd^\otimes(\omega),
  \end{equation}
  which is compatible with the adjoint action of $\UAut^\otimes(\omega)$ on the left and the conjugation action of $\UAut^\otimes(\omega)$ on the right.

  If $\UAut^\otimes(\omega)$ is representable by a smooth group scheme over $\Spec(R')$, then (c.f. \cite[II.4.11.7]{SGA3I}) the functor $\ULie(\UAut^\otimes(\omega))$ is the representable by the vector bundle $\BV(\Lie(\UAut^\otimes(\omega))$ over $\Spec(R')$ whose global sections are given by the Lie algebra $\Lie(\UAut^\otimes(\omega))=\ULie(\UAut^\otimes(\omega))(R')$ of $\UAut^\otimes(\omega)$. So in this case we can write \eqref{LieEndIso} as
  \begin{equation} \label{LieEndIso2}
    \BV(\Lie(\UAut^\otimes(\omega)) \cong \UEnd^\otimes(\omega).
  \end{equation}
\end{construction}

We now fix a non-archimedean field $k$ and a flat affine group scheme $\CG$ over $k^\circ$. 

\begin{lemma} \label{LocalFiniteness}
  Any flat $\rho_\CG$-comodule over $k^\circ$ is the filtered union of its dualizable sub-comodules.
\end{lemma}
\begin{proof}
  Comodules under a coalgebra are the filtered union of their sub-comodules which are of finite type over the base ring. In our case, since we are working with a flat comodule over a valuation ring, these sub-comodules are torsion free and hence projective.
\end{proof}
By a strict monomorphism $X \into Y$ (resp. a strict epimorphism $Y \to Z$) in an exact category, we mean such a morphism which fits into an exact sequence $X \to Y \to Z$. A strict subquotient of an object $X$ in such a category is an object of the form $Y'/Y$ for two strict subobjects $Y \subset Y'$ of $X$.

The following is shown in \cite[Proposition 3.1]{MR3778991} in case $k^\circ$ is a discrete valuation ring. But the argument in \emph{loc.cit.} works over arbitrary valuation rings $k^\circ$.
\begin{proposition} \label{ClosedImmCrit}
  Let $h\colon \CH \to \CG$ be a homomorphism of flat affine group schemes over $k^\circ$. The following are equivalent:
  \begin{enumerate}
  \item The homomorphism $h$ is a closed immersion.
  \item Every object $V \in \Rep^\circ \CH$ is a strict subquotient of an object of the form $h^*(W)$ for some $W \in \Rep^\circ \CG$.
  \end{enumerate}
\end{proposition}

\begin{lemma} \label{TensorGenLemma}
  Let $\CG$ be a flat affine group scheme over $k^\circ$ and $X \in \Rep^\circ \CG$. The following are equivalent:
  \begin{enumerate}
  \item The homomorphism $G \to \GL(X)$ induced by the representation $X$ is a closed immersion.
  \item Every representation $Y \in \Rep^\circ \CG$ is a strict subquotient of $\oplus_{i=1}^n X^{\otimes r_i} \otimes (X^*)^{\otimes s_i}$ for some integers $r_i$ and $s_i$.
  \end{enumerate}
\end{lemma}
\begin{proof}
  First we note that for $\CG=\GL_n$ and $V$ the standard representation, condition (ii) is satisfied. Indeed, over a field, this is proven in \cite[Section 3.5]{MR547117}, and the same argument works over the valuation ring $k^\circ$. 

Using this, the claim follows from Proposition \ref{ClosedImmCrit}.
\end{proof}
\begin{definition}
  We call a representation $X \in \Rep^\circ \CG$ satisfying the conditions of Lemma \ref{TensorGenLemma} a \emph{tensor generator} of $\Rep^\circ \CG$.
\end{definition}

\begin{lemma}
  If $\CG$ is of finite type over $k^\circ$, then there exists a tensor generator $X \in \Rep^\circ \CG$.
\end{lemma}
\begin{proof}
By Lemma \ref{LocalFiniteness}, there exists a dualizable sub-comodule $X \subset \rho_{\CG}$ which generates $\rho_G$ as a $k^\circ$-algebra. We consider the $k^\circ$-linear map $X \otimes X^* \to \rho$ which is adjoint to the comodule map $X \to X\otimes \rho_\CG$. Using the fact that $\rho_\CG$ is a Hopf algebra, one checks that the composition $$X \toover{x \mapsto x \otimes \epsilon|_X} X \otimes X^* \to \rho_\CG,$$ where $\epsilon \colon \rho_\CG \to k^\circ$ denotes the counit, is the inclusion $X \into \rho$. Hence the induced algebra homomorphism $\operatorname{Sym}(X \otimes X^*) \to \rho_\CG$ is surjective. But this homomorphism corresponds to the map $G \to \GL(\omega_\CG(X)) \to \End(\omega_\CG(X))$ of schemes given by the action of $G$ on $\omega_\CG(X)$. Hence $G \to \GL(\omega_\CG(X))$ is a closed immersion.
\end{proof}

\begin{lemma} \label{BCEpi}
  Let $\ell$ be a non-archimedean extension of $k$. Every object $X \in \Rep^\circ \CG_{\ell^\circ}$ admits a strict epimorphism $Y_{\ell^\circ} \onto X$ for some object $Y \in \Rep^\circ \CG$.
\end{lemma}
\begin{proof}
   Let $\rho_\CG^\text{l} \in \Rep \CG$ be the left regular representation, that is the action of $\CG$ on its ring of function induced by the action of $\CG$ on itself by left multiplication. The comodule map $X \to \rho_\CG^\text{l} \otimes_{k^\circ} X$ identifies $X$ with a saturated sub-comodule of the comodule $\rho_\CG^\text{l} \otimes_{k^\circ} \omega_\CG(X) \in \Rep \CG_{\ell^\circ}$. (Here by a saturated sub-comodule we mean one for which the quotient $(\rho_\CG^\text{l} \otimes_{k^\circ} \omega_\CG(X)) / X$ is torsion-free or equivalently flat.) Using Lemma \ref{LocalFiniteness}, we find a finitely generated projective sub-comodule $Y \subset \rho_\CG^\text{l}$ such that $X \subset Y \otimes_{k^\circ} \omega(X) \subset  \rho_{\CG}^\text{l} \otimes_{k^\circ}  \omega(X)$. Since $X$ is saturated in $\rho_{\CG_{\ell^\circ}}^\text{l} \otimes_{k^\circ} \omega(X)$, it is also saturated in $Y \otimes_{k^\circ} \omega(X)$. In other words we have constructed a strict monomorphism $X \into Y_{\ell^\circ}^d$ for some $d\geq 0$. By dualizing this implies the claim.
\end{proof}

\begin{proposition} \label{FibrewiseSmoothness2}
  Let $\CG$ be an affine group scheme of finite type over $k^\circ$ with smooth special fiber $\CG_{\tilde k}$ such that $\dim(\CG_{\tilde k})=\dim(\CG_k)$.

  \begin{enumerate}
  \item The schematic closure $\CG^\text{fl} \subset \CG$ of the generic fiber is open and closed in $\CG$ as well as smooth over $k^\circ$.
  \item If in addition there exists a closed subgroup scheme $\CH \subset \CG$ which is smooth over $k^\circ$ and for which the quotient scheme $\CG_{\tilde k}/\CH_{\tilde k}$ is connected, then $\CG$ is smooth over $k^\circ$.
  \end{enumerate}
\end{proposition}
\begin{proof}
(i)  The closed subscheme $\CG^\text{fl} \subset \CG$ is a group scheme which is flat and of finite type over $k^\circ$. Hence it is of finite presentation over $k^\circ$ by \cite[Theorem 3]{MR193088}. By the upper semicontinuity of the dimension of fibres of finite type morphisms of schemes, all fibers of $\CG$ and $\CG^\text{fl}$ have the same dimension. Hence by our assumptions on $\CG$, the special fiber $\CG^\text{fl}_{\tilde k}$ is open in $\CG_{\tilde k}$ and hence smooth. By \cite[12.1.7]{EGAIV3}, the set of points $s \in \Spec(k^\circ)$ with a smooth fiber $\CG^\text{fl}_s$ is open in $\Spec(k^\circ)$. Hence all fibers of $\CG^\text{fl}$ are smooth, and so $\CG^\text{fl}$ is smooth over $k^\circ$.

  Similarly, by \cite[VIB.2.5]{SGA3II}, the fact that the special fiber of the inclusion $\CG^\text{fl} \into \CG$ is etale implies that this inclusion is etale everywhere. Hence this inclusion is an open immersion.

(ii)    It suffices to prove the claim after replacing $k^\circ$ by its strict Henselization. Let $C$ be a connected component of $\CG_{\tilde k}$. By assumption there exists an element $h \in \CH(\tilde k) \cap C(\tilde k)$. Since $\CH$ is smooth and $k^\circ $ Henselian, we can lift $h$ to a point $h'\in \CH(k^\circ)$. Then $\CG^\text{fl} \cdot h'$ is an open subscheme of $\CG$ which is smooth over $k^\circ$ and contains $C$. Hence this subscheme is contained in $\CG^\text{fl}$ and so $\CG^\text{fl}$ contains $C$. By varying $C$ over all connected components this shows that $\CG^\text{fl}_{\tilde k}=\CG_{\tilde k}$. By the upper semicontinuity of the dimension of the fibres of $\CG \setminus \CG^\text{fl}$ this implies $\CG=\CG^\text{fl}$, which proves the claim.
\end{proof}

\begin{proposition}[{c.f. \cite[Lemma 7]{CornutFiltrations} and \cite[A.8.10 (i)]{CGP}}] \label{CentSmooth}
  Let $D$ and $G$ be affine group schemes over some base scheme $S$, with $D$ multiplicative and with $G$ smooth. Then the centralizer $\Cent_G(\chi)$ of any homomorphism $\chi\colon D \to G$ over $S$ is a smooth closed subscheme of $G$.

  Moreover, the Lie algebra $\Lie(\Cent_G(\chi)) \subset \Lie(G)$ of the centralizer is equal to the subbundle $\Lie(G)^D$ of fixed points of the action of $D$ on $\Lie(G)$ via $\chi$ and the adjoint action of $G$ on $\Lie(G)$.
\end{proposition}

\begin{theorem}[Grothendieck] \label{ChiDeform}
  Assume that $k$ is Henselian and let $A$ be an abelian group. If $\CG$ is smooth over $k^\circ$, then every homomorphism $\chi\colon D^A_{\tilde k} \to \CG_{\tilde k}$ can be deformed to a homomorphism $D^A_{k^\circ} \to \CG$.
\end{theorem}
\begin{proof}
  By \cite[Theorem 1]{CornutFiltrations}, the functor of group homomorphisms $D^A_{k^\circ} \to \CG$ is representable by a smooth an separated scheme over $k^\circ$. Hence the fact that $k$ is Henselian implies that every $\tilde k$-point of this scheme can be lifted to $k^\circ$ by \cite[18.5.17]{EGAIV4}.
 \end{proof}

 In case $H$ is reductive, the following claim is given by \cite[Lemma 7.1]{MartensThaddeus} (due to B. Conrad). Our argument here contains significant simplifications suggested by the referee.
 \begin{proposition} \label{TorusFact}
   Let $p$ be the characteristic exponent of $\ell$ and let $A$ be  an abelian group whose prime-to-$p$ torsion subgroup embeds into $\BQ/\BZ$. Let $\ell$ be any field and $H$ a smooth connected algebraic group over $\ell$. Then any homomorphism $\chi\colon D^A_\ell \to H$ factors through a torus of $H$.
 \end{proposition}
 \begin{proof}
   We proceed by a series of reduction steps:
   
   Since $H$ is of finite type over $\ell$, by \cite[IX.6.8]{SGA3II} the homomorphism $\chi$ factors through a closed immersion $\chi'\colon D^{A'} \into H$ for some finitely generated subgroup $A' \subset A$. So after replacing $A$ by $A'$ we may assume that $A$ is finitely generated. Then we can write $A=A_1\times A_2 \times A_3$ where $A_1=\mathbb{Z}/n\mathbb{Z}$ for some $n$ prime to $p$, $A_2$ is a finite $p$-group and $A_3$ is free.

   We consider the centralizer $H_1 \defeq Z_G(D^{A_1})^\circ$, which is smooth by Proposition \ref{CentSmooth} and connected by construction. By a theorem of Steinberg (\cite[8.1]{MR0230728}) any semisimple element of $H$ is contained in a maximal torus. Therefore the factor $D^{A_1}$ of $D^A$ is contained in such a torus and hence in $H_1$. The factors $D^{A_2}$ and $D^{A_3}$ are contained in $H_1$ as well since they are connected. So we may replace $H$ by $H_1$ and assume that $D^{A_1}$ is central in $H$.

   Next we consider the subgroup $H_2 \defeq Z_G(D^{A_2})^\circ$ which is again smooth and connected. It again contains $D^{A_2}$ and $D^{A_3}$ since they are connected. We choose a maximal torus $T_2$ of $H_2$ which contains the torus $D^{A_3}$ and a maximal torus $T_1$ of $H$ containing $T_2$. By \cite[XII.4.5]{SGA3II} every central multiplicative closed subgroup scheme of a smooth affine group scheme over $\ell$ is contained in every maximal torus of such a group scheme.  Hence $D^{A_1} \subset T_1$ and $D^{A_2} \subset T_2 \subset T_1$. So $D^{A} \subset T_1$, which is what we wanted.
\end{proof}
\begin{lemma} \label{LiftingLemma}
  Assume that $\Gamma$ has rank one and that the norm on $k$ is discrete. Let $H$ be a smooth connected group scheme over $\tilde k$. Every homomorphism $\chi\colon D^{\Gamma/|k^*|}_{\tilde k} \to H$ factors through a maximal torus $T \subset H$.
\end{lemma}
\begin{proof}
  By assumption there exists an embedding $\Gamma/|k^*| \into \BR/\BZ$. So the claim is a special case of Proposition \ref{TorusFact}.

\end{proof}

\begin{definition}
  A \emph{residually maximal split} torus of $\CG$ is a split closed subtorus $\CT \subset \CG$ for which $\CT_{\tilde\ell}$ is a maximal split torus of $\CG_{\tilde k}$.
\end{definition}
\begin{proposition} \label{ResSplitProp}
  Assume that $\CG$ is smooth over $k^\circ$ and that $k$ is Henselian.
  \begin{enumerate}
  \item There exists a residually maximal split torus $\CT$ of $\CG$. More precisely, every maximal split torus of $\CG_{\tilde k}$ can be lifted to such a torus.
  \item Any two such tori are conjugate under $\CG(k^\circ)$.
    \item Every homomorphism $\chi\colon \CS \to \CG$ from a (not necessarily of finite type) multiplicative group scheme $\CS$ with torsion-free character group factors through a residually maximal split torus of $\CG$.
  \end{enumerate}
\end{proposition}
\begin{proof}
  (i) Let $T \subset \CG_{\tilde k}$ be a maximal split torus. By Theorem \ref{ChiDeform}, we can deform the inclusion $T \into \CG_{\tilde k}$ to a homomorphism $\CT \to \CG$ over $k^\circ$ whose source is a split torus $\CT$ over $k^\circ$. Then by \cite[IX.6.6 and IX.2.5]{SGA3II}, the fact that $\CT \to \CG$ is a closed immersion in the special fiber implies that it is a closed immersion.

  (ii) For two residually maximal split tori $\CT$ and $\CT'$ of $\CG$, we consider the subsheaf $\operatorname{Transp}_\CG(\CT,\CT')$ of $\CG$ consisting of those sections of $\CG$ which conjugate $\CT$ to $\CT'$. By \cite[XI.5.2]{SGA3II}, this is representable by a smooth closed subscheme of $\CG$. Since by the conjugacy of maximal split tori it has a point over $\tilde k$, the fact that $k$ is Henselian ensures that is has a point over $k^\circ$.

  (iii) We consider the centralizer $\Cent_{\CG}(\chi) \subset \CG$ which is smooth by Proposition \ref{CentSmooth}. Since $\CS_{\tilde k}$ is a split torus, there exists a maximal split torus $T \subset \CG_{\tilde k}$ which is contained in $\Cent_\CG(\chi)_{\tilde k}$. By \cite[XI.6.8]{SGA3II}, the schematic image of $\chi_{\tilde k}$ is a split subtorus of $\CG_{\tilde k}$. So since this image centralizes $\chi$, it is contained in $T$.

  By (i) we can lift $T$ to a residually maximal split torus $\CT$ of $\Cent_\CG(\chi)$. This is then a residually maximal split torus of $\CG$. Since $\CT$ is split, there exists a unique homomorphism $\tilde\chi\colon \CS \to \CT$ with special fiber $\chi_{\tilde k}$. Then \cite[XI.5.1]{SGA3II} implies $\chi=\tilde\chi$, which shows that $\chi$ factors through $\CT$.
\end{proof}

\begin{remark} \label{MaxSplitRmk}
 Proposition \ref{ResSplitProp} in particular implies that the split rank of $\CG_{\tilde\ell}$ is less or equal than the split rank of $\CG_\ell$. In case these split ranks agree, it follows that the residually maximal split tori of $\CG$ are those whose special and generic fibers are maximal split tori of the respective fibers of $\CG$. So in this case we will also refer to these tori simply as maximal split tori of $\CG$. The generic fibers of these are then those maximal split tori of $\CG_\ell$ which extend to a split subtorus of $\CG$.
\end{remark}

\begin{theorem} \label{Etoffe}
  Assume that $k^\circ$ is a Henselian discrete valuation ring. Let $\ell$ be a Henselian non-archimedean extension of $k$, let $\breve k$ be a strict Henselization of $k$, and let $\breve \ell$ be a strict Henselization of $\ell$ containing $\breve k$. Let $\CG$ be a smooth affine group scheme over $k^\circ$ and $\CH$ a smooth affine group scheme over $\ell^\circ$.

  Every homomorphism $h\colon \CG_\ell \to \CH_\ell$ which maps $\CG(\breve k^\circ)$ to $\CH(\breve \ell^\circ)$ extends uniquely to a homomorphism $\CG_{\ell^\circ} \to \CH$ over $\ell^\circ$.
\end{theorem}
\begin{proof}
  In case $k=\ell$, this is \cite[1.7.6]{BT2}. For general $\ell$, one can argue in the same way: If we show that
  \begin{equation} \label{EtoffeEq}
      \ell^\circ[\CG_{\ell^\circ}]=\{ f \in \ell[\CG_\ell] \mid f(\CG(\breve k^\circ)) \subset \breve \ell^\circ \},
  \end{equation}
  then a homomorphism $h$ as in the claim maps $\ell^\circ[\CH]$ to $\ell^\circ[\CG_{\ell^\circ}]$ and hence extends to a homomorphism $\CG_{\ell\circ} \to \CH$. If it exists, such an extension is unique by the flatness of $\CG$ and $\CH$.

  To show \eqref{EtoffeEq}, only the inclusion ``$\supseteq$'' requires an argument. So let $f \in \ell[\CG_\ell]$ be such that $f(\CG(\breve k^\circ)) \subset \breve \ell^\circ$.

  Since the valuation on $\ell$ is discrete, the set $\{a \in \ell^{\circ\circ} \mid af \in \ell^\circ[\CG_{\ell^\circ}] \subset \ell[\CG_\ell]\}$ contains an element of minimal valuation. Let $a$ be such an element. Then the image of $af$ in $\tilde\ell[\CG_{\tilde \ell}]$ is non-zero.

  But on the other hand, since $\breve k$ is strictly Henselian, the set $\CG(\tilde{ \breve k})$, which is equal to the image of $\CG(\breve k^\circ) \to \CG(\tilde{\breve k})$, is Zariski dense in $\CG_{\tilde{\breve k}}$. By the assumption on $f$, the image of $af$ in $\tilde\ell[\CG_{\tilde \ell}]$ vanishes on this Zariski-dense set and is hence zero. This is a contradiction.
\end{proof}

The following result generalizes Deligne's theorem on the fqpc-local isomorphy of any two fiber functors on a Tannakian category. In case $\ell^\circ$ is a discrete valuation ring it is due to Broshi (\cite[Theorem 1.2]{MR2965898}) and in case $\CG$ is smooth it is due to Lurie (\cite[5.11]{LurieTannaka}). For a more detailed discussion of the history of this result we refer the reader to the introduction of \cite{Tonini}.

\begin{theorem} \label{FFIso}
  Let $\ell$ be a non-archimedean extension of $k$. Any two fiber functors on $\Rep^\circ \CG$ over $\ell^\circ$ are isomorphic fpqc-locally on $\ell^\circ$.
\end{theorem}
\begin{proof}
  By applying \cite[Theorem A]{Tonini} to the stack $\CX=[\ast/\CG]$ of $\CG$-torsors for the fpqc topology, one obtains that the functor from the category of fpqc $\CG_{\ell^\circ}$-torsors over $\ell^\circ$ to the category of fiber functors $\Rep^\circ \CG \to \Mod_{\ell^\circ}$ which sends a torsor to the associated twist of the standard fiber functor is an isomorphism. This statement is equivalent to the one of Theorem \ref{FFIso}.
\end{proof}

\subsection{Normed Fiber Functors}
We fix non-archimedean fields $k \subset \ell$ and and a flat affine group scheme $\CG$ of finite type over $k^\circ$.
\begin{definition} \label{NFFDef}
  \begin{enumerate}
  \item A \emph{normed fiber functor} on $\Rep^\circ \CG$ over $\ell$ is a $k^\circ$-linear exact tensor functor $\alpha \colon \Rep^\circ \CG \to \Norm^\circ(\ell)$.
  \item A \emph{morphism} of normed fiber functors is a morphism of tensor functors.
  \item Given a non-archimedean extension $m$ of $\ell$ and a normed fiber functor $\alpha$ on $\Rep^\circ \CG$ over $\ell$, we obtain a normed fiber functor $\alpha_\ell$ on $\Rep^\circ \CG$ over $m$ by composing $\alpha$ with the base change functor $\Norm^\circ(\ell) \to \Norm^\circ(m)$. 
  \item Given a normed fiber functor $\alpha \colon \Rep^\circ \CG \to \Norm^\circ(\ell)$, we obtain the fiber functor $$\forg \circ \alpha\colon \Rep^\circ \CG \to \Vec(\ell).$$ We call this the \emph{underlying fiber functor} of $\alpha$.
  \end{enumerate}
\end{definition}

Given a normed fiber functor $\alpha$ with underlying fiber functor $\omega$, for an object $X \in \Rep^\circ \CG$, we will often denoted the norm on $\omega(X)$ given by $\alpha(X)$ by $\alpha_X$. We will also abbreviate the associated submodules $\omega(X)^{\alpha_X \leq \gamma}$ to $\omega(X)^{\alpha\leq\gamma}$.

We fix a normed fiber functor $\alpha$ on $\Rep^\circ \CG$ over $\ell$.

\begin{definition}
  We let $\Gamma_\alpha \subset \Gamma$ be the set of elements $\gamma \in \Gamma$ for which there exists some $X \in \Rep^\circ \CG$ and some non-zero element $x \in \alpha(X)$ satisfying $\alpha_X(x)=\gamma$.
\end{definition}

\begin{lemma}
  The set $\Gamma_\alpha$ is a subgroup of $\Gamma$. It contains $|\ell^*|$ and the quotient $\Gamma_\alpha/|\ell^*|$ is finitely generated.
\end{lemma}
\begin{proof}
  The compatibility of $\alpha$ with tensor products shows that $\Gamma_\alpha$ is closed under multiplication. Similarly the existence of duals in $\Rep^\circ \CG$ show that $\Gamma_\alpha$ is closed under taking inverses. The tensor functor $\alpha$ sends the tensor unit of $\Rep^\circ \CG$ to the tensor unit of $\Norm^\circ(\ell)$, that is it assigns the norm of $\ell$ to the trivial representation of $\CG$. Hence the group $\Gamma_\alpha$ contains $|\ell^*|$. To see that $\Gamma_\alpha/|\ell^*|$ is finitely generated, we take a tensor generator $X$ of $\Rep^\circ \CG$. Then the fact that every $Y \in \Rep^\circ \CG$ is a strict subquotient of a direct sum of tensor products of $X$ and $X^\vee$ implies that $\Gamma_\alpha$ is generated by $|\ell^*|$ and the finitely many norms of a splitting basis of $\alpha(X)$.
\end{proof}

\begin{construction}
  Let $\omega\colon \Rep^\circ \CG \to \Mod_{\ell^\circ}$ be a fiber functor. We associate a normed fiber functor $\alpha_\omega\colon \Rep^\circ \CG \to \Norm^\circ(\ell)$ to $\omega$ as follows: For each $X \in \Rep^\circ \CG$, the lattice $\omega(X)$ in $\omega(X)_\ell$ gives by Construction \ref{LatticeNorm} a norm $\alpha_{\omega(X)}$ on $\omega(X)_\ell$. One checks that this norm is compatible with tensor products in $X$, as well as exact and $k^\circ$-linear in $X$ and hence gives the desired functor $\alpha_\omega$.

\end{construction}
The normed fiber functor $\alpha_\omega$ just constructed satisfies $\Gamma_{\alpha_\omega} \subset |\ell^*|$. All such normed fiber functor arise in this way:
\begin{lemma} \label{SpecialFunctorClass}
  Assume that $\alpha$ satisfies $\Gamma_\alpha \subset |\ell^*|$. Then $\alpha$ is canonically isomorphic to the normed fiber functor $\alpha_\lambda$ associated to the fiber functor $\lambda \colon X \mapsto \omega(X)^{\alpha\leq 1}$ on $\Rep^\circ \CG$ over $\ell^\circ$.
\end{lemma}
\begin{proof}
  This follows from Lemma \ref{SpecialNormClass} by verifying that $\lambda$ is a fiber functor: To verify compatibility with tensor products, consider objects $X, Y \in \Rep^\circ \CG$. Since $\alpha$ is a tensor functor, there is an inclusion $\omega(X)^{\alpha\leq 1} \otimes \omega(Y)^{\alpha\leq 1} \into \omega(X \otimes Y)^{\alpha\leq 1}$ inside $\omega(X \otimes Y)$. To show that this is an equality, let $z \in \omega(X \otimes Y)^{\alpha\leq 1}$. Then we can write $z=\sum_{i=1}^n x_i \otimes y_i$ for elements $x_i \in \omega(X)$ and $y_i \in \omega(Y)$ satisfying $\alpha(x_i)\alpha(y_i) \leq 1$. By our assumption, there exist elements $\lambda_i \in \ell$ satisfying $|\lambda_i|=\alpha(x_i)$. Then the presentation $z=\sum_i (x_i/\lambda_i) \otimes (\lambda_iy_i)$ shows that $z \in \omega(X)^{\alpha\leq 1} \otimes \omega(Y)^{\alpha\leq 1}$.

  The exactness of $\lambda$ is given by Lemma \ref{NormExact}.
\end{proof}

\subsection{Stabilizer Groups} \label{StabSS2}
We continue with the setup from the previous subsection and let $\alpha \colon \Rep^\circ \CG \to \Norm^\circ(\ell)$ be a normed fiber functor. We denote the underlying fiber functor of $\alpha$ by $\omega$.

\begin{definition} \label{SchStabDef}
 For an $\ell^\circ$-algebra $R$, we let $\UAut^\otimes(\alpha)(R)$ be the set of tuples $$(\phi_X^{\gamma})_{X \in \Rep^\circ \CG, \gamma \in \Gamma}$$ satisfying the following conditions:
  \begin{enumerate}
  \item For all $X \in \Rep^\circ \CG$ and $\gamma \in \Gamma$, the object $\phi_X^{\gamma}$ is an $R$-linear automorphism of the $R$-module $\omega(X)^{\alpha\leq \gamma}_R$.
  \item For all morphisms $X \to Y$ in $\Rep^\circ \CG$ and all $\gamma\leq \delta$ in $\Gamma$, the diagram
    \begin{equation*}
      \xymatrix{
        \omega(X)^{\alpha \leq \gamma}_R \ar[r] \ar[d]_{\phi_X^{\gamma}} & \omega(Y)^{\alpha \leq \delta}_R \ar[d]^{\phi_Y^{\delta}} \\
        \omega(X)^{\alpha \leq \gamma}_R \ar[r] & \omega(Y)^{\alpha \leq \delta}_R,
}
    \end{equation*}
in which the horizontal arrows are the base change to $R$ of the composition of the inclusion $\omega(X)^{\alpha \leq \gamma} \into \omega(X)^{\alpha\leq \delta}$ and the homomorphism $\omega(X)^{\alpha\leq  \delta} \to \omega(Y)^{\alpha \leq \delta}$ induced from $X \to Y$, commutes.
\item For all $X \in \Rep^\circ \CG$, all $\lambda \in \ell^*$ and all $\gamma \in \Gamma$, the diagram
    \begin{equation*}
      \xymatrix{
        \omega(X)^{\alpha \leq \gamma}_R \ar[r] \ar[d]_{\phi_X^{\gamma}} & \omega(X)^{\alpha \leq \gamma |\lambda|}_R \ar[d]^{\phi_X^{\gamma |\lambda|}} \\
        \omega(X)^{\alpha \leq \gamma}_R \ar[r] & \omega(X)^{\alpha \leq \gamma |\lambda|}_R,
}
    \end{equation*}
    in which the horizontal arrows are the base change to $R$ of the isomorphism $$\omega(X)^{\alpha \leq \gamma} \to \omega(X)^{\alpha \leq \gamma |\lambda|},\; v\mapsto \lambda v,$$ commutes.
  \item For all $X, Y \in \Rep^\circ \CG$ and all $\gamma,\delta \in \Gamma$, the diagram
    \begin{equation*}
      \xymatrix{
        \omega(X)^{\alpha \leq \gamma}_R \otimes \omega(Y)^{\alpha \leq \delta}_R \ar[r] \ar[d]_{\phi_X^{\gamma} \otimes \phi_Y^{\delta}} & \omega(X\otimes Y)^{\alpha \leq \gamma\delta}_R \ar[d]^{\phi_{X\otimes Y}^{\gamma \delta}} \\
          \omega(X)^{\alpha \leq \gamma}_R \otimes \omega(Y)^{\alpha \leq \delta}_R    \ar[r]   &\omega(X\otimes Y)^{\alpha \leq \gamma\delta}_R,
        }
      \end{equation*}
in which the horizontal arrows are the base change to $R$ of the inclusion $$\omega(X)^{\alpha \leq \gamma} \otimes \omega(Y)^{\alpha \leq \delta} \into \omega(X\otimes Y)^{\alpha \leq \gamma\delta}$$ inside $\omega(X\otimes Y)$, commutes.
    \end{enumerate}
Under termwise composition, the set $\UAut^\otimes(\alpha)(R)$ is naturally a group. For any homomorphism $R \to S$ of $\ell^\circ$-algebras, there is a natural base change homomorphism $\UAut^\otimes(\alpha)(R) \to \UAut^\otimes(\alpha)(S)$. So we obtain a group presheaf $\UAut^\otimes(\alpha)$ over $\ell^\circ$.
 \end{definition}

 \begin{remark}
   By construction, for each $X \in \Rep^\circ \CG$ there is natural homomorphism $$\UAut^\otimes(\alpha) \to \UAut(\alpha_X), \; (\phi_Y^{\gamma})_{Y \in \Rep^\circ \CG, \gamma \in \Gamma} \mapsto (\phi_X^{\gamma})_{\gamma \in \Gamma}$$ via which specifying a section of $\UAut^\otimes(\alpha)$ is equivalent to specifying sections of all $\UAut(\alpha_{X})$ which are functorial in $X$ as well as compatible with tensor products. 
 \end{remark}

 \begin{proposition} \label{UAutRepr}
   \begin{enumerate}
   \item The presheaf $\UAut^\otimes(\alpha)$ is a sheaf for the fpqc-topology.
   \item  Let $R$ be an $\ell^\circ$-algebra such that for all $X \in \Rep^\circ \CG$ and all $\gamma \in \Gamma$ the $R$-module $\omega(X)^{\alpha\leq \gamma}_R$ is locally free of finite type. Then the group scheme $\UAut^\otimes(\alpha)_R$ is representable by an affine group scheme of finite type over $R$. 
   \item For any tensor generator $X \in \Rep^\circ \CG$, the induced homomorphism $$\UAut^\otimes(\alpha) \to \UAut(\alpha_{X})$$ is representable by a closed immersion.
   \end{enumerate}
\end{proposition}

By Lemma \ref{UAutReprConc}, in particular (ii) applies for $R=\ell^\circ$ in case the valuation on $\ell$ is discrete and $\Gamma$ has rank one, and always for $R=\tilde \ell$ or $R=\ell$.

\begin{proof}
  Claim (i) follows from descent theory. By Proposition \ref{UAutRepr0}, claim (iii) implies claim (ii).

  To prove (iii), let $X \in \Rep^\circ \CG$ be a tensor generator. Let $Y=\oplus_{i=1}^n X^{\otimes r_i} \otimes (X^*)^{\otimes s_i}$ for some integers $n$, $r_i$ and $s_i$ and consider a strict monomorphism $Z \into Y$ in $\Rep^\circ \CG$. Let $\CH_{Z \into Y}$ be the subfunctor of $\UAut(\alpha_{X})$ whose $R$-valued points for some $\ell^\circ$-algebra $R$ are those elements of $\UAut(\alpha_{X})(R)$ whose action on each $\omega(Y)^{\alpha \leq \gamma}_R$ maps the submodule $\omega(Z)^{\alpha\leq \gamma}_R$ to itself. On can check that the inclusion $\CH_{Z \into Y} \into \UAut(\alpha_X)$ is representable by a closed immersion.

  Next, for any two such strict monomorphisms $Z \into Y_1$ and $Z \into Y_2$, let $\CH_{Z \into Y_1,Z \into Y_2}$ be the subsheaf of the intersection $\CH_{Z \into Y_1} \cap \CH_{Z \into Y_2}$ inside $\UAut(\alpha_X)$ whose $R$-valued points are those whose two actions on $\omega(Z)^{\alpha \leq \gamma}_R$ coincide for all $\gamma \in \Gamma$. Again one can check that $\CH_{Z \into Y_1,Z \into Y_2} \into \UAut(\alpha_X)$ is representable by a closed immersion.

By intersecting these subsheaves over all such pairs of inclusions we obtain a subsheaf $\CH$ of $\UAut(\alpha_X)$ whose $R$-valued sections have a canonical action on $\omega(Z)^{\alpha\leq\gamma}_R$ for all $Z \in \Rep^\circ \CG$ which admit such a strict monomorphism $Z \into Y$ and all $\gamma \in \Gamma$. Via Lemma \ref{NormExact} this in turn gives a canonical action of $\CH$ on $\omega(Z)^{\alpha\leq\gamma}$ for all strict subquotients of such $Y$ and all $\gamma \in \Gamma$. Since $X$ is a tensor generator, this applies to all $Z \in \Rep^\circ \CG$. Finally, one checks that the condition that these actions are compatible with tensor products in $Z$ again gives a subsheaf $\CH' \subset \CH$ which is representable by a closed immersion. Then $\CH' = \UAut^\otimes(\alpha)$, which proves (iii).
\end{proof}

Lemma \ref{GenUAutDesc} implies:
\begin{lemma} \label{GenUAutDesc2}
   Assume that for each element $\gamma \in \Gamma$ there exists an element $x \in \ell^*$ satisfying $|x| \leq \gamma$. Then for varying $\gamma\in \Gamma$ and $X \in \Rep^\circ \CG$, the inclusions $\omega(X)^{\alpha\leq \gamma} \into \omega(X)$ induce an isomorphism $\UAut^\otimes(\alpha)_\ell \cong \UAut^\otimes(\omega)$.
\end{lemma}

\begin{definition} \label{UEndDef}
  For an $\ell^\circ$-algebra $R$, we let $\UEnd^\otimes(\alpha)(R)$ be the set of tuples 
  \begin{equation*}
    (h_X)_{X \in \Rep^\circ \CG} \in \prod_{X\in\Rep^\circ \CG} \End(\alpha_X)_R
  \end{equation*}
 satisfying the following conditions:    
  \begin{enumerate}
    \item They are functorial in the sense that for all morphisms $X \to Y$ in $\Rep^\circ \CG$, the images of $h_X$ and $h_Y$ in $\Hom(\alpha_{X},\alpha_{Y})_R$ coincide.
    \item They are compatible with tensor products in the sense that for all $X, Y \in \Rep^\circ \CG$, the image of $h_X \otimes \id + \id\otimes h_Y $ under the natural homomorphism $$\End(\alpha_X)_R \otimes_R \End(\alpha_Y)_R \to \End(\alpha_X \otimes \alpha_Y)_R=\End(\alpha_{X\otimes Y})_R$$ is equal to $h_{X\otimes Y}$. 
  \end{enumerate}

  The set $\UEnd^\otimes(\alpha)(R)$ naturally forms an abelian group under termwise addition. This construction is functorial in $R$, and so we obtain a presheaf of abelian groups $\UEnd^\otimes(\alpha)$ over $\ell^\circ$.

The sheaf $\UAut^\otimes(\alpha)$ acts naturally on $\UEnd(\alpha)$ via termwise conjugation via \eqref{AutCompMor}.
\end{definition}

\begin{example} \label{SpSchAut}
  Assume that $\Gamma_\alpha \subset |\ell^*|$. Then $\alpha=\alpha_{\lambda}$ for the canonical fiber functor $\lambda\colon X \mapsto \omega(X)^{\alpha\leq 1}$ on $\Rep^\circ \CG$ over $\ell^\circ$ by Lemma \ref{SpecialFunctorClass}. Let $(\phi^\gamma_X)_{X,\gamma} \in \UAut^\otimes(\alpha)(R)$ for some $\ell^\circ$-algebra $R$. Then by condition (ii) in Definition \ref{SchStabDef}, for any $X \in \Rep^\circ \CG$, the automorphisms $\phi_X^\gamma$ for varying $\gamma$ are all determined by $\phi_X^1$. The $\phi_X^1$ for varying $X$ give an element of $\UAut^\otimes(\lambda)(R)$. This gives a canonical isomorphism
  \begin{equation*}
    \UAut^\otimes(\alpha)\cong \UAut^\otimes(\lambda).
  \end{equation*}

  Analogously we find a canonical isomorphism
  \begin{equation*}
    \UEnd^\otimes(\alpha) \cong \UEnd^\otimes(\lambda).
  \end{equation*}
\end{example}

\subsection{Splittings} \label{SplittingSS2}
We continue with the setup from the previous subsection.

\begin{definition}
  Let $\omega\colon \Rep^\circ \CG \to \Mod_\ell$ be a fiber functor. An \emph{integral model} of $\omega$ is a fiber functor $\lambda\colon \Rep^\circ \CG \to \Mod_{\ell^\circ}$ together with an isomorphism $\psi\colon \lambda_\ell \isoto \omega$. We will often omit $\psi$ and simply talk about an integral model $\lambda$.

  Given such a $\lambda$ and an element $g \in \UAut^\otimes(\omega)(\ell)$, we let $g \cdot \lambda$ be the integral model of $\omega$ obtained by replacing $\psi$ by the isomorphism $\lambda_\ell \toover{\psi} \omega \toover{g} \omega$.
\end{definition}
\begin{construction} \label{ChiNormCons}
  Let $\lambda$ be a fiber functor on $\Rep^\circ \CG$ over $\ell^\circ$ and $\chi\colon D^\Gamma_{\ell^\circ} \to \UAut^\otimes(\lambda)$ a homomorphism. We construct a normed fiber functor $\alpha_{\lambda,\chi}$ with underlying fiber functor $\lambda_\ell$ as follows: 

  For $X \in \Rep^\circ \CG$, Construction \ref{ChiNormCons0} applied to the $\ell^\circ$-lattice $\lambda(X)$ in $\lambda(X)_\ell$ together with the homomorphism $\chi_X\colon D^\Gamma_{\ell^\circ} \toover{\chi} \UAut^\otimes(\lambda) \to GL(\lambda(X))$ gives a norm $\alpha_{\lambda(X),\chi_X}$ on $\lambda(X)_\ell$.

These norms are functorial and exact in $X$ as well as compatible with tensor products and hence define a normed fiber functor $\alpha_{\lambda,\chi}\colon \Rep^\circ \CG \to \Norm^\circ(\ell)$ which comes with a canonical isomorphism $\forg\circ\alpha_{\lambda,\chi}\cong \lambda_\ell$.
\end{construction}

\begin{definition} \label{SplittingDef}
  \begin{enumerate}
  \item  A pair $(\lambda,\chi)$ consisting of an integral model $\lambda$ of $\omega$ and a homomorphism $\chi\colon D^\Gamma_{\ell^\circ} \to \UAut^\otimes(\lambda)$ \emph{splits} $\alpha$, if the included isomorphism $\lambda_\ell \cong \omega$ identifies $\alpha$ with the normed fiber functor $\alpha_{\lambda,\chi}$ from Construction \ref{ChiNormCons}.
  \item The normed fiber functor $\alpha$ is splittable if there exists such a pair which splits $\alpha$.
  \end{enumerate}
\end{definition}

We also a give the following variant of this notion which doesn't involve the fiber functor $\lambda$:

\begin{definition} \label{IntSplitDef}
  Let $R$ be a $\ell^\circ$-algebra and consider a homomorphism $\chi\colon D^\Gamma_R \to \UAut^\otimes(\alpha)_R$ or $\chi\colon D^{\Gamma/|\ell^*|}_R \to \UAut^\otimes(\alpha)_R$. For each $X \in \Rep^\circ \CG$ and each $\gamma \in \Gamma$, the group sheaf $\UAut^\otimes(\alpha)_R$ acts naturally on $\omega(X)^{\alpha \leq \gamma}_R/ \omega(X)^{\alpha < \gamma}_R$. We say that $\chi$ \emph{splits} $\alpha$ over $R$ if for all such $X$ and $\gamma$, all weights of $\chi$ on $\omega(X)^{\alpha \leq \gamma}_R/ \omega(X)^{\alpha < \gamma}_R$ are in $\gamma |\ell^*|$.
\end{definition}

\begin{proposition} \label{SplitExtEquiv}
  For a homomorphism $\chi\colon D^\Gamma_\ell \to \UAut^\otimes(\omega)$ the following are equivalent:
  \begin{enumerate}
  \item There exists an integral model $\lambda$ of $\omega$ together with an extension $\chi_1\colon D^\Gamma_{\ell^\circ} \to \UAut^\otimes(\lambda)$ of $\chi$ such that the pair $(\lambda,\chi_1)$ splits $\alpha$.
  \item The homomorphism $\chi$ extends to a homomorphism $\chi_2\colon D^\Gamma_{\ell^\circ} \to \UAut^\otimes(\alpha)$ which splits $\alpha$ over $\ell^\circ$. 
  \end{enumerate}
If this is the case, then such objects $\chi_1$, $\chi_2$ and $\lambda$ are uniquely determined by the above.
\end{proposition}
\begin{proof}
  This follows by the same argument as Proposition \ref{SplitExtEquiv0}, by checking in addition that the constructions there are functorial and compatible with tensor products.
\end{proof}

For later use we note the following facts:

\begin{lemma} \label{SplittingUnique2}
 Let $R$ be a $\ell^\circ$-algebra such that every connected component $C$ of $\Spec(R)$ has non-empty special fiber $C_{\tilde \ell}$ and $\CT \subset \UAut^\otimes(\alpha)_R$ a split closed subtorus.
  \begin{enumerate}
  \item There exists at most one homomorphism $\chi\colon D^{\Gamma/|\ell^*|}_R \to \CT$ which splits $\alpha$ over $R$.
  \item Let $S$ be a faithfully flat $R$-algebra. Any homomorphism $\chi\colon D^{\Gamma/|\ell^*|}_S \to \CT_S$ which splits $\alpha$ over $S$ descends to a homomorphism $D^{\Gamma/|\ell^*|}_R \to \CT$ which splits $\alpha$ over $R$.
  \end{enumerate}
\end{lemma}
\begin{proof}
  (i) Using a tensor generator of $\Rep^\circ \CG$ and Proposition \ref{UAutRepr}, one reduces to Lemma \ref{SplittingUnique0}.

  (ii) By (i), the two pullbacks of $\chi$ to $S\otimes_R S$ are equal. Hence the claim follows using flat descent.
\end{proof}

\begin{lemma} \label{SplitLift1}
  Let $T \subset \UAut^\otimes(\alpha)_{\tilde\ell}$ be a subtorus and $\chi\colon D^{\Gamma/|\ell^*|}_{\tilde \ell} \to T \subset \UAut^\otimes(\alpha)_{\tilde\ell}$ a homomorphism which splits $\alpha$ over $\tilde\ell$.
  \begin{enumerate}
  \item The homomorphism $\chi$ factors through the maximal split subtorus of $T$.
  \item The homomorphism $\chi$ can be lifted to a homomorphism $D^\Gamma_{\tilde\ell} \to T$ which splits $\alpha$ over $\tilde\ell$.
  \end{enumerate}
\end{lemma}

\begin{proof}
  (i) Let $X$ be a tensor generator of $\Rep^\circ \CG$. By Proposition \ref{UAutRepr} the induced homomorphism $i_X\colon \UAut^\otimes(\alpha)_{\tilde\ell} \to \UAut(\alpha_X)_{\tilde\ell}$ is a closed immersion. So Lemma \ref{SplitLift} applied to $\alpha_X$ shows (i).

  (ii) By (i) we may replace $T$ by its maximal split subtorus. Then the surjectivity of
  \begin{equation*}
    X_*(T) \otimes \Gamma \to X_*(T) \otimes \Gamma/|\ell^*|
  \end{equation*}
  shows that $\chi$ can be lifted to a homomorphism $D^{\Gamma}_{\tilde\ell} \to T$. Then by Definition \ref{IntSplitDef} any such lift splits $\chi$ over $\tilde\ell$.
\end{proof}

\subsection{Base Change of Normed Fiber Functors} \label{BCSS2}
We continue with the setup from the previous subsection and study the bevaviour of $\UAut^\otimes(\alpha)$ under base change. For this let $m$ be a non-archimedean extension of $\ell$.
\begin{definition}
  We let $\alpha_m$ be normed fiber functor on $\Rep^\circ \CG$ over $m$ given by the composition $\Rep^\circ \CG \to \Norm^\circ(\ell) \to \Norm^\circ(m)$ of $\alpha$ with the base change functor $\Norm^\circ(\ell)\to \Norm^\circ(m)$ from Lemma \ref{NormBC}.
\end{definition}

\begin{construction}    \label{UAutBCC}
  By Construction \ref{UAutBCC0}, for each $X \in \Rep^\circ \CG$ there is a natural base change homomorphism $\UAut(\alpha_{X})_{m^\circ} \to \UAut(\alpha_{X,m})$. One checks that for a $m^\circ$-algebra $R$ and for an element $\phi=(\phi^{\gamma}_X)_{\gamma,X} \in \UAut^\otimes(\alpha)(R)$, the images of the sections $(\phi^{\gamma}_X)_{\gamma} \in \UAut(\alpha_{X})(R)$ under these base change homomorphisms fit together to a section $b^{m/\ell}(\phi) \in \UAut^\otimes(\alpha_{m})(R)$. This gives a base change homomorphism
  \begin{equation*}
    b^{m/\ell}\colon \UAut^\otimes(\alpha)_{\tilde m} \to \UAut^\otimes(\alpha_{m})
  \end{equation*}
  over $m^\circ$.
\end{construction}

\begin{lemma} \label{BCProps}
 Assume that for each element $\gamma \in \Gamma$ there exists an element $x \in \ell^*$ satisfying $|x| \leq \gamma$. Then the generic fiber $b^{m/\ell}_m$ of the base change homomorphism from Construction \ref{UAutBCC} is an isomorphism.
\end{lemma}
\begin{proof}
  One readily checks that under the canonical isomorphisms $\UAut^\otimes(\alpha)_m \cong \UAut^\otimes(\omega)_m \cong \UAut^\otimes(\alpha_{m})$ from Lemma \ref{GenUAutDesc2}, the homomorphism $b^{m/\ell}_m$ is the identity on $\UAut^\otimes(\omega)_m$.
\end{proof}

Lemma \ref{UnrBC0} implies:
\begin{lemma}\label{UnrBC}
  If $|m^*|=|\ell^*|$, then $b^{m/\ell}\colon \UAut^\otimes(\alpha)_{m^\circ} \to \UAut^\otimes(\alpha_m)$ is an isomorphism.
\end{lemma}

Now we assume that $|m^*|=\Gamma$.
\begin{construction} \label{ChiRamCons}
  For each $X \in \Rep^\circ \CG$, Construction \ref{ChiCons} gives us a homomorphism $$\chi^{m/\ell} \colon D^{\Gamma/|\ell^*|}_{\tilde m} \to \UAut^\otimes(\alpha_{X,m})_{\tilde m}.$$ Lemma \ref{BCGrProps} implies that these homomorphisms fit together to a homomorphism
  \begin{equation*}
    \chi^{m/\ell} \colon D^{\Gamma/|\ell^*|}_{\tilde m} \to \UAut^\otimes(\alpha_m)_{\tilde m}.
  \end{equation*}

We denote by $L^{m/\ell} \subset \UAut^\otimes(\alpha_m)_{\tilde m}$ the centralizer of this homomorphism.

\end{construction}

Proposition \ref{DilProp1} implies:
\begin{proposition} \label{BCDilat}
  \begin{enumerate}
  \item The base change homomorphism $$b^{m/\ell}_{\tilde m}\colon \UAut^\otimes(\alpha)_{\tilde m} \to \UAut^\otimes(\alpha_{m})_{\tilde m}$$ factors through $L^{m/\ell}$. 
  \item  Assume that $|\ell^*|$ is dense in $|m^*|$. Let $R$ be a flat $m^\circ$-algebra and $g \in \UAut^\otimes(\alpha_{m})(R)$. Then $g$ is in the image of $b^{m/\ell}\colon \UAut^\otimes(\alpha)(R) \to \UAut^\otimes(\alpha_{m})(R)$ if and only if the image of $g$ in $\UAut^\otimes(\alpha_{m})(R \otimes_{m^\circ} \tilde m)$ lies in $L^{\ell/m}(R \otimes_{m^\circ} \tilde m)$. 
  \end{enumerate}

\end{proposition}

Proposition \ref{SplBCComp} implies:
\begin{proposition} \label{SplBCComp2}
  \begin{enumerate}
  \item A homomorphism $\chi\colon D^{\Gamma/|\ell^*|}_{\tilde m} \to \UAut^\otimes(\alpha)_{\tilde m}$ splits $\alpha$ over $\tilde m$ if and only if $b^{m/\ell}_{\tilde m}\circ  \chi=\chi^{m/\ell}$.
  \item If a homomorphism $\chi\colon D^{\Gamma/|\ell^*|} \to \UAut^\otimes(\alpha)$ splits $\alpha$ over $\ell^\circ$, then $b^{m/\ell}$ restricts to an isomorphism
    \begin{equation*}
      {\Cent}_{\UAut(\alpha)_{m^\circ}}(\chi_{m^\circ}) \isoto \Cent_{\UAut^\otimes(\alpha_m)}(b^{m/\ell} \circ \chi_{m^\circ}).
    \end{equation*}
  \end{enumerate}
\end{proposition}

\subsection{The subgroups $U_\delta(\alpha) \subset \UAut^\otimes(\alpha)_{\tilde\ell}$} \label{USS2}
We continue with the setup from the previous subsection.

\begin{lemma} \label{UdEquiv}
  Let $\phi \in \UAut^\otimes(\alpha)(R)$ for some $\tilde\ell$-algebra $R$ and let $\delta \in \Gamma^{\leq 1}$. The following are equivalent:
  \begin{enumerate}
  \item For all $X \in \Rep^\circ \CG$, the image $\phi_X \in \UAut(\alpha_X)(R)$ of $\phi$ lies in $U_{\leq \delta}(\alpha_X)(R)$.
  \item For some tensor generator $X \in \Rep^\circ \CG$, the element $\phi_X$ lies in $U_{\leq \delta}(\alpha_X)(R)$.
  \end{enumerate}
\end{lemma}
\begin{proof}
  Only the direction $(ii) \implies (i)$ requires an argument. Let $X$ be as in (ii) and $Z \in \Rep^\circ \CG$. Then $Z$ is a strict subquotient of an object $Y=\oplus_{i=1}^n X^{\otimes r_i}\otimes (X^*)^{\otimes s_i}$ of $\Rep^\circ \CG$. Using Lemma \ref{FiltExact1} it follows that $\phi_Z \in U_{\leq \delta}(\alpha_Z)(R)$.
\end{proof}

\begin{definition}
  For $\delta \in \Gamma^{\leq 1}$, we let $U_{\leq \delta}(\alpha) \subset \UAut^\otimes(\alpha)_{\tilde \ell}$ be the sub-presheaf consisting of those sections which satisfy the equivalent conditions of Lemma \ref{UdEquiv}.
\end{definition}

Then Lemma \ref{UdEquiv} implies:
\begin{lemma} \label{UdInt}
  Let $X \in \Rep^\circ \CG$ be a tensor generator. Under the induced closed immersion $\UAut^\otimes(\alpha)_{\tilde\ell} \to \UAut(\alpha_X)_{\tilde \ell}$ (c.f. Proposition \ref{UAutRepr} (iii)), the presheaf $U_{\leq \delta}(\alpha)$ is equal to the intersection of $U_{\leq \delta}(\alpha_X)$ and $\UAut^\otimes(\alpha)_{\tilde \ell}$ inside $\UAut(\alpha_X)_{\tilde \ell}$.
\end{lemma}

Together with Proposition \ref{UdRepr1} this implies:
\begin{proposition} \label{UdRepr2}
  For $\delta \in \Gamma^{\leq 1}$, the sub-presheaf $U_{\leq \delta}(\alpha)$ of $\UAut^\otimes(\alpha)_{\tilde \ell}$ is representable by a closed subscheme.
\end{proposition}

\begin{definition}
  For $\delta \in \Gamma^{\leq 1}$, we let $U_{<\delta}(\alpha) \defeq \colim_{\mu < \delta}U_{\leq \mu}(\alpha) \subset \UAut^\otimes(\alpha)_{\tilde \ell}$. 
\end{definition}

Corollary \ref{UaSplit} and Lemma \ref{UdInt} imply:
\begin{lemma}
  For each $\delta \in \Gamma^{\leq 1}$, there exists $\mu \in \Gamma^{<\delta}$ for which $U_{<\delta}(\alpha)=U_{\leq \mu}(\alpha)$. In particular $U_{<\delta}(\alpha)$ is representable by a closed subscheme of $\UAut^\otimes(\alpha)_{\tilde\ell}$.
\end{lemma}
Lemma \ref{UComm1} implies:
\begin{lemma} \label{UComm2}
  Let $\delta, \delta' \in \Gamma^{\leq 1}$.
  \begin{enumerate}
  \item  The commutator of $U_{\leq \delta}(\alpha)$ and $U_{\leq \delta'}(\alpha)$ is contained in $U_{\leq \delta\delta'}(\alpha)$.
  \item  The commutator of $U_{\leq \delta}(\alpha)$ and $U_{< \delta'}(\alpha)$ is contained in $U_{< \delta\delta'}(\alpha)$.
  \end{enumerate}
\end{lemma}

\begin{definition}
  Let $\delta \in \Gamma^{\leq 1}$. By Lemma \ref{UComm2}, the group scheme $U_{<\delta}(\alpha)$ is normal in $U_{\leq \delta}(\alpha)$. Hence we may consider the quotient scheme
  \begin{equation*}
    \gr_{\delta} \UAut^\otimes(\alpha)_{\tilde\ell} \defeq U_{\leq \delta}(\alpha)/ U_{<\delta}(\alpha).
  \end{equation*}
\end{definition}

\begin{lemma} \label{GrdCU}
  For $\delta \in \Gamma^{<1}$, the group scheme $\gr_{\delta}\UAut^\otimes(\alpha)_{\tilde \ell}$ is unipotent and commutative.
\end{lemma}
\begin{proof}
   By Lemma \ref{UdInt}, a tensor generator $X \in \Rep^\circ \CG$ induces a closed immersion
  \begin{equation*}
    \gr_{\delta}\UAut^\otimes(\alpha)_{\tilde \ell} \into \gr_{\delta} \UAut^\otimes(\alpha_X)_{\tilde \ell}.
  \end{equation*}
  Since $\gr_{\delta} \UAut^\otimes(\alpha_X)_{\tilde \ell}$ is a vector group scheme by Proposition \ref{GrdDesc1}, this implies the claim.
\end{proof}

We also note that Lemma \ref{cl} implies:
\begin{lemma} \label{UdeltaTriv}
  For $\delta \in \Gamma^{\leq 1}$, the group scheme $U_{\leq\delta}(\alpha)$ is trivial unless $\delta \in I_\ell$.
\end{lemma}

Lemma \ref{SplitConj} implies:
\begin{lemma} \label{SplitConj1}
  Let $R$ be a $\tilde \ell$-algebra. For any two homomorphisms $$\chi,\chi'\colon D^{\Gamma/|\ell^*|}_R \to \UAut^\otimes(\alpha)_R$$ which split $\alpha$ over $R$, there exists a unique section $u \in U_{<1}(\alpha)(R)$ for which $\chi'=\leftexp{u}{\chi}$.
\end{lemma}

Let now $m$ be a non-archimedean extension of $\ell$ for which $|m^*|=\Gamma$. Then Constructions \ref{UAutBCC} and \ref{ChiRamCons} give us homomorphisms
\begin{equation*}
  b^{m/\ell}_{\tilde m}\colon \UAut^\otimes(\alpha)_{\tilde m} \to \UAut^\otimes(\alpha_m)_{\tilde m}
\end{equation*}
and
\begin{equation*}
  \chi^{m/\ell} \colon D^{\Gamma/|\ell^*|}_{\tilde m} \to \UAut^\otimes(\alpha_m)_{\tilde m}.
\end{equation*}
We consider the relationship of the group schemes $U_{\leq\delta}(\alpha)$ with these homomorphisms:
\begin{lemma} \label{UBKer}
  The homomorphism $b^{m/\ell}_{\tilde m} \colon \UAut^\otimes(\alpha)_{\tilde m} \to \UAut^\otimes(\alpha_m)_{\tilde m}$ has kernel $U_{<1}(\alpha)_{\tilde m}$.
\end{lemma}
\begin{proof}
  From the way we defined $b^{m/\ell}$ using \eqref{BCPresCons}, we see that a section $\phi \in \UAut^\otimes(\alpha)(R)$ over some $\tilde m$-algebra $R$ is in $\ker(b^{m/\ell})$ if and only if for any $\gamma \in \Gamma$, any $X \in \Rep^\circ \CG$ and any $x \in \omega(X)^{\alpha\leq \gamma}_R$, the element $\phi(x)-x$ maps to zero in $\omega(X)^{\alpha_m \leq \gamma}_R$. So it suffices to show that $\ker(V^{\alpha\leq\gamma}_{\tilde m} \to V^{\alpha_m \leq \gamma}_{\tilde m})=F^{<1}(V^{\alpha\leq\gamma}_{\tilde m})$ for any $(V,\alpha) \in \Norm^\circ(\ell)$. This follows from Lemma \ref{VKer}.
\end{proof}

\begin{construction}
  The homomorphism $\chi^{m/\ell}\colon D^{\Gamma/|\ell^*|}_{\tilde m} \to \UAut^\otimes(\alpha_m)_{\tilde m}$ from Construction \ref{ChiRamCons} together with the action of $\UAut^\otimes(\alpha_m)_{\tilde m}$ on $\UEnd^\otimes(\alpha_m)_{\tilde m}$ by conjugation induce a weight decomposition of $\UEnd^\otimes(\alpha_m)_{\tilde m}$ which we write as
  \begin{equation}\label{TWeightDec}
    \UEnd^\otimes(\alpha_m)_{\tilde m} \cong \oplus_{[\mu]\in \Gamma/|\ell^*|} \gr^{m/\ell}_{[\mu]}(\UEnd^\otimes(\alpha_m)_{\tilde m}).
  \end{equation}

  From the various definitions, one sees that this decomposition can be equivalently described by saying that for any $\tilde m$-algebra $R$, a section $(h_X)_X \in \UEnd^\otimes(\alpha_m)_{\tilde m}(R)$ lies in the submodule $\gr^{m/\ell}_{[\mu]}(\UEnd^\otimes(\alpha_m)_{\tilde m})(R)$ if and only if for all $X \in \Rep^\circ \CG$, the element $h_X \in \End(\alpha_X)_R$ lies in the summand $\gr^{m/\ell}_{[\mu]}(\End(\alpha_X)_R)$ from Definition \ref{BCDecDef}.
\end{construction}

\begin{construction} \label{dCons}
  Let $\delta \in I_\ell \cap \Gamma^{<1}$ and let $x \in m^{\circ\circ}$ be an element of norm $\delta$.

  Consider a $\tilde m$-algebra $R$, an element $u \in U_{\leq\delta}(\alpha)(R)$ and an object $X \in \Rep^\circ \CG$.  The homomorphism \eqref{CompComp} associated to our chosen $x$, the normed vector space $V=\alpha(X \otimes X^*)$ and $\gamma=1$ gives a surjection
\begin{equation} \label{EndComp}
  F^{\leq\delta} (\End(\alpha_X)_R) \onto \gr_{[\delta]}^{m/\ell}(\End(\alpha_{X,m})_{\tilde m})_R
\end{equation}
with kernel $F^{<\delta}( \End(\alpha_X)_R)$.

Let $u_X$ be the image of $u$ in $\UAut(\alpha_X)(R)$ and $\tilde u_X \in \End(\alpha_X)^*_R$ the element corresponding to $u_X$ under \eqref{AutCompMor}. Then by the definition of $U_{\leq\delta}(\alpha)$, the element $\tilde u_X - \id_{\omega(X)}$ is contained in $F^{\leq\delta}(\End(\alpha_X)_R)$. We denote by $d(u)_X \in \gr_{[\delta]}^{m/\ell}(\End(\alpha_{X,m})_{\tilde m})_R$ its image under \eqref{EndComp}.

If $Y$ is a second object of $\Rep^\circ \CG$, then the identity 
\begin{multline*}
  \tilde u_{X\otimes Y} - \id_{\omega(X\otimes Y)} = \\(\tilde u_X - \id_{\omega(X)}) \otimes \id_{\omega(Y)} + \id_{\omega(X)} \otimes (\tilde u_Y - \id_{\omega(Y)}) + (\tilde u_X -\id_{\omega(X)}) \otimes (\tilde u_Y - \id_{\omega(Y)})
\end{multline*}
holds in $\End(\alpha_{X\otimes Y})_R$. Since $(\tilde u_X -\id_{\omega(X)}) \otimes (\tilde u_Y - \id_{\omega(Y)}) \in F^{\leq \delta^2}(\End(\alpha_{X\otimes Y})_R)$ lies in the kernel of the homomorphism \eqref{EndComp} associated to $X \otimes Y$, this shows that the $d(u)_X$ for varying $X$ are compatible with tensor products in the sense of Definition \ref{UEndDef}. So they define an element
\begin{equation*}
  d(u) \in \gr_{[\delta]}^{m/\ell}(\UEnd^\otimes(\alpha_m)_{\tilde m})(R).
\end{equation*}
The assignment $u \mapsto d(u)$ is functorial in $R$ and so we obtain a morphism
\begin{equation}
  \label{dMor}
  d \colon U_{\leq\delta}(\alpha)_{\tilde m} \to \gr_{[\delta]}^{m/\ell} (\UEnd^\otimes(\alpha_m)_{\tilde m})
\end{equation}
of presheaves over $\tilde m$.

\end{construction}

\begin{lemma} \label{dProps}
  In the situation of Construction \ref{dCons} the following hold:
  \begin{enumerate}
  \item The morphism $d$ is a homomorphism with respect to the additive group structure on its target.
  \item The kernel of $d$ is equal to $U_{<\delta}(\alpha)_{\tilde m}$.
  \end{enumerate}
\end{lemma}
\begin{proof}
  (i) With the notation from Construction \ref{dCons}, for elements $u,u' \in U_{\leq\delta}(\alpha)(R)$ we find
  \begin{equation*}
    \tilde u_X \circ \tilde u'_X-\id_{\omega(X)}=(\tilde u_X-\id_{\omega(X)}) + (\tilde u'_X -\id_{\omega(X)}) + (\tilde u_X - \id_{\omega(X)}) \circ (\tilde u'_X - \id_{\omega(X)})
  \end{equation*}
  in $\End(\alpha_X)_R$. Since $(\tilde u_X - \id_{\omega(X)}) \circ (\tilde u'_X - \id_{\omega(X)}) \in F^{\leq\delta^2}(\End(\alpha_X)_R)$ lies in the kernel of \eqref{EndComp}, this implies (i).

  (ii) follows from the fact that the kernel of \eqref{EndComp} is equal to $F^{<\delta}(\End(\alpha_X)_R)$.
\end{proof}
\subsection{Main result} \label{MainThmSS}
Now we assume that $\ell$ is discretely valued and that $\Gamma$ is of rank one. The following is our main result about normed fiber functors $\alpha$ over such a field and the associated group schemes $\UAut^\otimes(\alpha)$.

\begin{theorem} \label{MainTheorem}
  Let $\CG$ be a smooth affine group scheme over $\ell^\circ$ and $\alpha\colon \Rep^\circ \CG \to \Norm^\circ(\ell)$ a normed fiber functor.
  \begin{enumerate}
 
  \item The sheaf $\UAut^\otimes(\alpha)$ is representable by a smooth affine group scheme over $\ell^\circ$.
  \item For each $\delta<1$, the group scheme $\gr_\delta \UAut^\otimes(\alpha)_{\tilde \ell}$ is a vector group scheme.
  \item For each $\delta < 1$, the group schemes $U_{\leq \delta}(\alpha)$ and $U_{<\delta}(\alpha)$ are split unipotent group schemes.
  \item There exists a homomorphism $\chi\colon D^{\Gamma}_{\tilde\ell} \to \UAut^\otimes(\alpha)_{\tilde \ell}$ which splits $\alpha$ over $\tilde\ell$.
  \item For every homomorphism $\chi\colon D^{\Gamma/|\ell^*|}_{\tilde\ell} \to \UAut^\otimes(\alpha)_{\tilde\ell}$ splitting $\alpha$ over $\tilde \ell$, the centralizer $\Cent_{\UAut^\otimes(\alpha)_{\tilde\ell}}(\chi)$ of $\chi$ gives a semidirect product decomposition $$\UAut^\otimes(\alpha)_{\tilde \ell}=U_{<1}(\alpha) \rtimes \Cent_{\UAut^\otimes(\alpha)_{\tilde\ell}}(\chi).$$
  \item The special fiber $\UAut^\otimes(\alpha)_{\tilde\ell}$ has the same rank as $\CG_{\tilde\ell}$.
  \end{enumerate}
  If the non-archimedean field $\ell$ is in addition Henselian, then the following hold:
  \begin{enumerate}[resume]
  \item The normed fiber functor $\alpha$ is splittable.
      \item For every residually maximal split torus $\CT$ of $\UAut^\otimes(\alpha)$, there exists a unique homomorphism $\chi\colon D^{\Gamma/|\ell^*|}_{\ell^\circ} \to \CT \subset \UAut^\otimes(\alpha)$ which splits $\alpha$ over $\ell^\circ$. 
  \item Any two homomorphisms $\chi\colon D^{\Gamma/|\ell^*|}_{\ell^\circ} \to \UAut^\otimes(\alpha)$ which split $\alpha$ over $\ell^\circ$ are conjugate by an element of $\UAut^\otimes(\alpha)(\ell^\circ)$.

  \end{enumerate}
\end{theorem}
The proof of Theorem \ref{MainTheorem} takes up the rest of this subsection.

By Proposition \ref{UAutRepr} and Lemma \ref{UAutReprConc}, the assumption on $\ell$ ensures that the functor $\UAut^\otimes(\alpha)$ is representable by an affine group scheme of finite type over $\ell^\circ$. Let $\ell \subset m$ be a strictly Henselian non-archimedean extension for which $|m^*|=\Gamma$. Then by Lemma \ref{SpecialFunctorClass} there exists a canonical fiber functor $\lambda$ on $\Rep^\circ \CG$ over $m^\circ$ for which $\alpha_m=\alpha_{\lambda}$. Example \ref{SpSchAut} gives us a canonical isomorphism $\UAut^\otimes(\alpha_m) \cong \UAut^\otimes(\lambda)$. By Theorem \ref{FFIso}, the group scheme $\UAut^\otimes(\lambda)$ is isomorphic to $\CG_{m^\circ}$ fpqc-locally on $m^\circ$ and so in particular smooth.

We will consider the homomorphism $\chi^{m/\ell}\colon D^{\Gamma/|\ell^*|}_{\tilde m}\to \UAut^\otimes(\alpha_m)_{\tilde m}$ from Construction \ref{ChiRamCons} as well as its centralizer $L^{m/\ell}$ which is smooth by Proposition \ref{CentSmooth}.
\begin{lemma}
  \label{ProofLemma1}
  \begin{enumerate}
  \item  $\dim(\UAut^\otimes(\alpha_m)_{\tilde m})= \dim(\UAut^\otimes(\alpha)_{\tilde m})=\dim(U_{<1}(\alpha)_{\tilde m}) + \dim(L^{m/\ell})$
  \item For each $\delta < 1$, the subquotient $\gr_\delta \UAut^\otimes(\alpha)_{\tilde m}$ of $\UAut^\otimes(\alpha)_{\tilde m}$ is a vector group scheme.
  \item The group scheme $U_{<1}(\alpha)$ is a smooth unipotent group scheme.
  \end{enumerate}
\end{lemma}
\begin{proof}

 Example \ref{SpSchAut} and \eqref{LieEndIso2} give canonical isomorphisms
\begin{equation*}
  \UEnd^\otimes(\alpha_m) \cong \UEnd^\otimes(\lambda) \cong \BV(\Lie(\UAut^\otimes(\lambda))) \cong \BV(\Lie(\UAut^\otimes(\alpha_m)))
\end{equation*}
over $m^\circ$ which are $\UAut^\otimes(\alpha_m)$-equivariant. By this equivariance, under this isomorphism, the weight decomposition \eqref{TWeightDec} of $\UEnd^\otimes(\alpha_m)_{\tilde m}$ is the weight decomposition of $\Lie(\UAut^\otimes(\alpha_m)_{\tilde m})$ induced by the homomorphism $\chi^{m/\ell}\colon D^{\Gamma/|\ell^*|}_{\tilde m}\to \UAut^\otimes(\alpha_m)_{\tilde m}$ and the adjoint action of $\UAut^\otimes(\alpha_m)$ on $\Lie(\UAut^\otimes(\alpha_m))$. We write this decomposition as
\begin{equation*}
  \Lie(\UAut^\otimes(\alpha_m)_{\tilde m}) =\oplus_{[\mu]\in \Gamma/|\ell^*|}\gr^{m/\ell}_{[\mu]}(\Lie(\UAut^\otimes(\alpha_m)_{\tilde m})).
\end{equation*}
We now look at the filtration of $U_{<1}(\alpha)_{\tilde m} \subset \UAut^\otimes(\alpha)_{\tilde m}$ by the subgroups $U_{\leq \delta}(\alpha)_{\tilde m}$ for elements $\delta \in I_\ell \setminus\{1\} \subset \Gamma$ (c.f. Lemma \ref{UdeltaTriv}): For each such $\delta$, by Lemma \ref{dProps}, the morphism \eqref{dMor} gives a monomorphism
\begin{equation} \label{grImm}
  \gr_{\delta} \UAut^\otimes(\alpha)_{\tilde m} \into \BV(\gr^{m/\ell}_{[\delta]}(\Lie(\UAut^\otimes(\alpha_m)_{\tilde m})))
\end{equation}
of group schemes over $\tilde m$. Such a monomorphism is a closed immersion.

By Lemma \ref{IellInj} the composition $I_\ell \into \Gamma \onto \Gamma/|\ell^*|$ is injective. Using this and \eqref{grImm} we find:
\begin{align*} \label{DimEst1}
  \dim(U_{<1}(\alpha)_{\tilde m}) &= \sum_{\delta\in I_\ell, \delta<1} \dim(\gr_{\delta} \UAut^\otimes(\alpha)_{\tilde m}) \\ 
                       &\leq \sum_{[\delta] \in \Gamma/ |\ell^*|,\; [\delta] \not= [1]} \dim_{\tilde m}(\gr^{m/\ell}_{[\delta]}(\Lie(\UAut^\otimes(\alpha_m)_{\tilde m}))) \\
  &= \dim_{\tilde m}(\Lie(\UAut^\otimes(\alpha_m)_{\tilde m})) - \dim_{\tilde m}(\gr^{m/\ell}_{[1]}(\Lie(\UAut^\otimes(\alpha_m)_{\tilde m}))).
\end{align*}

By Proposition \ref{CentSmooth}, the submodule $\gr^{m/\ell}_{[1]}(\Lie(\UAut^\otimes(\alpha_m)_{\tilde m}))$ appearing here is the Lie algebra of $L^{m/\ell}$. Using this, this inequality translates to
\begin{equation} \label{DimEq1}
  \dim(U_{<1}(\alpha)_{\tilde m}) \leq \dim(\UAut^\otimes(\alpha_m)_{\tilde m}) - \dim(L^{m/\ell}).
\end{equation}

By Lemma \ref{UBKer} the morphism $b^{\ell/m}\colon \UAut^\otimes(\alpha)_{\tilde m} \to L^{m/\ell}$ has kernel $U_{<1}(\alpha)_{\tilde m}$. The upper semi-continuity of the fiber dimension of finite type morphisms gives $$\dim(\UAut^\otimes(\alpha_m)_{\tilde m})=\dim(\UAut^\otimes(\alpha_m)_m)=\dim(\UAut^\otimes(\alpha)_m) \leq \dim(\UAut^\otimes(\alpha)_{\tilde m}).$$ Together these facts give the reverse inequality
\begin{equation} \label{DimEq2}
  \dim(\UAut^\otimes(\alpha_m)_{\tilde m}) - \dim(U_{<1}(\alpha)_{\tilde m}) \leq \dim(\UAut^\otimes(\alpha)_{\tilde m}) - \dim(U_{<1}(\alpha)_{\tilde m}) \leq \dim(L^{m/\ell}).
\end{equation}
Hence \eqref{DimEq1} and \eqref{DimEq2} are equalities, which shows (i).

This implies that source and target of \eqref{grImm} have the same dimension, and hence that \eqref{grImm} is an isomorphism. This shows (ii). Then varying $\delta$ in (ii) and flat descent imply (iii).
\end{proof}

\begin{lemma} \label{ProofLemma2}
  The group scheme $\UAut^\otimes(\alpha)_{\tilde m}$ is smooth and lies in an exact sequence
  \begin{equation*}
    1 \to U_{<1}(\alpha)_{\tilde m} \to \UAut^\otimes(\alpha)_{\tilde m} \toover{b^{m/\ell}} L^{m/\ell}
  \end{equation*}
  in which the image of $b^{m/\ell}$ is open in $L^{m/\ell}$.
\end{lemma}
\begin{proof}
 By Lemma \ref{UBKer} the morphism $b^{m/\ell}$ has kernel $U_{<1}(\alpha)_{\tilde m}$. Since $L^{m/\ell}$ is smooth, Lemma \ref{ProofLemma1} (i) implies that the image of $b^{m/\ell}$ must be open in $L^{m/\ell}$. So $\UAut^\otimes(\alpha)_{\tilde m}$ is an extension of the smooth group schemes $U_{<1}(\alpha)_{\tilde m}$ and $b^{\ell/m}(\UAut^\otimes(\alpha)_{\tilde m})$ and hence smooth.
\end{proof}

\begin{lemma} \label{ProofLemma3}
  \begin{enumerate}
    \item  For every maximal torus $T$ of $\UAut^\otimes(\alpha)_{\tilde\ell}$, there exists a homomorphism $\chi\colon D^\Gamma_{\tilde \ell} \to T \subset \UAut^\otimes(\alpha)_{\tilde\ell}$ which splits $\alpha$ over $\tilde\ell$.
    \item The special fiber $\UAut^\otimes(\alpha)_{\tilde\ell}$ has the same rank as $\CG_{\tilde\ell}$.
  \end{enumerate}
\end{lemma}
\begin{proof}

(i) By Lemma \ref{LiftingLemma}, the homomorphism $\chi^{m/\ell}$ factors through some maximal torus of $L^{m/\ell}$. Since $\chi^{m/\ell}$ is central in $L^{m/\ell}$, the geometric conjugacy of the maximal tori of $L^{m/\ell}$ shows that this actually holds for every maximal torus of $L^{m/\ell}$.

Let $T'\defeq b^{m/\ell}_{\tilde m}(T_{\tilde m})$. It follows from Lemma \ref{ProofLemma2} and the fact that $U_{<1}(\alpha)$ is unipotent by Lemma \ref{ProofLemma1}, that the homomorphism $b^{m/\ell}_{\tilde m}$ restricts to a bijection $T_{\tilde m} \isoto T'$ and $T'$ is a maximal torus of $L^{m/\ell}$. So $\chi^{m/\ell}$ factors through $T'$. Let $\chi\colon D^{\Gamma/|\ell^*|}_{\tilde m} \to \UAut^\otimes(\alpha)_{\tilde m}$ be the composition of $\chi^{m/\ell}$ with the inverse of the isomorphism $T_{\tilde m} \isoto T'$. Then by Proposition \ref{SplBCComp2} (i), the homomorphism $\chi$ splits $\alpha$ over $\tilde m$. Hence by Lemma \ref{SplittingUnique2} it descends to a homomorphism $D^{\Gamma/|\ell^*|}_{\tilde \ell} \to T$, which we also denote by $\chi$ and which splits $\alpha$ over $\tilde \ell$. Next, by Lemma \ref{SplitLift1}, we may lift $\chi$ to a homomorphism $D^{\Gamma}_{\tilde \ell} \to T$ which splits $\alpha$ over $\tilde\ell$.

(ii) By Lemma \ref{LiftingLemma}, the centralizer $L^{m/\ell}$ of $\chi^{m/\ell}$ contains some maximal torus of $\UAut^\otimes(\alpha_m)_{\tilde m}$. Hence $L^{m/\ell}$, $\UAut^\otimes(\alpha_m)_{\tilde m}$ and the inner form $\CG_{\tilde \ell}$ of $\UAut^\otimes(\alpha_m)_{\tilde\ell}$ all have the same rank. Using this, the above implies (ii).
\end{proof}

\begin{lemma} \label{ProofLemma4}
  \begin{enumerate}
  \item   The group scheme $\UAut^\otimes(\alpha)$ is smooth over $\ell^\circ$.
  \item  For every homomorphism $\chi\colon D^{\Gamma/|\ell^*|}_{\tilde\ell} \to \UAut^\otimes(\alpha)_{\tilde\ell}$ splitting $\alpha$ over $\tilde \ell$, the centralizer $\Cent_{\UAut^\otimes(\alpha)_{\tilde\ell}}(\chi)$ of $\chi$ gives a semidirect product decomposition $$\UAut^\otimes(\alpha)_{\tilde \ell}=U_{<1}(\alpha) \rtimes \Cent_{\UAut^\otimes(\alpha)_{\tilde\ell}}(\chi).$$
  \item If $\ell$ is Henselian, the normed fiber functor $\alpha$ is splittable.
  \end{enumerate}
\end{lemma}
\begin{proof}
  Since claims (i) and (ii) are fpqc-local on $\Spec(\ell^\circ)$, using Lemma \ref{UnrBC} we may replace $\ell$ by its Henselization. So we assume that $\ell$ is Henselian.
  
  By Lemmas \ref{ProofLemma1} and \ref{ProofLemma3}, the assumptions of Proposition \ref{FibrewiseSmoothness2} are satisfied. So Proposition \ref{FibrewiseSmoothness2} (i) gives us an smooth open and closed subgroup scheme $\UAut^\otimes(\alpha)^\text{fl} \subset \UAut^\otimes(\alpha)$.

Let now $T \subset \UAut^\otimes(\alpha)_{\tilde\ell}$ be a maximal torus and $\chi\colon D^\Gamma_{\tilde\ell} \to T$ as in Lemma \ref{ProofLemma3}. Since $T\subset \UAut^\otimes(\alpha)^\circ_{\tilde\ell} \subset \UAut^\otimes(\alpha)^\text{fl}_{\tilde\ell}$, by Theorem \ref{ChiDeform}, the homomorphism $\chi$ can be deformed to a homomorphism $\xi\colon D^\Gamma_{\ell^\circ} \to \UAut^\otimes(\alpha)^\text{fl} \into \UAut^\otimes(\alpha)$ over $\ell^\circ$. Such a homomorphism $\xi$ now splits $\alpha$ over $\ell^\circ$. By Proposition \ref{SplitExtEquiv} this proves (iii).

Let now $\xi'$ be the composition $D^{\Gamma/|\ell^*|}_{\ell^\circ} \to D^{\Gamma}_{\ell^\circ} \toover{\xi} \UAut^\otimes(\alpha)$. Then by Proposition \ref{SplBCComp2} (ii), the homomorphism $b^{m/\ell}$ restricts to an isomorphism 
\begin{equation} \label{SomeCentIso}
  \Cent_{\UAut^\otimes(\alpha)_{m^\circ}}(\xi') \isoto \Cent_{\UAut^\otimes(\alpha_m)}(b^{m/\ell}\circ\xi').
\end{equation}

Since $\UAut^\otimes(\alpha_m)$ is smooth, by Proposition \ref{CentSmooth} so is $\Cent_{\UAut^\otimes(\alpha_m)}(b^{\ell/m} \circ \xi')$. Hence by \eqref{SomeCentIso} so is the group scheme $\Cent_{\UAut^\otimes(\alpha)}(\xi')$.  Since $U_{<1}(\alpha)$ is connected, Proposition \ref{FibrewiseSmoothness2} (ii) applied to $\CH=\Cent_{\UAut^\otimes(\alpha)}(\xi')$ shows that $\UAut^\otimes(\alpha)$ is smooth. This proves (i).

Let finally $\tilde\chi\colon D^{\Gamma/|\ell^*|}_{\tilde\ell} \to \UAut^\otimes(\alpha)_{\tilde\ell}$ be a homomorphism splitting $\alpha$ over $\tilde\ell$. By Lemma \ref{LiftingLemma} there exists a maximal torus $T$ of $\UAut^\otimes(\alpha)_{\tilde\ell}$ through which $\tilde\chi$ factors. Then by Lemma \ref{SplitLift1} we may lift $\tilde\chi$ to a homomorphism $\chi\colon D^\Gamma_{\tilde\ell} \to \UAut^\otimes(\alpha)_{\tilde\ell}$. If we apply the above to this $\chi$, the fact that $(b^{m/\ell}\circ\xi')_{\tilde m}=\chi^{m/\ell}$ by Proposition \ref{SplBCComp2} shows that taking the special fiber of \eqref{SomeCentIso} proves (ii). 
\end{proof}

\begin{lemma} \label{ProofLemma6}
  For each $\delta < 1$, the group scheme $\gr_\delta \UAut^\otimes(\alpha)_{\tilde\ell}$ is a vector group scheme. 
\end{lemma}
\begin{proof}
By Lemma \ref{UdeltaTriv} it suffices to consider the case $\delta \in I_\ell$.

  We already know from Lemma \ref{ProofLemma1} that $\gr_\delta \UAut^\otimes(\alpha)_{\tilde m}$ is a vector group scheme. We descend this statement to $\tilde\ell$ as follows: If $\operatorname{char}(\tilde \ell)=0$, the statement holds by \cite[Cor. B.2.7]{CGP}. So we may assume that $\operatorname{char}(\tilde\ell)=p>0$.

  By Lemma \ref{ProofLemma3}, there exists a homomorphism $D^\Gamma_{\tilde\ell} \to \UAut^\otimes(\alpha)_{\tilde \ell}$ splitting $\alpha$. Via conjugation on the normal subgroup schemes $U_{\leq \delta}(\alpha) \subset \UAut^\otimes(\alpha)_{\tilde\ell}$, such a homomorphism induces an action of $D^\Gamma_{\tilde\ell}$ on $\gr_\delta\UAut^\otimes(\alpha)_{\tilde\ell}$. By descent $\gr_\delta \UAut^\otimes(\alpha)_{\tilde \ell}$ is a smooth commutative $p$-torsion group scheme. Hence by \cite[Thm. B.4.3]{CGP}, this group scheme is a vector group scheme if the action of $D^\Gamma_{\tilde\ell}$ has no fixed points. For this, since $\delta \in I_\ell$, it suffices to show that $D^\Gamma_{\tilde m}$ acts on $\gr_\delta \UAut^\otimes(\alpha)_{\tilde m}$ with weights in $\delta |\ell^*|$. This follows for example using the injection $\gr_\delta \UAut^\otimes(\alpha)_{\tilde m} \into  \gr_{[\delta]}^{m/\ell} (\UEnd^\otimes(\alpha_m)_{\tilde m})$ given by \eqref{dMor}.  
\end{proof}

Now we are almost done:
\begin{proof}[Proof of Theorem \ref{MainTheorem}]
  Statements (i), (v) and (vii) are given by Lemma \ref{ProofLemma4}. Statement (ii) is given by Lemma \ref{ProofLemma6} and implies (iii). Statements (iv) and (vi) are given by Lemma \ref{ProofLemma3}.

  To prove (viii), let $\CT$ be a residually maximal split torus of $\UAut^\otimes(\alpha)$. By applying Lemma \ref{ProofLemma3} to a maximal torus $T$ of $\UAut^\otimes(\alpha)_{\tilde\ell}$ containing $\CT_{\tilde\ell}$, we see that there exists a homomorphism $\chi'\colon D^\Gamma_{\tilde\ell} \to \CT_{\tilde\ell}$ which splits $\alpha$ over $\tilde\ell$. Such a $\chi'$ then deforms to a homomorphism $\xi'\colon D^\Gamma_{\ell^\circ} \to \CT$ which splits $\alpha$ over $\ell^\circ$. This shows the existence claim in (viii). Unicity in (viii) is given by Lemma \ref{SplittingUnique2}. Using the unicity statement in (viii), claim (ix) now follows from Proposition \ref{ResSplitProp} (ii).
\end{proof}

\section{A Tannakian formalism for Bruhat-Tits buildings} \label{BTSec}
In this section we let $\Gamma$ be the totally ordered group $(\BR^{>0},\cdot, \leq)$, we let $\ell$ be a discretely valued Henselian non-archimedean field and $\CG$ a smooth affine group scheme over $\ell^\circ$. We denote $\CG_\ell$ by $G$. We also let $\breve\ell$ be a maximal unramified extension of $\ell$.

\subsection{The set $N^\otimes(\CG)$}

\begin{definition}
  Let $\omega \colon \Rep^\circ G \to \Mod_{\ell}$ be a fiber functor.
  \begin{enumerate}
  \item A \emph{norm $\theta$ on $\omega$} is a tuple $(\theta_X)_{X\in\Rep^\circ G}$ consisting of splittable norms $\theta_X$ on each $\omega(X)$ which are functorial in $X$ in such a way that the assignment $X \mapsto (\omega(X_\ell),\theta_{X_\ell})$ gives a normed fiber functor $\alpha_\theta$ on $\Rep^\circ \CG$ over $\ell^\circ$.
  \item Given a norm $\theta$ on $\omega$ and an element $g \in \UAut^\otimes(\omega)(\ell)$, we let $g\cdot \theta$ be the norm $(g_X \cdot \theta_X)_{X \in \Rep^\circ X}$, where $g_X$ denotes the element of $\GL(\omega(X))$ given by $g$. (In other words $(g\cdot\theta)_X(v)=\theta_X(g_X^{-1}v)$ for an element $v \in \omega(X)$.)
  \end{enumerate}

\end{definition}

\begin{remark}
  A norm on $\omega$ is the same as a normed fiber functor $\alpha\colon \Rep^\circ \CG\to \Norm^\circ(\ell)$ together with an isomorphism $\forg\circ \alpha\cong \omega$.
\end{remark}

\begin{definition}
  We let $N^\otimes(\CG)$ be the set of norms $\theta$ on $\omega_{G}$. We equip this set with the $G(\ell)$-action given by $\theta \mapsto g\cdot \theta$.

  For each $X \in \Rep^\circ \CG$, the lattice $\omega_\CG(X) \subset \omega_G(X)$ defines a norm $\alpha_{\omega_\CG(\CX)}$ on $\omega_G(X)$. For varying $X$, these norms define a norm on $\omega_G$ which we denote by $\theta_\CG \in N^\otimes(\CG)$.
\end{definition}

Note that, by construction, the stabilizer in $G(\ell)$ of a point $\theta \in N^\otimes(\CG)$ is equal to the group $\UAut^\otimes(\alpha_\theta)(\ell^\circ)$. 
\begin{lemma} \label{NormBCLemma}
  Let $\ell \subset m$ be a non-archimedean extension. For each norm $\theta$ on $\omega_G$, there exists a unique norm $\theta_m$ on $\omega_{G_m}$ for which the following diagram commutes:
  \begin{equation*}
    \xymatrix{
      \Rep^\circ \CG \ar[r]^{\alpha_\theta} \ar[d] & \Norm^\circ(\ell) \ar[d] \\
      \Rep^\circ \CG_{m^\circ} \ar[r]_{\alpha_{\theta_m}}& \Norm^\circ(m) \\
      }
    \end{equation*}
  \end{lemma}
  \begin{proof}
    Let $\theta$ be a norm on $\omega_G$. By Theorem \ref{MainTheorem}, there exist an integral model $\lambda\colon \Rep^\circ \CG \to \Mod_{\ell^\circ}$ of $\omega_G$ together with a homomorphism $\chi\colon D^\Gamma_{\ell^\circ} \to \UAut^\otimes(\lambda)$ which split $\alpha_\theta$. By Theorem \ref{FFIso}, the fiber functor $\lambda$ differs from $\omega_\CG$ by a $\CG$-torsor $I$ over $\ell^\circ$. Twisting $\omega_{\CG_{m^\circ}}$ by $I_{m^\circ}$ gives a fiber functor $\lambda_{m^\circ}$ which fits in the following commutative diagram:
  \begin{equation*}
    \xymatrix{
      \Rep^\circ \CG \ar[r]^{\lambda} \ar[d] & \Mod_{\ell^\circ} \ar[d] \\
      \Rep^\circ \CG_{m^\circ} \ar[r]_{\lambda_{m^\circ}}& \Mod_{m^\circ} \\
      }
    \end{equation*}    
    Since $\UAut^\otimes(\lambda)$ is the sheaf of $\CG$-linear automorphisms of $I$, and analogously $\UAut^\otimes(\lambda_{m^\circ})$ is the sheaf of automorphisms of $I_{m^\circ}$, there is a canonical isomorphism $\UAut^\otimes(\lambda_{m^\circ})=\UAut^\otimes(\lambda)_{m^\circ}$.

    Using this, one checks that the normed fiber functor $\alpha_{\lambda_{m^\circ},\chi_{m^\circ}} \colon \Rep^\circ \CG_{m^\circ} \to \Norm^\circ(m)$ extends $\alpha_\theta$. The underlying fiber functor of $\alpha_{\lambda_{m^\circ},\chi_{m^\circ}}$ extends $\omega_G$. It follows from Lemma \ref{BCEpi}, that up to isomorphism there is only one such fiber functor, namely $\omega_{G_m}$. So there exists an isomorphism $\forg \circ \alpha_{\lambda_{m^\circ},\chi_{m^\circ}}\cong \omega_{G_m}$, using which $\alpha_{\lambda_{m^\circ},\chi_{m^\circ}}$ defines a norm $\theta_m$ on $\omega_m$ as desired.

The unicity of $\theta_m$ follows using Lemma \ref{BCEpi}.
  \end{proof}
\begin{construction} \label{BFunc}
  We consider the functoriality of $N^\otimes(\CG)$: Let $\ell \subset m$ be a non-archimedean Henselian field extension. Let furthermore $\CG'$ be a smooth affine group scheme over $m^\circ$ together with a homomorphism $h\colon \CG_{m^\circ} \to \CG'$. We denote by $h^*\colon \Rep^\circ \CG'\to \Rep^\circ \CG_{m^\circ}$ the associated tensor functor.

  For a norm $\theta \in N^\otimes(\CG)$, we obtain from Lemma \ref{NormBCLemma} the norm $\theta_m \in N^\otimes(\CG_{m^\circ})$. To this in turn we may associate the norm $(\theta_{m,h^*(X)})_{X \in \Rep^\circ \CG'} \in N^\otimes(\CG')$.

 Altogether we obtain a natural map
    \begin{equation*}
      N^\otimes(\CG) \to N^\otimes(\CG')
    \end{equation*}
    which is equivariant with respect to $G(\ell) \to \G'(m)$.
\end{construction}

\begin{lemma} \label{BFuncInj}
  In Construction \ref{BFunc}, if the homomorphism $h$ is a closed immersion, then the induced map $N^\otimes(\CG) \to N^\otimes(\CG')$ is an injection.
\end{lemma}
\begin{proof}
  The commutative diagram in Lemma \ref{BFunc} shows that the map $$N^\otimes(\CG) \to N^\otimes(\CG_{m^\circ}), \; \theta\mapsto \theta_m$$ is an injection. Proposition \ref{ClosedImmCrit} implies that $N^\otimes(\CG_{m^\circ}) \to N^\otimes(\CG')$ is an injection.
\end{proof}

\begin{construction} \label{GalB}
  Now we take $m$ to be a non-archimedean extension of $\ell$ which is Galois over $\ell$. For $\sigma \in \Gal(m/\ell)$ and $X \in \Rep^\circ \CG_{m^\circ}$, the $m^\circ$ module $$\sigma^*(X)\defeq X \otimes_{m^\circ,\sigma} m^\circ$$ is naturally a $\CG_{m^\circ}$-representation which comes with the semilinear bijection $b_\sigma\colon X \to \sigma^*(X), v \mapsto v \otimes 1$. Using this, we may associate to any norm $\theta=(\theta_X)_{X\in \Rep^\circ \CG_{m^\circ}} \in N^\otimes(\CG_{m^\circ})$ the norm $\sigma\cdot\theta$ given by $(\sigma\cdot\theta)_X(v)=\theta_{(\sigma^{-1})*(X)}(b_{\sigma^{-1}}(x))$. This defines a left action of $\Gal(m/\ell)$ on $N^\otimes(\CG_{m^\circ})$.
\end{construction}

\begin{theorem} \label{GalBProps}
  In the situation of Construction \ref{GalB}, the following hold:
  \begin{enumerate}
  \item The inclusion $N^\otimes(\CG) \to N^\otimes(\CG_{m^\circ})$ factors through the fixed-point set $N^\otimes(\CG_{m^\circ})^{\Gal(m/\ell)}$.
  \item  If $m$ is tamely ramified over $\ell$, the inclusion $N^\otimes(\CG) \into N^\otimes(\CG_{m^\circ})^{\Gal(m/\ell)}$ is a bijection.
  \end{enumerate}
\end{theorem}
\begin{proof}
  (i) By Lemma \ref{BCEpi}, every norm $\theta \in N^\otimes(\CG_{m^\circ})$ is determined by the norms $\theta_{X_{m^\circ}}$ for $X \in \Rep^\circ \CG$. But for $X \in \Rep^\circ \CG$ and $\sigma \in \Gal(m/\ell)$, the representations $X_{m^\circ}$ and $\sigma^*(X_{m^\circ})$ are isomorphic in such a way that $b_\sigma$ becomes the natural action of $\sigma$ on $X_{m^\circ}$. In other words, $(\sigma\cdot\theta)_{X_{m^\circ}}$ is equal to $\sigma\cdot \theta_{X_m^\circ}$ where the later action is the one from Definition \ref{GalNormAction}. Using this, it follows that if $\theta$ is in the image of $N^\otimes(\CG)$, then it is fixed by the $\Gal(m/\ell)$-action.

  (ii) Let $\theta \in N^\otimes(\CG_{m^\circ})$ be a $\Gal(m/\ell)$-invariant norm. By the above, Proposition \ref{TameDesc0} gives for each $X \in \Rep^\circ \CG$ the splittable norm $\theta_X \defeq \theta_{X_{m^\circ}}|_{X_\ell}$ on $X_\ell$ which by base change to $m$ gives back $\theta_{X_{m^\circ}}$. These are by construction functorial in $X$ and we need to check that these norms are compatible with tensor products and exact. Exactness is given by Lemma \ref{NormExactnessDescent} and compatibility with tensor products can be checked after base change to $m$. 
\end{proof}

\begin{construction} \label{ChiNormCons3}
  For every homomorphism $\chi\colon D_{\ell^\circ}^\Gamma \to \CG=\UAut^\otimes(\omega_\CG)$, Construction \ref{ChiNormCons} gives us the normed fiber functor $\alpha_{\omega_\CG,\chi}$ split by $(\omega_\CG,\chi)$. This functor comes with a natural isomorphism $\psi\colon \forg\circ \alpha_{\omega_\CG,\chi} \cong \omega_G$ via which we consider $\alpha_{\omega_\CG,\chi}$ as a norm $\theta_{\chi}$ on $\omega_\CG$.   
\end{construction}

\begin{proposition} \label{StrongerSplitting}
  Assume that every \'etale $\CG$-torsor over $\ell^\circ$ is trivial. Then for every norm $\theta \in N^\otimes(\CG)$, there exist residually maximal split torus $\CT$ of $\CG$, a cocharacter $\chi\colon D^\Gamma_{\ell^\circ} \to \CT \subset \CG$ and an element $g \in G(\ell)$ such that $\theta=g\cdot \theta_\chi$.
\end{proposition}
\begin{proof}
  By Theorem \ref{MainTheorem}, there exists some integral model $\lambda$ of $\omega_G$ together with a homomorphism $\chi\colon D_{\ell^\circ}^\Gamma \to \UAut^\otimes(\lambda)$ such that $(\lambda,\chi)$ splits $\alpha_\theta$. By Theorem \ref{FFIso}, the two fiber functors $\lambda$ and $\omega_\CG$ on $\Rep^\circ \CG$ differ by a flat $\CG$-torsor over $\ell^\circ$. Since $\CG$ is smooth, any such torsor is trivial \'etale-locally on $\ell^\circ$. Hence our assumption ensures that there exists an isomorphism $\Phi\colon \omega_\CG \isoto \lambda$ of fiber functors. Under the given isomorphisms $\lambda_\ell \cong \omega_G \cong \omega_{\CG,\ell}$, the generic fiber of such a $\Phi$ is given by multiplication by some element $g \in G(\ell)$. Then $\lambda=g\cdot \omega_{\CG}$ and conjugation by $g^{-1}$ gives an isomorphism $\UAut^\otimes(\lambda) \isoto \CG$ over $\ell^\circ$. We denote the image $D^{\Gamma}_{\ell^\circ} \to \CG$ of $\chi$ under this isomorphism by $\leftexp{g^{-1}}{\chi}$. Then Lemma \ref{SplitTransl} shows that $\theta=g\cdot \theta_{\leftexp{g^{-1}}{\chi}}$. Finally there exists a residually maximal split torus $\CT$ of $\CG$ through which $\chi$ factors by Proposition \ref{ResSplitProp}.
\end{proof}

\subsubsection*{Description of $N^\otimes(\CG)$ for split tori}

In case $\CG$ is a split torus $\CT$ over $\ell^\circ$, whose generic fiber we denote by $T$, we can describe the $T(\ell)$-set $I(\CT)$ completely:

Since every homomorphism $D^{\Gamma}_{\ell} \to T$ extends uniquely to a homomorphism $D^{\Gamma}_{\ell^\circ} \to \CT$, Construction \ref{ChiNormCons3} gives the natural map
\begin{equation} \label{BTT}
  X_*(T)_\Gamma\defeq X_*(T) \otimes \Gamma \to N^\otimes(\CT), \chi \mapsto \theta_{\chi}.
\end{equation}

Furthermore, there is a natural isomorphism
\begin{equation}
  \label{TIso}
  X_*(T)_{|\ell^*|} \cong T(\ell)/\CT(\ell^\circ)
\end{equation}
which for a cocharacter $\chi\colon \Gm \to T$ and an element $x \in \ell^*$ sends $\chi \otimes |x|$ to $\chi(x)T(\ell^\circ)$. We let $t \in T(\ell)$ act on $X_*(T)_\Gamma$ via translation by minus the element of $X_*(T)_{|\ell^*|}$ corresponding to $t\CT(\ell^\circ)$ under \eqref{TIso}.

We write the map $T(\ell) \to X_*(T)_{|\ell^*|}$ given by \eqref{TIso} as $t \mapsto |t|$.
\begin{proposition} \label{TBGDesc}
  The map \eqref{BTT} is a $T(\ell)$-equivariant bijection.
\end{proposition}
\begin{proof}
  First we show equivariance:  For $X \in \Rep^\circ \CT$, the norm $\theta_{\chi,X}$ on $\omega_\CT(X)_\ell$ associated to $\chi\colon D_\ell^\Gamma \to T$ is constructed in Construction \ref{ChiNormCons} by adding the norm associated to the lattice $\omega_\CT(X) \subset \omega_\CT(X)_\ell$ and the weights of $\chi$ on $\omega_\CT(X)_\ell$. The statement that \eqref{BTT} is $T(\ell)$-equivariant then amounts to saying that given a homomorphism $\chi'\colon \BG_m \to T$ and an element $x \in \ell^*$, replacing in this construction the lattice $\omega_\CT(X)$ by $\chi'(x)\omega_\CT(X)$ has the same effect as replacing $\chi$ by $\chi-|x|\chi'$. This follows from the construction of $\theta_\chi$.
  
The injectivity of \eqref{BTT} is a special case of Lemma \ref{SplittingUnique2}. Since $\CT$ is a split torus, every \'etale $\CT$-torsor over $\ell^\circ$ is trivial. Using this, Proposition \ref{StrongerSplitting} shows that every element of $N^\otimes(\CT)$ is in the $T(\ell)$-orbit of an element in the image of \eqref{BTT}. Using the equivariance of \eqref{BTT}, this implies surjectivity. 
\end{proof}

\subsection{A Tannakian formalism for Bruhat-Tits buildings of unramified groups} \label{TannakaSubsec}

Now we assume in addition that the geometric fibers of $\CG$ are connected and reductive. Then (c.f. \cite[Prop. 125]{CornutFiltrations}), the generic fiber $G$ splits over an unramified extension of $\ell$ and hence the extended Bruhat-Tits building of $G$, which we denote by $I(G)$, exists. By our assumptions on $\CG$, this group scheme is the hyperspecial parahoric integral model associated so some point $x \in I(G)$, and we also fix such a point $x$. By \cite[6.3.2]{CornutFiltrations}, the pointed building $(I(G),x)$ is naturally functorial in the Henselian ground field $\ell$.

  The split ranks of $\CG_{\tilde\ell}$ and $G$ are equal and so we are in the situation of Remark \ref{MaxSplitRmk}.
\begin{definition}
  Let $S(\CG)$ be the set of those maximal split tori of $G$ which extend to a maximal split torus $\CT \subset \CG$.
\end{definition}
Given a torus $T \in S(\CG)$, in the following we will always denote by $\CT \subset \CG$ the maximal split torus extending $T$. By Bruhat-Tits theory, the tori in $S(\CG)$ are the maximal split tori $T$ of $G$ for which $A(T)$ contains our choosen point $x$. 

For $T \in S(\CG)$, the inclusion $\CT \subset \CG$ induces a $T(\ell)$-equivariant map $N^\otimes(\CT) \to N^\otimes(\CG)$. It follows from Proposition \ref{TBGDesc} and Lemma \ref{BFuncInj} that this is an inclusion whose image is the set of norms $\theta$ on $\omega_{\CG}$ for which there exists a homomorphism $\chi\colon D^\Gamma_{\ell^\circ} \to \CT$ such that $(\omega_\CG,\chi)$ splits $\alpha_\theta$. We denote this image of $N^\otimes(\CT)$ in $N^\otimes(\CG)$ by $A^\otimes(\CT)$.

Furthermore, the structure of $A(T)$ as an affine $X_*(T)_\Gamma$-space together with our chosen base point $x \in A(T)$ gives canonical bijection
\begin{equation}
  \label{BTApp}
  X_*(T)_\Gamma \cong A(T)
\end{equation}
which sends $0$ to $x$. This bijection is $T(\ell)$-equivariant, where $T(\ell)$ acts on $X_*(T)_\Gamma$ via \eqref{TIso} as above.

Then our main result on Bruhat-Tits buildings is the following. We will obtain this result by combining Theorem \ref{MainTheorem} with \cite[Theorem 130]{CornutFiltrations}.

\begin{theorem} \label{BTComp}
  \begin{enumerate}
  \item There exists a unique $G(\ell)$-invariant bijection $c\colon N^\otimes(\CG) \isoto I(G)$ whose restriction to $A^\otimes(T)$ for every $T \in S(\CG)$ is the map $c_T$ defined by the commutative diagram
  \begin{equation} \label{cT}
    \xymatrix{
      & X_*(T)_\Gamma  \ar[ld]_\cong \ar[rd]^\cong & \\
      A^\otimes(T) \ar[rr]_{c_T} & & A(T)
      }
    \end{equation}
    in which the diagonal maps are given by \eqref{BTT} and \eqref{BTApp}.

  \item For each $\theta \in N^\otimes(\CG)$, the group scheme $\UAut^\otimes(\alpha_\theta)$ is the parahoric group scheme associated to $c(\theta)$, that is the unique smooth affine integral model of $G$ whose group of $\breve\ell^\circ$-points is equal to the stabilizer of $c(\theta)$ in $G(\breve\ell)$.
  \end{enumerate}
  \end{theorem}

  \begin{proof}
    For any point $y \in I(G)$, there exists an appartment containing both $x$ and $y$, which is hence of the form $A(T)$ for some $T \in S(\CG)$. So the map $c$ is unique if it exists.
    
    We first reduce to the case that $\ell$ is strictly Henselian: If $\Sigma\defeq \Gal(\breve \ell/\ell)$, then by \cite[Prop. 5.1.1]{MR0491992}, there exist a natural $\Sigma$-action on $I(G_{\breve\ell})$ and the natural map  $I(G) \to I(G_{\breve \ell})$ identifies $I(G)$ with $I(G_{\breve \ell})^\Sigma$. Since the formation of parahoric group schemes commutes with unramified base change, the image of $x$ in $I(G_{\breve \ell})$ is again hyperspecial, and $\CG_{\breve\ell^\circ}$ is the associated integral model. We assume that there is a bijection $c_{\breve\ell} \colon N^\otimes(\CG_{\breve\ell^\circ}) \to I(G_{\breve\ell})$ satisfying the condition from (i). Then for $\sigma \in \Sigma$ and $T \in S(G_{\breve\ell})$, the construction of the maps $c_T$ and of the $\Sigma$-action on $I(G_{\breve\ell})$ imply that the following diagram commutes:
  \begin{equation*}
    \xymatrix{
      A^\otimes(T) \ar[r]^{c_T} \ar[d]_{\sigma\cdot \underline{\;\;}} & A(T) \ar[d]_{\sigma\cdot \underline{\;\;}} \\
      A^\otimes(\sigma(T)) \ar[r]^{c_{\sigma(T)}} & A(\sigma(T)) 
      }
  \end{equation*}
  This implies that $c_{\breve\ell}$ is $\Sigma$-equivariant. Hence using Theorem \ref{GalBProps} it restricts to a $G(\ell)=G(\breve\ell)^\Sigma$-equivariant bijection $c\colon N^\otimes(\CG)=N^\otimes(\CG_{\breve\ell^\circ})^\Sigma \to I(G)=I(G_{\breve\ell})^\Sigma$.

    If $T \in S(\CG)$, then (c.f. \cite[Remark 126]{CornutFiltrations}) there exists a torus $\tilde T \in S(G_{\breve \ell})$ containing $T_{\breve \ell}$ such that $A(T)$ maps to $A(\tilde T)$. If $\CT \subset \CG$ and $\widetilde\CT \subset \CG_{\breve\ell}$ are the associated split subtori, then the homomorphism $T_{\breve\ell} \to \tilde T$ extends uniquely to a homomorphism $\CT_{\breve\ell^\circ} \to \widetilde \CT$. As in \cite[Remark 126]{CornutFiltrations}, this gives the following commutative diagram:
    \begin{equation*}
      \xymatrix{
        X_*(\tilde T)_\Gamma \ar[r]^\cong & A^\otimes(\widetilde\CT) \ar[r]^{c_{\tilde T}} & A(\widetilde T) \\
        X_*(T)_\Gamma \ar[u] \ar[r]^\cong & A^\otimes(\CT) \ar[r]^{c_{T}} \ar[u] & A(T) \ar[u]
        }
    \end{equation*}

 Using this, we see that $c$ satisfies the condition of (i). So statement (i) over $\breve\ell$ implies statement (i) over $\ell$.

 Since by Lemma \ref{UnrBC} the formation of $\UAut^\otimes(\alpha_\theta)$ commutes with base change to $\breve \ell$, it also suffices to check (ii) over $\breve\ell$, where it will follow from the bijectivity of $c$ once we prove (i).

  So we assume that $\ell$ is strictly Henselian and prove (i):

  We denote by $\tilde N^\otimes(\CG)$ the set $B'(\omega_\CG,\ell)$ from \cite[6.4.4]{CornutFiltrations}. This is the set of tuples $(\theta_X)_{X \in \Rep^\circ \CG}$ of norms $\theta_X$ on each $X_\ell$ which are $\ell^\circ$-linearly functorial in $X$ as well as compatible with tensor products. So $N^\otimes(\CG) \subset \tilde N^\otimes(\CG)$, and the tuples in $\tilde N^\otimes(\CG)$ differ from the ones in $N^\otimes(\CG)$ in that they are not necessarily exact in $X$.

  Let $F(G)$ be the set of $\Gamma$-indexed filtrations on $\omega_G$ in the sense of \cite[3.3.5]{CornutFiltrations}. So $F(G)$ is the set of tuples $(F_X)_{X \in \Rep^\circ G}$ of filtrations $F_X$ on each $\omega_G(X)$ which are $\ell$-linearly functorial and exact in $X$ as well as compatible with tensor products. Then (c.f. \cite[6.4.5]{CornutFiltrations}) the maps $$+\colon N(\omega_G(X)_\ell) \times F(\omega_G(X)_\ell) \to N(\omega_G(X)_\ell)$$ from Lemma \ref{PullMapExists} give a map
  \begin{equation*}
     + \colon \tilde N^\otimes(\CG) \times F(G) \to \tilde N^\otimes(\CG), \; ((\theta_X)_{X \in \Rep^\circ \CG}, (F_X)_{X \in \Rep^\circ G}) \mapsto (\theta_X + F_{X_\ell})_{X\in\Rep^\circ \CG}.
  \end{equation*}

  In \cite[Theorem 130]{CornutFiltrations}, Cornut shows that there exists a unique $G(\ell)$-invariant injection $\iota\colon I(G) \into \tilde N^\otimes(\CG)$ which for each $F \in F(G)$ sends $x + F$ to $\theta_\CG +F$. We will show that $\iota$ has image $N^\otimes(\CG)$, and then obtain our desired bijection $c$ as the inverse of $\iota$.

  For some homomorphism $\chi\colon D^\Gamma_\ell \to G$, let $F_\chi \in F(G)$ be the associated filtration as in \cite[3.1.2]{CornutFiltrations}. Then by comparing Construction \ref{ChiNormCons} with \cite[6.1.3]{CornutFiltrations}, we see that for any homomorphism $\chi\colon D^\Gamma_{\ell^\circ} \to \CG$, the norm $\theta_{\chi}$ is equal to $\theta_\CG+F_{\chi_\ell}$. Furthermore, by \cite[Subsection 3.10]{CornutFiltrations}, every $F \in F(G)$ is of the form $F_\chi$ for some homomorphism $\chi\colon D^\Gamma_\ell \to G$. By \cite[Theorem 60]{CornutFiltrations}, the stabilizer in $G$ of $F$ is a parabolic subgroup. Hence by the Iwasawa decomposition we may write every element $g\in G(\ell)$ in the form $hp$ with $h\in \CG(\ell^\circ)$ and $p$ stabilizing $F$. By suitably conjugating $\chi$, this implies that $F$ is of the form $F_{\chi'}$ for some homomorphism $\chi'\colon D^\Gamma_{\ell^\circ} \to \CG$. This shows that the image of $\iota$ is the set $\{\theta_{\chi} \mid \chi \colon D^\Gamma_{\ell^\circ} \to \CG \}$ and in particular contained in $N^\otimes(\CG)$.

  Since $\ell$ is strictly Henselian, by Proposition \ref{StrongerSplitting} every element of $N^\otimes(\CG)$ is of the form $g \cdot \theta_{\chi}$ for some homomorphism $\chi\colon D^\Gamma_{\ell^\circ} \to \CG$. But since $\iota$ is $G(\ell)$-equivariant, the set $\iota(I(G))$ is $G(\ell)$-equivariant. This implies $N^\otimes(\CG)=\iota(I(G))$.

  So $\iota$ is a bijection $I(G) \to N^\otimes(\CG)$, and by the above the defining property $\iota(x+F)=\theta_\CG+F$ translates to fact that the inverse $c$ of $\iota$ satisfies the desired property from (i).
\end{proof}

  \subsection{Functoriality of Bruhat-Tits buildings} \label{FuncSS}
Now we let more generally $G$ be a connected reductive group scheme over $\ell$ for which a valuation of the root datum of $G$ in the sense of \cite[6.2.1]{BT1}, and hence the extended building $I(G)$ of $G$, exists.

  \begin{definition} \label{ToralDef}
    Let $m$ be a non-archimedean extension of $\ell$, $H$ a connected reductive group over $m$ admiting a valuation of the root datum and $h\colon G_m \to H$ a homomorphism. A $G(\ell) \to H(m)$-equivariant map $I(G) \to I(H_m)$ is \emph{toral} if for every maximal split torus $T \subset G$ there exists a maximal split torus $S$ of $H_m$ for which $h(T_m)\subset S$ and $f(A(T)) \subset A(S)$ and such that the restricted map $A(T) \to A(S)$ of affine spaces is $X_*(T)_\Gamma \to X_*(S)_\Gamma$-equivariant.
  \end{definition}

  Let now $H$ be a second reductive group over $\ell$ satisfying the same assumption as $G$. In \cite{LandvogtFunctoriality}, Langvogt proves the following functoriality result:
  \begin{theorem}
    Let $h\colon G \to H$ be a group homomorphism over $\ell$ and let $m$ be a Galois extension of $\ell$. Then there exists a $G(m)$- and $\Gal(m/\ell)$-equivariant toral map
    \begin{equation*}
      f\colon I(G_m) \to I(H_m).
    \end{equation*}
  \end{theorem}

  Via a factorization of $h$ into an epimonomorphism and a monomorphism, the proof of this result reduces to the case that $h$ is either an epimorphism or a monomorphism. The case of an epimorphism is easier and treated in \cite[Theorem 2.1.8]{LandvogtFunctoriality}. For the case of a monomorphism, we can give a new, shorter proof as follows. The use of the projection $\pi\colon I(H_o) \to I(H_m)$ here is inspired by \cite{LandvogtFunctoriality}.

  \begin{definition}
    In \cite[Subsection 2.5]{BT1}, Bruhat and Tits construct natural metrics on $I(G)$ which are $G(\ell)$-equivariant, restrict to a Euclidean metric on each appartment, and make $I(G)$ into a CAT(0)-space. We call these \emph{standard metrics}. 
  \end{definition}
  Although these metrics are not unique, the resulting topology on $I(G)$ is independent of the choice of metric. The resulting notion of a geodesic between two points in $I(G)$ is also independent of the choice of metric since the geodesic between any two point is given by the straight line segment between these points in any apparment containing them. Hence any toral map between two buildings is continuous with respect to this topology on the buildings and maps geodesics to geodesics.

  \begin{lemma} \label{ToralRes}
    Let $\breve f\colon I(G_{\breve\ell}) \to I(H_{\breve\ell})$ be a $\Gal(\breve \ell/ \ell)$- and $G(\breve\ell) \to H(\breve\ell)$-equivariant toral map. Then the restriction
    \begin{equation*}
      f=\breve f|_{I(G)} \colon I(G)=I(G_{\breve\ell})^{\Gal(\breve\ell/\ell)}  \to I(H)=I(H_{\breve\ell})^{\Gal(\breve\ell/\ell)}
    \end{equation*}
    is a toral $G(\ell) \to H(\ell)$-equivariant map.
  \end{lemma}
  \begin{proof}
    The equivariance follows directly from the definition of $f$. To verify that $f$ is toral, let $T \subset G$ be a split maximal torus and $x$ a point in $A(T)$. We consider the parahoric integral model $\CP_x$ of $G$ (resp. $\CP_{f(x)}$ of $H$) associated to $x$ (resp. $f(x)$). Since $f$ is the restriction of $\breve f$, Theorem \ref{Etoffe} ensures that $h$ extends to a homomorphism $\CP_x \to \CP_{f(x)}$ over $\ell^\circ$. Since $x \in A(T)$, the torus $T$ extends to a split closed subtorus $\CT \subset \CP$, and by Proposition \ref{ResSplitProp}, the split torus $\CT$ maps to some residually maximal split subtorus $\CS$ of $\CP_{f(x)}$. The generic fiber $S$ of $\CS$ is a maximal split torus of $H$ whose appartment $A(S)$ contains $f(x)$. Since $T$ maps to $S$, the appartment $A(S)$ also contains $t\cdot f(x)=f(tx)$ for all $t \in T(\ell)$. Since $\breve f$ is toral, the map $\breve f$, and hence the map $f$, sends the geodesic between any two points $t_1\cdot x$ and $t_2\cdot x$ for elements $t_i \in T(\ell)$ to the geodesic between $f(t_1\cdot x)$ and $f(t_2\cdot x)$. Hence all these geodesics are contained in $A(S)$. Since the $T(\ell)$-orbit of $x$ forms a lattice in $A(T)$, the geodesics between all possible pairs of points of this orbit are dense in $A(T)$. This implies $f(A(T))\subset A(S)$, and also that the restriction of $f$ to $A(T)$ is $X_*(T)_\Gamma \to X_*(S)_\Gamma$-equivariant.
  \end{proof}

  \begin{lemma} \label{IsometricToral}
    Assume that $m$ is strictly strictly Henselian and let $f \colon I(G) \to I(H)$ be $G(\ell)\to H(\ell)$-equivariant.
    \begin{enumerate}
    \item If $f$ is an isometry with respect to some standard metrics on $I(G)$ and $I(H)$, then it is toral.
    \item  If $f$ is toral, then it is uniquely determined by the image of a single point $x \in I(G)$.
    \end{enumerate}

      \end{lemma}
      \begin{proof}
        (i)  We argue similary as in the proof of Lemma \ref{ToralRes}: Let $T \subset G$ be a maximal split torus and let $x \in A(T)$. By arguing with parahoric models as in the proof of Lemma \ref{ToralRes}, we find a split maximal torus $S$ of $H$ through which $T$ factors and such that $A(S)$ contains $f(x)$. Then $A(S)$ contains $f(t\cdot x)$ for all $t \in T(\ell)$. The assumption on $f$ implies that $f$ maps geodesics to geodesics, and so as in the proof of Lemma \ref{ToralRes} the fact that $A(S)$ contains the geodesics between any two such points $f(t\cdot x)$ implies that $f(A(T)) \subset A(S)$ that the restriction of $f$ to $A(T)$ is $X_*(T)_\Gamma \to X_*(S)_\Gamma$-equivariant.

        (ii) We again argue in the same way: Since $f$ is equivariant, the point $f(x)$ determines the images $f(t\cdot x)=t \cdot f(x)$ for all $t \in T(\ell)$ for some split torus $T \subset G$. Since $f$ is toral, it maps geodesics to geodesics, and the geodesics from $x$ to $t\cdot x$ for all such $t$ are dense in $I(G)$. So $f$ is uniquely determined by $f(x)$.
      \end{proof}

  \begin{proof}

    Using Lemma \ref{ToralRes}, the statement over $\breve m$ implies the statement over $m$ by restricting the function $f$. So we assume that $m$ is strictly Henselian.  We choose a special point $x \in I(G)$.

    In \cite[Theorem 5.1.2]{MR0491992}, Rousseau shows that $I(G_m)$ is functorial in $m$ via equivariant toral maps, and we fix such maps for all algebraic non-archimedean extensions of $m$. Consider some point $x' \in I(G)$. In any appartment $A(T)$ containing both $x$ and $x'$, the points $x$ and $x'$ differ by an element $\chi \in X_*(T)_\Gamma$. If $\chi$ lies in $X_*(T)_\mathbb{Q}$ for some $T$, then we say that $x$ and $x'$ are rational relative to each other. If this is the case then $\chi$ lies in $X_*(T)_{|n^*|}$ for some non-archimedean Galois extension $n$ of $\ell$, and since $T(n)$ acts on the appartment of some split maximal torus of $G_n$ containing $T$ by translation through the negative of the map $T(n) \to X_*(T)_{|n^*|}$ from \eqref{TIso},  the images of $x$ and $x'$ in $I(G_n)$ become conjugate under $T(n) \subset G(n)$.

  Using this, it follows that there exists a strictly Henselian discretely normed Galois extension $n$ of $m$ over which $G$ and $H$ become split and such that the image $x_n$ of $x$ in $I(G_n)$ is a hyperspecial point. Let $\CG_{x_n}$ be the parahoric integral model of $G_n$ associated to $x_n$. We consider the fixed-point set $F=I(H_n)^{\CG_{x_n}(n^\circ)}$ which is convex and by the Bruhat-Tits fixed-point theorem non-empty. Since $F$ is a union of facets of $I(H_n)$, it contains a point $x'$ which is rational relative to $x_n$. Since $x \in I(G)$, the natural action of $\Gal(n/\ell)$ on $I(H_n)$ maps $F$ to itself. The $\Gal(n/\ell)$-orbit of $x'$ is bounded and consists of elements which are rational relative to $x_n$. Hence the center $y$ of the convex hull of this orbit is rational relative to $x_n$ and fixed  by both $\CG(n^\circ)$ and $\Gal(n/\ell)$.  Let $\CH_y$ be the parahoric integral model of $H_n$ associated to $y$. By Theorem \ref{Etoffe}, the inclusion $h$ extends to a homomorphism $\CG_{x_n} \to \CH_{y}$ over $n^\circ$.

 Since $y$ is rational relative to $x_n$, there exists a further strictly Henselian discretely normed non-archimedean extension of $n$ such that the image $y_o$ of $y$ in $I(H_o)$ becomes hyperspecial. If $\CH_{y_o}$ is the associated hyperspecial integral model of $H_o$, by invoking Theorem \ref{Etoffe} again we obtain a homomorphism $h'\colon \CG_{x_o}=\CG_{x_n,o^\circ} \to \CH_{y,o^\circ} \to \CH_{y_o}$ of hyperspecial integral models over $o^\circ$. Then the commutative diagram 
    \begin{equation*}
      \xymatrix{
        N^\otimes(\CG_{x_n,o^\circ}) \ar[r] \ar[d]_\cong &  N^\otimes(\CH_{y_o}) \ar[d]^\cong \\
        I(G_o) \ar[r]_{f'} & I(H_o)
        }
      \end{equation*}
      in which the top horizontal map is induced by $h'$ and the vertical maps are the bijections from Theorem \ref{BTComp} associated to the points $x_o$ and $y_0$ respectively defines a $G(o)$-equivariant toral inclusion $f'\colon I(G_o) \to I(H_o)$.

      Now we fix a standard metric $d$ on $I(H_0)$. Such metrics are constructed in \cite[Subsection]{BT1} starting from a scalar product invariant under the affine Weyl group on some appartment. By taking this appartment to be one from $I(H_n)$, and the scalar product to be invariant under $\Gal(n/\ell)$, we can in addition ensure that the restriction of $d$ to $I(H_n)$ is $\Gal(n/\ell)$-equivariant. By restricting $d$ via $f'$ we obtain a standard metric on $I(G_o)$.

      The inclusion $I(H_m) \into I(H_o)$ has convex image since it is toral. Since $m$ is discretely valued, the building $I(H_m)$ is complete, c.f. \cite[2.5.12]{BT1}. Hence such a metric induces a projection $\pi\colon I(H_o) \to I(H_m)$ sending any point $y' \in I(H_o)$ to the closest point $\pi(y') \in I(H_m)$ to $y'$. Since $H(m)$ acts by isometries on both $I(H_o)$ and $I(H_m)$, the map $\pi$ is $H(m)$-equivariant. Hence we obtain a $G(m)$-equivariant map
      \begin{equation*}
        f=\pi\circ f'\colon I(G_m) \into I(G_o) \to I(H_o) \to I(H_m).
      \end{equation*}

      We will show that $f$ has the required properties. To this end, we first show that $f$ is isometric. Since $\pi$ is contractive, so is $f$. Let $v, w \in I(G_m)$. We choose a split maximal torus $T$ of $G_m$ for which $v,w \in A(T)$. The group $T(m)$ acts on $A(T)$ via translations through the surjective homomorphism $T(m) \onto X_*(T)\otimes |m^*|$ from \eqref{TIso}. Hence we can find elements $t \in T(m)$ for which the geodesic from $v$ to $t\cdot v$ in $A(T)$ gets arbitrarily close to $w$. Let $p$ be the point on this geodesic closest to $w$. Using the triangle inequality we find
      \begin{equation*}
        d(f(v),f(t\cdot v)) \leq d(f(v),f(w))+d(f(w),f(t\cdot v)) \leq d(v,w)+d(w,t\cdot v) \leq d(v,t\cdot v)+2d(w,p).
      \end{equation*}
      Since $f(t\cdot v)=t\cdot f(v)$, both $d(f(v),f(t\cdot v))$ and $d(v,t\cdot v)$ are equal to the length of the translation vector by which $t$ acts on $A(T)$. Since we can choose $t$ such that $d(w,p)$ becomes arbitrarily small, this implies $d(f(v),f(w))=d(v,w)$. So $f$ is isometric and hence toral by Lemma \ref{IsometricToral}.

      Finally we prove that $f$ is $\Gal(m/\ell)$-equivariant: Since $\Gal(n/\ell)$ acts by isometries on both $I(H_n)$ and $I(H_m)$, the restriction of $\pi$ to $I(H_n)$ is $\Gal(n/\ell)$-equivariant. Hence the point $f(x)=\pi(y)$ is $\Gal(m/\ell)$-invariant. Hence for $\sigma \in G(m/\ell)$, both $f$ and $\sigma\circ f \circ \sigma^{-1}$ are isometrical $G(m)$-equivariant maps $I(G_m) \to I(H_m)$ which map $x$ to the same point. By Lemma \ref{IsometricToral} this implies that these two maps are equal.
    \end{proof}

\subsection{Special Fibers of parahoric group schemes} \label{SpecFibSS}
We return to the assumptions of Subsection \ref{TannakaSubsec}. Using our methods we can also give the following description of the special fiber of a parahoric group scheme: (It is not clear to us whether this result is new.) Let $\theta \in N^\otimes(\CG)$ and $\chi\colon D^{\Gamma/|\ell^*|}_{\ell^\circ} \to \UAut^\otimes(\alpha_\theta)$ a homomorphism which splits $\alpha_\theta$. Then by Theorem \ref{MainTheorem}, the special fiber $\UAut^\otimes(\alpha_\theta)_{\tilde\ell}$ is a semidirect product of a split unipotent group scheme and the centralizer of $\chi_{\tilde\ell}$. The later is described the theorem below.

Since $\CG$ is a hyperspecial model, this shows that the centralizer of $\chi_{\tilde\ell}$ in $\UAut^\otimes(\alpha_\theta)_{\tilde\ell}$ is reductive, and hence that its identity component is a Levi subgroup of the identity component of $\UAut^\otimes(\alpha)_{\tilde\ell}$. Then $\Cent_{\UAut^\otimes(\alpha_\theta)}(\chi_{\tilde\ell})^\circ$ is a reductive lift of this Levi subgroup to the parahoric group scheme $\UAut^\otimes(\alpha_\theta)$. The existence of such lifts was previously shown by McNinch in \cite{MCNINCH2020}.

  \begin{theorem}
    Let $\theta \in N^\otimes(\CG)$ and $\chi\colon D^{\Gamma/|\ell^*|}_{\ell^\circ} \to \UAut^\otimes(\alpha_\theta)$ be a homomorphism which splits $\alpha_\theta$ over $\ell^\circ$.

    Let $T \in S(\CG)$ be such that $\theta \in A^\otimes(T)$ with split integral model $\CT \subset \CG$. Let $\chi' \colon D^\Gamma_{\ell^\circ} \to \CT$ the preimage of $\theta$ under the bijection \eqref{BTT} and $\bar \chi'$ be the composition $D^{\Gamma/|\ell^*|}_{\ell^\circ} \to D^\Gamma_{\ell^\circ} \toover{\chi'} \CT$. 

    Then the group schemes $\Cent_{\UAut^\otimes(\alpha_\theta)}(\chi)$ and $\Cent_{\CG}(\bar \chi')$ over $\ell^\circ$ are inner forms of each other in such a way that the conjugacy classes of $\chi$ and $\bar \chi'$ correspond to each other.
  \end{theorem}
  \begin{proof}
    By Proposition \ref{SplitExtEquiv}, the homomorphism $\chi'_{\ell}$ extends to a homomorphism $D^{\Gamma}_{\ell^\circ} \to \UAut^\otimes(\alpha_\theta)$ which splits $\alpha_\theta$ over $\ell^\circ$. Hence by Theorem \ref{MainTheorem} (viii), the homomorphisms $\chi_\ell$ and $\bar\chi'_\ell$ are conjugate by an element $g\in \UAut^\otimes(\alpha_\theta)(\ell^\circ)$. Since the statement is invariant under replacing $\chi$  by $\leftexp{g}{\chi}$, we may thus assume that $\chi_\ell= \bar\chi'_\ell$.

    Let now $\ell \subset m$ be a non-archimedean strictly Henselian extension of $\ell$ satisfying $|m^*|=\Gamma$. Then $\alpha_{\theta,m}=\alpha_\lambda$ for the integral model $\Rep^\circ \CG \to \Mod_{m^\circ},\; X \mapsto \omega_\CG(X)^{\theta_{X,m} \leq 1}$ of $\omega_G$ by Lemma \ref{SpecialFunctorClass}. As in the proof of Proposition \ref{StrongerSplitting}, Theorem \ref{FFIso} implies that there exists an element $g \in G(m)$ such that $g \cdot \lambda=\omega_{\CG,m^\circ}$ as an integral model of $\omega_G$.

    Since $\theta \in A^\otimes(T)$, the inclusion $T \into G$ extends (uniquely) to a closed immersion $\CT \into \UAut^\otimes(\alpha_\theta)$ over $\ell^\circ$. Since by Theorem \ref{MainTheorem} the kernel $U_{<1}(\alpha)$ of $b^{m/\ell}_{\tilde\ell}$ is unipotent, the homomorphism $b^{m/\ell}$ restricts to a fibrewise monomorphism $\CT_{m^\circ} \into \UAut^\otimes(\alpha_{\theta,m})$ over $m^\circ$. By \cite[IX.2.9]{SGA3II} this fibrewise monomorphism is a monomorphism and hence \cite[IX.2.5]{SGA3II} a closed immersion. So both $T_m$ and $\leftexp{g}{T_m}$ extend to a maximal split torus of $\CG_{m^\circ}$, c.f. Remark \ref{MaxSplitRmk}. Hence by Proposition \ref{ResSplitProp}, there exists an element $g'\in \CG(m^\circ)$ for which $T_m=\leftexp{g'g}{T_m}$. Hence after replacing $g$ by $g'g$ we find $g \in N_G(T)(m)$. Since $\CG$ is a hyperspecial integral model of $G$, the group $N_G(T)(m)$ is contained in $\CG(m^\circ) T(m)$. Hence after again replacing $g$ by some element of $\CG(m^\circ)g$, we may assume that $g \in T(m)$. 
    
    Let $\Phi\colon \UAut^\otimes(\alpha_{\theta,m}) \cong \CG_{m^\circ}$ be the isomorphism whose generic fiber is given by conjugation by $g$. Since by the above $\leftexp{g}{\chi}_\ell=\bar\chi'_\ell$, the isomorphism $\Phi$ sends $b^{m/\ell}\circ \chi_{m^\circ}$ to $\bar\chi'_{m^\circ}$. Hence, by Proposition \ref{SplBCComp2}, the homomorphisms $b^{m/\ell}$ and $\Phi$ induce isomorphisms
    \begin{equation*}
      \Cent_{\UAut^\otimes(\alpha_\theta)}(\chi)_{m^\circ} \cong \Cent_{\UAut^\otimes(\alpha_{\theta,m})}(b^{m/\ell} \circ \chi) \cong \Cent_\CG(\bar \chi')_{m^\circ}
    \end{equation*}
    which send $\chi_{m^\circ}$ to $\bar\chi'_{m^\circ}$. This proves the claim.
  \end{proof}
\bibliography{references}
\bibliographystyle{alpha}

\end{document}